\documentclass[reqno]{amsart}

\usepackage{amsmath}
\usepackage{tensor}
\usepackage{hyperref}
\hypersetup
{
	colorlinks,
	citecolor=blue,
	filecolor=black,
	linkcolor=blue,
	urlcolor=blue
}
\usepackage{amssymb}
\usepackage{amsthm}
\usepackage{mathrsfs}
\usepackage{color}

\usepackage
[
a4paper,% other options: a3paper, a5paper, etc
vmargin=2.9cm,
hmargin=2.4cm
]{geometry}
\usepackage{enumitem}
\usepackage{extarrows}
\usepackage{mathrsfs}
\usepackage[normalem]{ulem}

\newcommand\R{\mathbb R}
\newcommand\T{\mathbb T}
\newcommand\N{\mathbb N}
\newcommand\E{\mathbb E}
\newcommand\U{\mathbb U}
\newcommand\X{\mathcal X}
\newcommand\V{\mathcal V}
\newcommand\Y{\mathcal Y}
\newcommand\Z{\mathcal Z}
\newcommand\e{\varepsilon}
\newcommand\p{\mathbb P}
\newcommand\W{\mathcal W}
\newcommand\F{\mathcal F}
\newcommand\s{\mathcal S}
\newcommand\B{\mathcal B}
\newcommand\LL{\mathcal L}
\newcommand{\dd}{{\,\rm d}}
\newcommand{\Ls}{\Lambda^s}
\newcommand{\norm}[1]{\left\lVert#1\right\rVert}    
\newcommand\Pm[2][{\mathscr{P} }]{{\mathscr P}_{#2}}

\theoremstyle{plain}
\numberwithin{equation}{section}
\newtheorem{Theorem}{Theorem}[section]

\newtheorem{Lemma}{Lemma}[section]
\newtheorem{Definition}{Definition}[section]
\newtheorem{Assumption}{Assumption}

\theoremstyle{definition}

\newtheorem{Remark}{Remark}[section]

\newcommand{\abs}[1]{\lvert#1\rvert}

\begin{document}

	\title[Singular SDEs with applications]{A local-in-time theory for singular SDEs with applications to fluid models with transport noise}
	
	\author[D. Alonso-Or\'{a}n]{Diego Alonso-Or\'{a}n}
	\address{Institute for Applied Mathematics, University of Bonn,
		Endenicher Allee 60, D-53115 Bonn, Germany}
	\email{alonso@iam.uni-bonn.de}
	\author[C. Rohde]{Christian Rohde}
	%%    Address of record for the research reported here
	\address{Institut f\"{u}r Angewandte Analysis und Numerische Simulation, Universit\"{a}t Stuttgart, Pfaffenwaldring 57, 70569 Stuttgart, Germany}
	\email{christian.rohde@mathematik.uni-stuttgart.de}

	\author[H. Tang]{Hao Tang}
	%%    Address of record for the research reported here
	\address{Institut f\"{u}r Angewandte Analysis und Numerische Simulation, Universit\"{a}t Stuttgart, Pfaffenwaldring 57, 70569 Stuttgart, Germany}
	
	\email{hao.tang@mathematik.uni-stuttgart.de}
	\thanks{D.~Alonso-Or\'{a}n and H.~Tang are supported by the Alexander von Humboldt Foundation. C.~Rohde acknowledges  support by Deutsche Forschungsgemeinschaft (DFG,
		German Research Foundation) under Germany's Excellence Strategy - EXC 2075 - 39074001. }
	
	%    General info

	\date{\today}
	
	
	
	\begin{abstract}
		In this paper, we establish a local theory, i.e., existence, uniqueness and blow-up criterion, for a general family of singular SDEs in some Hilbert space.  The key requirement is an approximation property that allows us to  embed the singular  drift and diffusion mappings into a hierarchy of regular mappings that are invariant with respect to the Hilbert space and enjoy a cancellation property.\\  
		Various  nonlinear models in fluid dynamics with transport noise belong to this type of singular SDEs. 
		With a cancellation estimate for generalized Lie derivative operators, we  can construct such  regular approximations for cases involving the Lie derivative operators, or more generally, differential operators of order one with suitable coefficients.  In particular, we apply the abstract theory to achieve   novel  local-in-time results for the stochastic two-component Camassa--Holm (CH) system and  for the stochastic C\'ordoba-C\'ordoba-Fontelos (CCF) model. 
	\end{abstract}
	
	\maketitle
	
	%\section*{This is an unnumbered first-level section head}
	%This is an example of an unnumbered first-level heading.
	%
	%%% The correct journal style for \specialsection is all uppercase; a known bug
	%%% in amsart.cls prevents this, so input must be uppercase until it is fixed.
	%%\specialsection*{This is a Special Section Head}
	%\specialsection*{THIS IS A SPECIAL SECTION HEAD}
	%This is an example of a special section head%
	%%%%%%%%%%%%%%%%%%%%%%%%%%%%%%%%%%%%%%%%%%%%%%%%%%%%%%%%%%%%%%%%%%%%%%%%%
	%\footnote{Here is an example of a footnote. Notice that this footnote
	%text is running on so that it can stand as an example of how a footnote
	%with separate paragraphs should be written.
	%\par
	%And here is the beginning of the second paragraph.}%
	%%%%%%%%%%%%%%%%%%%%%%%%%%%%%%%%%%%%%%%%%%%%%%%%%%%%%%%%%%%%%%%%%%%%%%%%

	%	\tableofcontents
	%	

	\section{Introduction}	
	Let  $\X$, $\Y$ and $\Z$ be three separable Hilbert spaces such that  
	\begin{equation}\label{embed}
	\X\subset\Y \subset\Z. 
	\end{equation}
	We consider  the  initial value problem for a stochastic differential equation (SDE) with unknown process $X=X(t)$, $t\ge 0$, given by
	\begin{equation}
	\dd X=\big(b(t,X)+g(t,X)\big)\dd t+h(t,X)\dd\W,\ \ X(0) =X_0 \in \X. \label{Abstact irregular SDE}
	\end{equation}
	In \eqref{Abstact irregular SDE}, $\W$ denotes  a cylindrical Wiener process  defined on some  separable Hilbert space $\U$. The drift is given by  the sum of the mappings $b:[0,\infty)\times \X\rightarrow\X$ and 
	$g:[0,\infty)\times \X\rightarrow\Z$. The operator $ h:[0,\infty)\times \X\rightarrow \LL_2(\U;\Y)$
	stands for the diffusion coefficient  with $\LL_{2}(\U;\Y)$ being the space of  Hilbert-Schmidt operators from $\U$ to $\Y$.  We call \eqref{Abstact irregular SDE} a \emph{singular} initial value problem because $g$ and $h$ map $\X$  to the larger spaces $\Z$ and $\Y$,  i.e.~they are not invariant in $\X$.  We refer to  Sections \ref{notations definitions}, \ref{section:assumptions and results} for the precise setting. 
	
	For the entirely regular case  $\X=\Y =\Z$, it is well-known that  (local) Lipschitz conditions on $b(t,\cdot)+g(t,\cdot)$ and $h(t,\cdot)$ ensure 
	that  \eqref{Abstact irregular SDE} admits  unique (local) pathwise solution in $\X$. If additional
	monotonicity properties  on the coefficients are imposed, then the It\^{o} formula  for Gelfand-triple Hilbert spaces can be exploited to 
	assure global existence and continuity of solutions,  cf.~\cite{Kallianpur-Xiong-1995-book,Leha-Ritter-1985-MA,Krylov-Rozovskiui-1979-chapter,Prevot-Rockner-2007-book} 
	and the references therein. Notably this covers also the case when the Hilbert spaces form a  Gelfand triple. 
	
	In this work, we focus on the singular case which appears in particular for ideal fluid  models. 
	Indeed,   when we consider particular examples in Sobolev spaces $\X=H^s$, if $g(t,X)$ and $h(t,X)$ involve $\nabla X$ in a nonlinear way, or more generally, general Lie-type derivatives of $X$ (see our examples \eqref{SCH2 Ito form} and \eqref{Transport:nonloca:eq:ito} in the second part of the paper), then $g(t,X)$ and $h(t,X)$  can not be expected to be in $\X=H^s$, either. Likewise, the concept of  monotonicity fails to apply as it relies 
	on self-embedding drift and diffusion mappings.
	Working with the  abstract framework in \eqref{Abstact irregular SDE} 
	%
	%Being faced with the singular problem \eqref{Abstact irregular SDE} with initial data belonging to $\X$, finding a solution in the same space 
	entails another  difficulty as compared to the regular or the Gelfand-triple case: the It\^{o} formula is no longer available. 
	%
	%\begin{itemize}
	%	\item  Since $g(t,X)$ and $h(t,X)$ are  not invariant in $\X$,  they are in general not monotone\footnote{ only menttion that monotonicity requires self-embedding} in the sense of \cite{Pardoux-1972-CRAS,Prevot-Rockner-2007-book}; 
	%	
	%	\item The well-know It\^{o} formula in Hilbert space or under the Gelfand triple
	%	is not available; 
	%	
	%	\item More conditions are needed to obtain the time continuity of the solution. 
	%\end{itemize}
	To  highlight the latter difficulty, let us recall the classical It\^{o} formula for a  Gelfand triplet $V\hookrightarrow H\hookrightarrow V^*$, where $H$ is a separable Hilbert space with inner product $(\cdot,\cdot)$ and $H^*$ is its dual; $V$ is a Banach space such that the embedding $V\hookrightarrow H$ is dense. % $H$ is identified as $H^*$ by the Riesz isomorphism.  
	Then the following result is classical, see {\cite[Theorem I.3.1]{Krylov-Rozovskiui-1979-chapter} or \cite[Theorem 4.2.5]{Prevot-Rockner-2007-book}}. 
	
	\begin{center}
		\begin{minipage}{15cm}	
			\textit{\label{Classical Ito}
				Assume that $\mathbf{U}$ is a continuous $V^{*}$-valued stochastic process given by 
				$$\mathbf{U}(t)=\mathbf{U}(0)+\int_{0}^{t} \mathbf{g}(s)\, \mathrm{d} s+\int_{0}^{t} \mathbf{G}(s)\, \mathrm{d} \W(s), \quad t \in[0, T],$$
				where $\mathbf{G}\in L^{2}\left(\Omega \times[0, T] ; L_{2}(\U ; H)\right)$ and $\mathbf{g} \in L^{2}\left(\Omega \times[0, T] ; V^{*}\right)$ are both progressively measurable and $\mathbf{U}(0)\in L^2(\Omega;H)$  is $\F_0$-measurable. If $\mathbf{U}\in L^{2}\left(\Omega \times[0, T] ; V\right)$, then $\mathbf{U}$ is an $H$-valued continuous stochastic process and the  It\^{o} formula 
				\begin{align}
				\|\mathbf{U}(t)\|_{H}^{2}=\,&\|\mathbf{U}(0)\|_{H}^{2}+2 \int_{0}^{t}
				\tensor[_{V^*}]{\left\langle \mathbf{g}(s), \mathbf{U}(s)\right\rangle}{_{V}} \dd s
				\notag\\
				&+2 \int_{0}^{t} \left(\mathbf{G}(s) \,  \dd \W,\mathbf{U}(s)\right)_H +\int_{0}^{t}\|\mathbf{G}(s)\|_{\LL_{2}(\U; H)}^{2} \, \mathrm{d} s.\label{Classical Ito formula}
				\end{align}
				holds true $\p-a.s.$ for  all $t \in[0, T]$.	
			}	
		\end{minipage}
	\end{center}

	Notice that \eqref{Classical Ito formula} is applicable for $\mathbf{U}(0)\in H$, $\mathbf{U}\in L^{2}\left(\Omega \times[0, T] ; V\right)$, $\mathbf{g} \in L^{2}\left(\Omega \times[0, T] ; V^{*}\right)$ and $\mathbf{G}\in L^{2}\left(\Omega \times[0, T] ; L_{2}(\U ; H)\right)$.  
	However,  if $\mathbf{G}$ is singular (not invariant in $H$),  then  $\mathbf{G}\in L_{2}(\U ; H)$ is ambiguous. 
	Besides,  even though $\mathbf{g}$ is allowed to be less regular, \eqref{Classical Ito formula} requires $\mathbf{U}(t)$ to be more regular than $\mathbf{U}(0)$, i.e., $\mathbf{U}\in V\hookrightarrow H\ni \mathbf{U}(0)$. In many cases (for example,  stochastic ideal fluid models),  we do
	\textit{not} know that  this holds true. Hence, \eqref{Classical Ito formula}  is \textit{not} applicable in singular cases, and then the time continuity of the solution \textit{cannot} be obtained directly.\\

	The first major goal  of this paper is to establish a local-in-time theory for  \eqref{Abstact irregular SDE} generalizing classical results 
	for e.g.~the completely regular case  $\X=\Y =\Z$. The second goal of this work is to show that the abstract theory for \eqref{Abstact irregular SDE}
	can be used to establish new results for ideal fluid systems with noise.

	\begin{enumerate}
		\item 	To achieve the first goal we    fix in Section \ref{section:assumptions and results} the precise assumptions  on the regular drift $b$ and in particular on the 
		singular drift $g$ and diffusion $h$  (see Assumption \ref{Assum-A}).
		Then we  provide our  main results  for \eqref{Abstact irregular SDE}, including the existence, uniqueness, time regularity, and a result   characterizing the possible blow-up  of  pathwise solutions   (see Theorem \ref{Abstract irregular SDE: T}).  The key requirements 
		for the proof are the assumption on the existence of appropriate  Lipschitz-continuous and   monotone regularizations for the singular
		mappings. This allows us to exploit It\^{o}-like formulas as above.

		\item With the abstract framework at hand, for a large number of nonlinear SPDE 
		models,
		we are able to construct such  regular approximation schemes by using convolution operators and establishing a cancellation property for generalized Lie derivatives (cf. Lemma \ref{Liecancellations}).
		 To set the stage, in Section \ref{sec:application},  we   consider  two   models governing ideal flows  with 
		particularly interesting stochastic perturbation, namely 
		\begin{itemize}
			\item the  two-component Camassa-Holm (CH) system with transport noise \cite{Holm-Erwin-2019-Arxiv}, see \eqref{SCH2 transport noise} below,
			\item a nonlinear transport equation with non-local velocity,  referred to the C\'ordoba-C\'ordoba-Fontelos (CCF) model \cite{Cordoba-etal-2005-Annals}, with transport noise, see \eqref{Transport:nonloca:eq} below.
		\end{itemize}
		In both cases, we obtain a local-tin-time theory in the sense of  the abstract-framework Theorem \ref{Abstract irregular SDE: T}. 
		These results for both  models are new up to our knowledge and  can be found in Section \ref{sect:applications:main results}.
		Finally, we will explain in Section \ref{Sect:further examples}  how our abstract framework and  the regular approximation schemes
		can be applied to a broader class of fluid dynamics equations including the surface quasi-geostrophic (SQG) equation with transport noise. 
		
		%	\item In Appendix \ref{sec:appendix} we present some results which are needed throughout the  work. In particular, .
	\end{enumerate}

	\section{An abstract framework for a class of singular SDEs}\label{sect: Abstact irregular SDE} 
	
	\subsection{Notations and definitions}\label{notations definitions}
	To begin with, we introduce some notations.  We consider a probability space $(\Omega, \mathcal{F},\p)$,  where $\p$ is a probability measure on $\Omega$ and $\mathcal{F}$ is a $\sigma$-algebra. We endow the probability space $(\Omega, \mathcal{F},\p)$ with an increasing filtration $\{\mathcal{F}_t\}_{t\geq0}$, which is a right-continuous filtration on $(\Omega, \mathcal{F})$ such that $\{\mathcal{F}_0\}$ contains all the $\p$-negligible subsets. 
	For some separable Hilbert space $\U$  with a complete orthonormal basis $\{e_i\}_{i\in\N}$ the noise $\W$ in \eqref{Abstact irregular SDE} is  a cylindrical Wiener process, i.e., it is defined by
	\begin{equation}\label{define W}
	\W=\sum_{k=1}^\infty W_ke_k\ \ \p-a.s,
	\end{equation}
	where $\{W_k\}_{k\geq1}$ is a sequence of mutually independent standard 1-D Brownian motions. To guarantee the convergence of the above formal summation, we consider a   larger separable Hilbert space  $\U_0$ such that the canonical injection $\U\hookrightarrow \U_0$ is Hilbert--Schmidt.  
	Therefore, for any $T>0$, we have, cf. \cite{Prato-Zabczyk-2014-Cambridge,Gawarecki-Mandrekar-2010-Springer,Karczewska-1998-AUMCSS},
	$$\W\in C([0,T], \U_0)\ \ \p-a.s.$$  Note that  the choice of the  auxiliary Hilbert spaces $\U$ and $\U_0$  is not crucial for our analysis.  Thus we let  $\U$ and $\U_0$ be arbitrary but fixed in the sequel.\\ 
	For some time  $t>0$, the family  $\sigma\{x_1(\tau),\cdots,x_n(\tau)\}_{\tau\in[0,t]}$ stands for the completion of
	the union $\sigma$-algebra generated by $\left(x_1(\tau),\cdots,x_n(\tau)\right)$ for $\tau\in[0,t]$.
	$\E Y$ stands for the mathematical expectation
	of a random variable $Y$ with respect to $\p$. From now on $\s=(\Omega, \mathcal{F},\p,\{\mathcal{F}_t\}_{t\geq0}, \W)$ is called a stochastic basis.\\
	For any  Hilbert space  $\mathbb{H}$ the inner product is denoted by $(\cdot,\cdot)_{\mathbb{H}}$. Furthermore, the space 
	$\LL_2(\U; \mathbb{H})$  contains  all  Hilbert-Schmidt operators $Z:\U\rightarrow \mathbb{H}$  with finte  norm
	$\|Z\|^2_{\LL_2(\U; \mathbb{H})}=\sum_{k=1}^\infty\|Ze_k\|^2_{\mathbb{H}}.$
	As in \cite[Theorem 2.3.1]{Breit-Feireisl-Hofmanova-2018-Book}, we see that for an $\mathbb{H}$-valued progressively measurable stochastic process $Z$ with  $Z\in L^2\left(\Omega;L^2_{\rm loc}\left([0,\infty);\LL_2(\U; \mathbb{H})\right)\right)$,  one can define the It\^{o} stochastic integral
	\begin{equation*}
	\int_0^t Z \,\dd\W=\sum_{k=1}^\infty\int_0^t Z e_k\, \dd W_k.
	\end{equation*}
	Most notably for the analysis here, if $Z\in\LL_2(\U;\mathbb{H})$ and $\W$ is given as above, we have the Burkholder-Davis-Gundy (BDG) inequality 
	\begin{align*}
	\E\left(\sup_{t\in[0,T]}\left\|\int_0^t Z\,\dd\W\right\|_{\mathbb{H}}^p\right)
	\leq C \E\left(\int_0^T \|Z\|^2_{\LL_2(\U;\mathbb{H})}\,\dd t\right)^\frac{p}{2},\ \ p\geq1,
	\end{align*}
	or in terms of the coefficients,
	\begin{align*}
	\E\left(\sup_{t\in[0,T]}\left\|\sum_{k=1}^\infty\int_0^t Ze_k\,\dd W_k\right\|_{\mathbb{H}}^p\right)
	\leq C \E\left(\int_0^T \sum_{k=1}^\infty\|Ze_k\|^2_{\mathbb{H}}\,\dd t\right)^\frac{p}{2},\ \ p\geq1.
	\end{align*}

	Let $\mathbb X$ be a separable Banach space. $\B(\mathbb X)$ denotes the Borel sets of $\mathbb X$ and  $\Pm{}(\mathbb X)$ stands for the collection of Borel probability measures on $(\mathbb X,\B(\mathbb X))$. We denote $\Pm{r}(\mathbb X)$ the family of probability measures in $\Pm{}(\mathbb X)$  with finite moment of order $r\in[1,\infty)$, i.e.,
	$\Pm{r}(\mathbb X)=\left\{\mu:\int_{\mathbb X}\|x\|_{\mathbb X}^r \,\mu({\rm d}x)<\infty\right\}.$  For two Banach spaces $\mathbb X$ and $\mathbb Y$, $\mathbb X\hookrightarrow\mathbb Y$ means that $\mathbb X$ is  embedded continuously into $\mathbb Y$, and $\mathbb X\hookrightarrow\hookrightarrow\mathbb Y$ means that the embedding is compact.  
	
	For some set $E$, $\textbf{1}_{E}$ denotes  the indicator function on $E$.\\%, i.e., it is equal to one when $x\in   E$, and zero otherwise.

	Next, let us make precise two different notions of solutions in the Hilbert space $\X$ from \eqref{embed}  for  the Cauchy problem \eqref{Abstact irregular SDE}. 
	\begin{Definition}[Martingale solutions]\label{Martingale solution definition}
		Let $\mu_0\in \Pm{}(\X)$. A triple $(\s, X,\tau)$ is said to be a martingale solution to \eqref{Abstact irregular SDE} if
		\begin{enumerate}
			\item $\s=\left(\Omega, \mathcal{F},\p,\{\mathcal{F}_t\}_{t\geq0}, \W\right)$ is a stochastic basis and $\tau$ is a stopping time with respect to $\{\mathcal{F}_t\}_{t\ge0}$;
			\item $X(\cdot\wedge \tau):\Omega\times[0,\infty)\rightarrow \X$ is an 
			$\mathcal{F}_t$-progressively measurable process such that it is continuous in $\Z$, $\mu_0(\cdot)=\p\{X_0\in \cdot\}$ for all $\cdot\in\B(\X)$ and
			for every $t>0$,
			\begin{align}\label{define solution}
			X(t\wedge \tau)-&X_0
			=\int_0^{t\wedge \tau}
			\big(   b(t', X(t'))+ g(t',X(t')) \big) \, \dd t'+\int_0^{t\wedge \tau}h(t',X(t') )    \, \dd\W \ \ \p-a.s.
			\end{align}
			In \eqref{define solution}, 	$\int_0^{\cdot}
			\left\{b(t', X(t'))+ g(t',X(t'))\right\}\dd t'$ is the Bochner integral on $\Z$ and 
			$\int_0^{\cdot}h(t',X(t'))\,\dd\W$ is a continuous local martingale on $\Y$.

			\item If $\tau=\infty$ $\p-a.s.$, then we say that the martingale solution is global.
		\end{enumerate}

	\end{Definition}
	
	The stronger concept of pathwise solutions is provided in
	
	\begin{Definition}[Pathwise solutions]\label{pathwise solution definition}
		Let $\s=(\Omega, \mathcal{F},\p,\{\mathcal{F}_t\}_{t\geq0}, \W)$ be a fixed stochastic basis. Let $X_0$ be an $\X$-valued $\mathcal{F}_0$-measurable random variable.  A local pathwise solution to \eqref{Abstact irregular SDE} is a pair $(X,\tau)$, where $\tau$ is a stopping time satisfying $\p\{\tau>0\}=1$ and
		$X:\Omega\times[0,\tau]\rightarrow \X$  is an $\mathcal{F}_t$-progressively measurable process satisfying \eqref{define solution} and $X(\cdot\wedge \tau)\in C\left([0,\infty);\X\right)$ almost surely.
		
	\end{Definition}
	
	It follows from Definition \ref{Martingale solution definition} that,
	if a martingale solution exists, then \eqref{define solution} implies that  
	\begin{equation}\label{RHS definition}
	\int_0^{t\wedge \tau}\big(b(t', X(t'))+ g(t',X(t'))\big)\dd t'+\int_0^{t\wedge \tau}h(t',X(t')) \dd\W
	\end{equation}
	takes values in $\X$, even though $g$ and $h$ are not invariant in $\X$. Moreover, Definition \ref{pathwise solution definition} implies that if a pathwise solution exists, then \eqref{RHS definition} is continuous in time in $\X$. 
	
	To study the possible blow-up  of the solutions, we need the following concept of maximal solutions.
	
	\begin{Definition}[Maximal solutions]\label{maximal solution definition}
		Let $\s=(\Omega, \mathcal{F},\p,\{\mathcal{F}_t\}_{t\geq0}, \W)$ be a fixed stochastic basis. Let $X_0$ be an $\X$-valued $\mathcal{F}_0$-measurable random variable.  $(X,\tau^*)$ is called a maximal pathwise solution to \eqref{Abstact irregular SDE}   %if $\tau^*>0$ 	almost surely  and
		if there is an increasing sequence $\tau_n\rightarrow\tau^*$ such that for any $n\in\N$, $(X,\tau_n)$ is a pathwise solution satisfying 
		\begin{equation*}
		\sup_{t\in[0,\tau_n]}\|X\|_{\X}\geq n\  \ \ a.e.\   on\   \{\tau^*<\infty\}.
		\end{equation*}
		Particularly, if $\tau^*=\infty$ almost surely, then such a solution is called global.

	\end{Definition}
	
	%\subsection{SDE with both singular drift and diffusion}
	
	\subsection{Assumptions and main results}\label{section:assumptions and results}

	To study the existence of  martingale and pathwise solutions, we  need the following assumptions on  the three separable Hilbert spaces  $\X$, $\Y$, $\Z$
	from   \eqref{embed} and on the 
	coefficients $b$, $g$ and $h$ in \eqref{Abstact irregular SDE}. Recall that $\{e_i\}_{i\in\N}$ is a complete orthonormal basis of $\U$.
	
	\begin{Assumption}\label{Assum-A}
		The Hilbert spaces satisfy the embedding relation  $\X\hookrightarrow \Y \hookrightarrow\hookrightarrow\Z$ and  the coefficients 
		$b:[0,\infty)\times \X\rightarrow\X$, $g:[0,\infty)\times \X\rightarrow\Z$ and $h:[0,\infty)\times \X\rightarrow \LL_2(\U;\Y)$ are continuous in both variables. Let  $\V$ be  a Banach space satisfying $\Z\hookrightarrow\V$.\\
		There are non-decreasing  locally bounded functions  
		$f(\cdot),k(\cdot),q(\cdot) \in C\left([0,+\infty);[0,+\infty)\right)$ such that   the following conditions hold true.
		
		\begin{enumerate}[label={ $(\textbf{A}_\arabic*)$}]
			\item\label{b in X} 
			For all $(t,X)\in [0,\infty)\times \X$, we have 
			\begin{align}\label{A11}
			\|b(t,X)\|_{\X}\leq k(t)f\big(\|X\|_{\V}\big)\|X\|_{\X},
			\end{align}
			and for all $N\in\N$,
			\begin{align}\label{A12}
			\sup_{\|X\|_\X,\|Y\|_\X\le N}\left\{{\bf 1}_{\{X\ne Y\}}  \frac{\|b(t,X)-b(t,Y)\|_{\X}}{\|X-Y\|_{\X}}\right\} \le q(N)k(t).
			\end{align}
			Besides, for any bounded sequence $\{X_\e\}\subset\X$ such that $X_\e\rightarrow X$ in $\Z$,
			\begin{align}\label{A13}
			\lim_{n\rightarrow\infty} \|b(t,X_\e)-b(t,X)\|_{\Z}=0, \ t\geq0.
			\end{align}
			
			\item \label{gn hn condition} 
			For $\e\in(0,1)$ and $N\ge 1$ there exist  progressively measurable maps 
			$$g_\e: [0,\infty)\times\X \rightarrow\X,\ \ h_\e: [0,\infty)\times\X \rightarrow\LL_2(\U;\X)$$ 
			and constants $C_{\e,N}>0$ 
			such that for all $t\ge0$ the bounds 
			\begin{align}\label{A21}
			&\sup_{\e\in(0,1),\|X\|_\X\le N}
			\left\{\|g_\e(t,X)\|_{\Z}+\|g(t,X)\|_{\Z}+\|h_\e(t,X)\|_{\LL_2(\U;\Y)}+\|h(t,X)\|_{\LL_2(\U;\Y)}\right\}\le q(N)k(t),
			\end{align} 
			\begin{align}\label{A22}
			&\sup_{\|X\|_\X\le N} 
			\left\{\|g_\e(t,X)\|_\X +\|h_\e(t,X)\|_{\LL_2(\U;\X)} \right\} \le C_{\e,N}k(t),
			\end{align} 
			and
			\begin{align}\label{A23}
			&\sup_{\|X\|_\X,\|Y\|_\X\le N}\left\{{\bf 1}_{\{X\ne Y\}} \left(\frac{\|g_\e(t,X)-g_\e(t,Y)\|_\X}{\|X-Y\|_{\X}}
			+\frac{\|h_\e(t,X)-h_\e(t,Y)\|_{\LL_2(\U;\X)}}{\|X-Y\|_{\X}}\right)\right\} \le C_{\e,N}k(t)
			\end{align} 
			hold. Moreover, for any bounded sequence $\{X_\e\}\subset\X$ such that $X_\e\rightarrow X$ in $\Z$ and  for any $t>0$, we have 
			\begin{align}\label{A24}
			\lim_{\e\rightarrow0} \int_0^t\left|
			\tensor[_\Z]{\left\langle g_\e(t,X_\e(t'))-g(t,X(t')),\phi\right\rangle}{_{\Z^*}}
			\right|\,\dd t'=0\ \ \forall \,\phi\in \Z^*
			\end{align} 
			and
			\begin{align}\label{A25}
			\lim_{n\rightarrow\infty} \|h_\e(t,X_\e)-h(t,X)\|_{\LL_2(\U;\Z)}=0.
			\end{align} 
			Here  $	_\Z\langle\cdot,\cdot\rangle{_{\Z^*}}$ denotes the dual pairing in $\Z$.
			
			\item\label{growth gn hn}
			Let $g_\e$ and $h_\e$  as in \ref{gn hn condition}. For all $n\ge1$ and $(t,X)\in [0,\infty)\times \X$, we have 
			\begin{align}\label{A31}
			\sum_{i=1}^{\infty}\left|\left(h_\e(t,X)e_i,X\right)_{\X}\right|^2\leq k(t)f \big( \|X\|_{\V}\big)\|X\|^4_{\X}
			\end{align}
			and
			\begin{align}\label{A32}
			2\left(g_\e(t,X),X\right)_{\X}+\|h_\e(t,X)\|^2_{\LL_2(\U;\X)}\leq k(t)f(\|X\|_{\V}\big)\|X\|^2_{\X}.
			\end{align}

			\item\label{uniqueness requirement}
			For any $t\ge0$ and $N\geq1$, we have 
			\begin{align}\label{B11}
			\sup_{\|X\|_\X,\|Y\|_\X\le N}\left\{{\bf 1}_{\{X\ne Y\}} \frac{\|b(t,X)-b(t,Y)\|_\Z}{\|X-Y\|_{\Z}}	+\right\} \le q(N)k(t)
			\end{align} 
			and
			\begin{align}\label{B12}
			\sup_{\|X\|_\X,\|Y\|_\X\le N}\left\{{\bf 1}_{\{X\ne Y\}}  \frac{ 2\left(g(t,X)-g(t,Y),X-Y\right)_{\Z}+\|h(t,X)-h(t,Y)\|^2_{\LL_2(\U;\Z)}}{\|X-Y\|^2_{\Z}}\right\} \le q(N)k(t).
			\end{align}
			
			\item\label{continuity requirement}
			The embedding $\X\hookrightarrow\Z$ is dense, and
			there is a family of continuous linear operators $\{T_\e : \Z\to \X\}_{\e\in(0,1)}$  such that 
			\begin{equation}\label{B21}
			\|T_\e X\|_\X\le \|X\|_\X,\ \ \lim_{n\rightarrow\infty} \|T_\e  X-X\|_\X=0 \qquad  ( X\in \X)
			\end{equation} 
			and for all   $ t\geq0,\ N\ge1$ 
			\begin{equation}\label{B23}
			\sup_{\e\in(0,1),\|X\|_\X\le N} 2\left(T_\e  g(t,X), T_\e X\right)_{\X}
			+\|T_\e h(t,X)\|^2_{\LL_2(\U;\X)}\leq q(N)k(t),
			\end{equation}
			\begin{align}\label{B24}
			\sup_{\e\in(0,1),\|X\|_\X\le N} 
			\sum_{i=1}^{\infty}\left|\left(T_\e h(t,X)e_i,T_\e X\right)_{\X}\right|^2\leq q(N)k(t)
			\end{align}
			hold.
			\item	\label{blow-up requirement}
			There is a family  of continuous linear operators $\{Q_\e:\Z\to \X  \}_{\e \in(0,1)} $  such that \eqref{B21}  with $Q_\e $ replacing $T_\e $ 
			and

			\begin{align}\label{blow-up requirement 2}
			\sup_{\e\in(0,1)} 
			\sum_{i=1}^{\infty}\left|\left(Q_\e  h(t,X)e_i,Q_\e  X\right)_{\X}\right|^2\leq k(t)f\big(\|X\|_{\V}\big)\|X\|^2_{\X}
			\|Q_\e  X\|^2_{\X},
			\end{align}
			\begin{equation}\label{blow-up requirement 1}
			\sup_{\e\in(0,1)} 2\left(Q_\e  g(t,X), Q_\e  X\right)_{\X}
			+\|Q_\e  h(t,X)\|^2_{\LL_2(\U;\X)}\leq k(t)f \big( \|X\|_{\V} \big)\|X\|^2_{\X}
			\end{equation}
			hold true for $t\ge 0$.

		\end{enumerate}
	\end{Assumption}

	Note that for the singular mappings the constants $C_{\e,N}$ in Assumption \ref{Assum-A}  are  non-decreasing in $N$ for $\e$ fixed and explode for $\e\to 0$ with $N$ fixed.  
	Then we can state our main results for  the initial value problem \eqref{Abstact irregular SDE}. 
	
	\begin{Theorem}\label{Abstract irregular SDE: T}  Considering \eqref{Abstact irregular SDE}, we have the following results.
		
		\begin{enumerate} [label={\bf (\roman*)}]
			\item \label{Abstract irregular SDE: M}	Let Assumptions \ref{b in X}-\ref{growth gn hn} hold. Then, for any $\mu_0\in \Pm{2}(\X)$, \eqref{Abstact irregular SDE} has a local martingale solution $(\s,X,\tau)$ in the sense of Definition \ref{Martingale solution definition}.  
			
			\item\label{Abstract irregular SDE: P}  Let $\s=(\Omega, \mathcal{F},\p,\{\mathcal{F}_t\}_{t\geq0}, \W)$ be a fixed stochastic basis. If Assumptions \ref{b in X}-\ref{continuity requirement} hold, then for any $\mathcal{F}_0$-measurable random variable $X_0\in L^2(\Omega;\X)$,  \eqref{Abstact irregular SDE} has a local unique pathwise solution $(X,\tau)$, in the sense of Definition \ref{pathwise solution definition} such that
			\begin{equation}\label{L2 moment bound}
			X(\cdot\wedge \tau)\in L^2\left(\Omega; C\left([0,\infty);\X\right)\right).
			\end{equation} 
			
			\item\label{Abstract irregular SDE: P:time and blow-up} Let $(X,\tau^*)$ be the maximal solution to \eqref{Abstact irregular SDE}, in the sense of Definition \ref{maximal solution definition}, under \ref{b in X}-\ref{continuity requirement}.    If additionally  \ref{blow-up requirement}  holds true, then the  blow-up occurs in $\X$ as well as $\V$, i.e.
			\begin{equation}\label{Blow-up criterion}
			\textbf{1}_{\left\{\limsup_{t\rightarrow \tau^*}\|X\|_{\X}=\infty\right\}}=\textbf{1}_{\left\{\limsup_{t\rightarrow \tau^*}\|X\|_{\V}=\infty\right\}}\ \ \p-a.s.
			\end{equation}
		\end{enumerate}

	\end{Theorem}
	
	\begin{Remark}\label{Remark Assum}
		
		We first remark that the singular terms $g$ and $h$ are in general not monotone in the sense of  \cite{Pardoux-1972-CRAS,Prevot-Rockner-2007-book}. So,  the well-known  approximation scheme under a Gelfand triple developed for quasi-linear  SPDEs does not work for the present model.
		Motivated by \cite{Ren-Tang-Wang-2020-Arxiv}, 
		we will  employ a regularization argument to overcome this difficulty.
		Let us give some  explanations on  Assumption \ref{Assum-A} that makes  precise the required regularization procedure.
		
		\begin{itemize}
			\item The condition \ref{b in X} provides the local Lipschitz continuity  
			for the regular drift coefficient $b(t,X)$ and bounds  its growth.
			Assumption \ref{gn hn condition}  requires the local Lipschitz continuity  on the 
			approximations $g_\e$ and $h_\e$ of the singular terms $g$ and $h$, which together with \ref{b in X}  will ensure local-in-time existence for some approximate problem.  In Section \ref{Application to SALT} we will show how to construct 
			such approximations using mollifiers.   
			
			\item \ref{growth gn hn} can be viewed as a renormalization type condition in the following sense. 
			Formally speaking,  even though $g$ and $h$ are not  invariant with respect to $\X$  (hence $(g(t,X),X)_{\X}$ and $\|h(t,X)\|_{\LL_2(\U;\X)}$ may be infinite), we require that $(g(t,X),X)_{\X}+\|h(t,X)\|_{\LL_2(\U;\X)}$ can be controlled. In fact, \ref{growth gn hn}  specifies this relationship for $g_\e$ and $h_\e$ such that $(g_\e(t,X),X)_{\X}$ and $\|h_\e(t,X)\|_{\LL_2(\U;\X)}$ make sense.

			\item  Since $g$ and $h$ are singular, we need \ref{uniqueness requirement} on the joint space $\Z$  to guarantee pathwise uniqueness.
			
			\item As explained in the introduction, we  can \textit{not} use the It\^{o} formula \eqref{Classical Ito formula} to obtain the time continuity of the solution directly. This is why we need to assume \ref{continuity requirement} and \ref{blow-up requirement} to establish time continuity and blow-up criterion, respectively.    \ref{blow-up requirement} is stronger than \ref{continuity requirement} because we need both, the
			validity of the  It\^{o} formula and the growth condition. However,  the dense embedding $\X\hookrightarrow\Z$ is not necessary for deriving the  blow-up criterion. Moreover, in applications, usually one can take $T_\e =Q_\e $.

			\item In view of Assumption  \ref{Assum-A}, it is worthwhile noticing that the regular drift $b$ will \textit{not} be used to
			control the singular terms, i.e. our result  covers the case $b\equiv 0$, where both the drift and diffusion in \eqref{Abstact irregular SDE} are singular.  However, we assume that the problem \eqref{Abstact irregular SDE} has a regular part to cover more  ideal fluid models. %\sout{We also %remark that, 
			%if a solution exists, then \eqref{define solution} implies that  $\int_0^{t\wedge \tau}\left(b(t', X(t'))+ g(t',X(t'))\right)\dd t'+\int_0^{t\wedge %\tau}h(t',X(t'))\dd\W$ takes values in $\X$, even though $g$ and $h$ are not invariant in $\X$}. 
		\end{itemize}

	\end{Remark}
	
	\begin{Remark}%\label{martingale initial remark}
		We remark that in  \cite{Debussche-Glatt-Temam-2011-PhyD}, an abstract fluid model involving a Stokes operator (viscous term) and a regular noise coefficient is studied. Here, we are able to  deal with ideal fluid models without viscosity and  noise  of transport type. Moreover, in  \cite{Debussche-Glatt-Temam-2011-PhyD}, the martingale solution exists under the condition that the initial measure has finite moment of order $r>8$ (See \cite[Theorem 6.1]{Debussche-Glatt-Temam-2011-PhyD}). In the present work,  we only require $r=2$, i.e., $\mu_0\in\Pm{2}(\X)$  in \ref{Abstract irregular SDE: M} in Theorem \ref{Abstract irregular SDE: T}.
	\end{Remark}

	\begin{Remark}
		We also remark that when the noise coefficient $h(t,X)$ is as regular as the solution $X$ and the singularity of \eqref{Abstact irregular SDE} only arises in $g$, namely 
		$b:[0,\infty)\times \X\rightarrow\X$, $h:[0,\infty)\times \X\rightarrow \LL_2(\U;\X)$ and $g:[0,\infty)\times \X\rightarrow\Z$,  one can also obtain a local theory as in Theorem \ref{Abstract irregular SDE: T} even under weaker conditions as in Assumption \ref{Assum-A}.
	\end{Remark}

	\subsection{Proof of \ref{Abstract irregular SDE: M} in Theorem \ref{Abstract irregular SDE: T}}\label{proof of Thm-M}
	
	For the sake of clarity, we split the proof into the following subsections.
	
	\subsubsection{Approximation scheme and uniform estimates}
	
	For $\mu_0\in\Pm{2}(\X)$, we first fix a stochastic basis $\s$ and a random variable $X_0$ such that the distribution law of $X_0$ is $\mu_0$.
	For any $R>1$, we let $\chi_R(x):[0,\infty)\rightarrow[0,1]$ be a $C^{\infty}$-function such that $\chi_R(x)=1$ for $x\in[0,R]$ and $\chi_R(x)=0$ for $x>2R$. Then we consider a cut-off version of \eqref{Abstact irregular SDE} given by 
	\begin{equation} \label{cut-off problem}
	%\left\{
	\begin{aligned}
	{\rm d}X&=\chi^2_R\big(\|X\|_{\V}\big)\left[b(t,X)+g(t,X)\right]{\rm d}t+\chi_R\big(\|X\|_{\V}\big)h(t,X)\,{\rm d}\W,\\ 
	X(0)&=X_0.
	\end{aligned}
	% \right.
	\end{equation}
	We have not posed any structural properties like monotonicity on the singular mappings $g,h$  that ensure the existence of solutions for \eqref{cut-off problem}. Therefore we employ the 
	regular approximations $g_\e$ and $h_\e$ from Assumption \ref{Assum-A} which leads us to the regular  approximative version 
	\begin{equation} \label{approximation cut problem}
	%\left\{
	\begin{aligned}
	{\rm d}X&=H_{1,\e}(t,X)\,{\rm d}t+H_{2,\e}(t,X)\,{\rm d}\W,\\ 
	H_{1,\e}(t,X)&=\chi^2_R\big(\|X\|_{\V}\big)\left(b(t,X)+g_\e(t,X)\right),\\
	H_{2,\e}(t,X)&=\chi_R\big(\|X\|_{\V}\big)h_\e(t,X),\\
	X(0)&=X_0.
	\end{aligned} %\right.
	\end{equation}
	For \eqref{approximation cut problem} we can obtain the following global existence result. 
	\begin{Lemma}\label{Tightness}
		For $\mu_0\in\Pm{2}(\X)$, we fix a stochastic basis $\s$ and a $\mathcal{F}_0$-measurable random variable $X_0$ such that the distribution of $X_0$ is $\mu_0$.  
		Let $R>1$ be fixed.\\
		For each $\e\in(0,1)$, the problem \eqref{approximation cut problem} has a global solution $X_\e$. Moreover, for any sequence  $\{\e_n\}_{n\in\N}$ and  for any $T>0$, we have  that
		\begin{align}
		\nu_{\e_n}(\cdot)=\p\left\{(X_{\e_n},\W)\in\cdot\right\} \label{nu n define}
		\end{align}
		defines  a tight sequence in $\Pm{}\left(C([0,T];\Z)\times C([0,T];\U_0)\right)$.
	\end{Lemma}
	
	\begin{proof}
		From \ref{b in X}, \ref{gn hn condition}, it is easy to see that for each $n\ge1$, $H_{1,\e}(t,X)$ and $H_{2,\e}(t,X)$ are locally Lipschitz in $X\in\X$. Moreover, the growth of $\|H_{1,\e}(\cdot,X)\|_{\X}$ and $\|H_{2,\e}(\cdot,X)\|_{\LL_{2}(\U;\X)}$ is  controlled by the % locally bounded,
		continuous function $k(t)$. Therefore, for each $\e \in (0,1)$, there is a stopping time $\tau_\e^*>0$ almost surely such that the problem
		\eqref{approximation cut problem}  has a unique solution $X_\e\in  L^2\left(\Omega;C([0,\tau_\e^*);\X)\right)$, see \cite{Leha-Ritter-1985-MA} or \cite[Theorem 5.1.1]{Kallianpur-Xiong-1995-book}. Next, we prove that the solution is actually a global solution.  Using the It\^{o} formula in $\X$ for the regular mappings $g_\e, h_\e $, we find
		\begin{align}
		{\rm d}\|X_\e\|^2_{\X}
		=\,&2
		\sum_{k=1}^{\infty}\chi_R\big(\|X_\e\|_{\V}\big)\left(h_\e(t,X_\e)e_k,X_\e\right)_{\X}\,{\rm d}W_k
		+2\chi^2_R \big(\|X_\e\|_{\V}\big)
		\left(b(t,X_\e),X_\e\right)_{\X}\,{\rm d}t\nonumber\\
		&+2\chi^2_R\big(\|X_\e\|_{\V}\big)
		\left(g_\e(t,X_\e),X_\e\right)_{\X} \, {\rm d}t+
		\chi^2_R \big(\|X_\e\|_{\V}\big)\|h_\e(t,X_\e)\|_{\LL_2(\U;\X)}^2\, {\rm d}t\nonumber\\
		=:&\sum_{k=1}^{\infty} J_{1,k}\,{\rm d}W_k+\sum_{i=2}^4J_i \, {\rm d}t.
		\label{X n 2}
		\end{align}
		%	where $\{e_k\}_{k\in\N}$ is a complete orthonormal basis of $\U$.
		For any $T>0$, we 
		integrate \eqref{X n 2}, take a supremum for $t\in[0,T]$ and then use the BDG inequality, \ref{b in X} and \ref{growth gn hn} to find a
		constant $C=C_R>0$ depending on $R$ such that 
		\begin{align*}
		\E\sup_{t\in[0,T]}\|X_\e\|^2_{\X}-\E\|X_0\|^2_{\X}
		\lesssim\,& \E\left(\int_0^T\sum_{k=1}^\infty
		J^2_{1,k}
		\dd t\right)^{\frac12}
		+\int_0^T|J_2|\,{\rm d}t+\int_0^T
		\left|J_3+J_4\right|\,{\rm d}t\notag\\
		\lesssim\,&\E\left(\int_0^Tk(t)\chi^2_R(\|X_\e\|_{\V})
		f(\|X_\e\|_{\V})\|X_\e\|^4_{\X}\dd t\right)^{\frac12}\notag\\
		&+\int_0^Tk(t)\chi^2_R(\|X_\e\|_{\V})f(\|X_\e\|_{\V})\|X_\e\|^2_{\X}\,{\rm d}t\nonumber\\
		\leq\,&C_{R}\E\left(\sup_{t\in[0,T]}\|X_\e\|^2_{\X}\int_0^Tk(t)
		\|X_\e\|^2_{\X}\dd t\right)^{\frac12}+C_{R}\int_0^Tk(t)\|X_\e\|_{\X}^{2}\,{\rm d}t\notag\\
		\leq\,&\frac{1}{2}\E\sup_{t\in[0,T]}\|X_\e\|_{\X}^{2}+C_{R}\int_0^Tk(t)\E\sup_{t'\in[0,t]}\|X_\e\|_{\X}^{2}\,{\rm d}t.
		\end{align*}
		Via Gr\"{o}nwall's inequality, we arrive at the $\e$-independent bound
		\begin{align}
		\sup_{\e\in(0,1)}\E\sup_{t\in[0,T]}\|X_\e(t)\|^2_{\X}
		\leq C(R,X_0,T).\label{Xn uniform bound}
		\end{align}
		Since $T>0$ can be chosen arbitrarily, we see in particular that    $X_\e$ is a global solution for each $\e \in (0,1)$.\\
		Moreover, the bound \eqref{Xn uniform bound} implies that 
		the stopping times 
		\begin{equation}\label{tau N n}
		\tau_N^{\e}:=\inf\{t\ge 0: \sup_{t'\in[0,t]}\|X_\e\|_{\X}\ge N\},\ N\ge1,\ \e\in(0,1)
		\end{equation}
		satisfy
		\begin{equation}\label{tau N n<T prob} 
		\p(\tau_N^{\e}<T) \le \p\left(\sup_{t\in[0,T]}\|X_\e\|_{\X}\ge N\right)\leq\frac{C(R,X_0,T)} {N^2}.
		\end{equation}
		%		Let $\chi^2_R\big(\|X\|_{\V}\big)\left[b(t,X)+g_\e(t,X)\right]=H_{1,\e}(t,X)$ and $\chi_R\big(\|X\|_{\V}\big)h_\e(t,X)=H_{2,\e}(t,X).$ 
		Now we turn to prove the tightness result  on the Borel measure in  \eqref{nu n define}. 
		For any given  $\delta\in ( 0,1)$, we get that
		\begin{align}
		& \hspace*{-0.5cm}\E\sup_{[t_1,t_2]\subset [0,T],t_2-t_1<\delta}\left(1\wedge\|X_\e(t_2)-X_\e(t_1)\|_{\Z}\right)\notag\\
		\leq\,&\E\left(\sup_{[t_1,t_2]\subset [0,T],t_2-t_1<\delta}\left(1\wedge\|X_\e(t_2)-X_\e(t_1)\|_{\Z}\right)
		{\mathbf 1}_{\{\tau^{\e}_{N}<T\}}\right)\notag\\
		&+\E\left(\sup_{[t_1,t_2]\subset [0,T],t_2-t_1<\delta}\left(1\wedge\|X_\e(t_2)-X_\e(t_1)\|_{\Z}\right)
		{\mathbf 1}_{\{\tau^{\e}_{N}\geq T\}}\right)\notag\\
		\leq\,& \p\{\tau^{\e}_{N}<T\}
		+\E\left(\sup_{[t_1,t_2]\subset [0,T\wedge \tau^{\e}_{N}],t_2-t_1<\delta}
		\left(1\wedge\|X_\e(t_2)-X_\e(t_1)\|_{\Z}\right)
		{\mathbf 1}_{\{\tau^{\e}_{N}\geq T\}}\right)\notag\\
		\leq\,& \frac{C(R,X_0,T)} {N^2}
		+\E\left(\sup_{[t_1,t_2]\subset [0,T\wedge \tau^{\e}_{N}],t_2-t_1<\delta}
		\left(1\wedge\|{\mathbf 1}_{\{\tau^{\e}_{N}\geq T\}}X_\e(t_2)-{\mathbf 1}_{\{\tau^{\e}_{N}\geq T\}}X_\e(t_1)\|_{\Z}\right)
		\right)\label{continuity 2 parts}
		\end{align}
		holds. Note that we used   the $\e$-independent bound  \eqref{tau N n<T prob} for the last inequality.
		To estimate the expectation 
		term in \eqref{continuity 2 parts} we utilize the approximative problem  \eqref{approximation cut problem} directly. We 
		start with the drift term $H_{1,\e}$.
		On account of \eqref{tau N n}, \ref{b in X} and \ref{gn hn condition} and the BDG inequality, there are a non-decreasing,  locally bounded function 
		$a(\cdot) \in C\left([0,+\infty);[0,+\infty)\right)$ and a constant $C>0$ independent of $\e$ such that we have 
		\begin{align}\label{est1}
		& \hspace*{-1cm}\E\left\|\int_{t_1}^{t_2}{\mathbf 1}_{\{\tau^{\e}_{N}\geq T\}}H_{1,\e}(t',X_\e(t')\,{\rm d}t'\right\|_{\Z}\notag\\[1.5ex]
		\leq\,& |t_2-t_1|\E
		\sup_{t\in[0,T\wedge\tau_N^{\e}]}\|H_{1,\e}(t,X_\e(t)\|_{\Z}\notag\\
		\leq\,& C|t_2-t_1|\E\sup_{t\in[0,T\wedge\tau_N^{\e}]}
		\left(\chi^2_R \big(\|X_\e\|_{\V}\big)
		k(t)f(\|X_\e\|_{\V})\|X_\e\|_{\X}
		+\chi^2_R\big(\|X_\e\|_{\V}\big)\|g_\e(X_\e)\|_{Z}\right)\notag\\
		\leq\,& Ck(T)|t_2-t_1|\E\sup_{t\in[0,T]}\left(f(CN)N+q(N)\right)
		\leq Ca(N)k(T)|t_2-t_1|.
		\end{align}
		For the diffusion operator $H_{2,\e}$ and the stochastic integral, the bound  \eqref{tau N n}, \ref{gn hn condition} and the BDG inequality  imply 
		\begin{align}\label{est2}
		&\hspace*{-1.0cm}\E\left(\left\|\int_{t_1}^{t_2}{\mathbf 1}_{\{\tau^{\e}_{N}\geq T\}}H_{2,\e}(t',X_\e(t') {\rm d}\W
		\right\|_{\Z}\right)\\
		\leq\, &\E\left(\sup_{t_*\in[t_1,t_2]}\left\|\int_{t_1}^{t_*}
		{\mathbf 1}_{\{\tau^{\e}_{N}\geq T\}}H_{2,\e}(t',X_\e(t')) {\rm d}\W
		\right\|_{\Z}\right)\notag\\
		\leq\,& C \E\left(\int_{t_1}^{t_2} \|{\mathbf 1}_{\{\tau^{\e}_{N}\geq T\}}H_{2,\e}(t',X_\e(t')\|^2_{\LL_2(\U; \Z)} {\rm d}\tau\right)^\frac{1}{2}\notag\\
		\leq\,& C|t_2-t_1|^{\frac{1}{2}}\E \sup_{t\in[0,T\wedge\tau_N^{\e}]}\left(q(N)k(t)\right)\notag\\
		\leq\,& Ca(N)k(T)|t_2-t_1|^{\frac{1}{2}}.
		\end{align}
		Combining the  estimates \eqref{est1}, \eqref{est2}, for any $\delta\in(0,1)$, one has
		\begin{align*}
		&\hspace*{-1.0cm}\E\sup_{[t_1,t_2]\subset [0,T\wedge \tau^{\e}_{N}],t_2-t_1<\delta}\|{\mathbf 1}_{\{\tau^{\e}_{N}\geq T\}}X_\e(t_2)-{\mathbf 1}_{\{\tau^{\e}_{N}\geq T\}}X_\e(t_1)\|_{\Z}\\
		\leq \,&
		C\E\sup_{[t_1,t_2]\subset [0,T\wedge \tau^{\e}_{N}],t_2-t_1<\delta}\left\|\int_{t_1}^{t_2} {\mathbf 1}_{\{\tau^{\e}_{N}\geq T\}}H_{1,\e}(t',X_\e(t'))\,{\rm d}t'\right\|_{\Z}\notag\\
		&+C\E\sup_{[t_1,t_2]\subset [0,T\wedge \tau^{\e}_{N}],t_2-t_1<\delta}\left\|\int_{t_1}^{t_2} {\mathbf 1}_{\{\tau^{\e}_{N}\geq T\}}H_{2,\e}(t',X_\e(t'))\,{\rm d}\W\right\|_{\Z}\\
		\leq\, & Ca(N)k(T)\delta^{\frac{1}{2}}.
		\end{align*}
		Therefore, returning to \eqref{continuity 2 parts},    the last  estimate  implies that for all $\delta\in(0,1)$,
		\begin{align*}
		\E\sup_{[t_1,t_2]\subset [0,T],t_2-t_1<\delta}\left(1\wedge\|X_\e(t_2)-X_\e(t_1)\|_{\Z}\right) 
		%	\leq\,& \inf_{N\ge1}\left\{\frac{C(R,X_0,T)} {N^2}
		%	+\E\left(\sup_{[t_1,t_2]\subset [0,T\wedge \tau^{\e}_{N}],t_2-t_1<\delta}
		%	\|{\mathbf 1}_{\{\tau^{\e}_{N}\geq T\}}X_\e(t_2)-{\mathbf 1}_{\{\tau^{\e}_{N}\geq T\}}X_\e(t_1)\|_{\Z}
		%	\right)\right\}\\
		\leq\,& \inf_{N\ge1}\left\{\frac{C(R,X_0,T)} {N^2}+Ca(N)k(T)\delta^{\frac{1}{2}}\right\}.
		\end{align*}
		Because $a(\cdot)$ is non-decreasing, we have
		\begin{align*}
		\lim_{\delta\rightarrow 0}\sup_{\e\in(0,1)}
		\E\sup_{[t_1,t_2]\subset [0,T],t_2-t_1<\delta} \|X_\e(t_2)-X_\e(t_1)\|_{\Z} 
		=0.
		\end{align*}
		Thus, we obtain that, for any $\delta>0$, the limit
		\begin{align}
		\lim_{\delta\rightarrow 0}\sup_{\e\in(0,1)} 
		\p\left(\sup_{[t_1,t_2]\subset [0,T],t_2-t_1<\delta} \|X_\e(t_2)-X_\e(t_1)\|_{\Z}>\delta\right) 
		=0\label{equi continuity}
		\end{align}
		holds.
		Since $\X\hookrightarrow\hookrightarrow \Z$, for each $t\ge0$, $\p(X_\e(t)\in\cdot)$ is tight in $\Pm{}(\Z)$. This together with \eqref{equi continuity} means  for any vanishing sequence $\{\e_n\}_{n\in\N}$ that (cf. \cite[Theorem 3.17]{Gawarecki-Mandrekar-2010-Springer})
		$$\mu_{\e_n}(\cdot)=\p\left\{X_{\e_n}\in\cdot\right\}$$ is a tight sequence in $\Pm{}\left(C([0,T];\Z)\right)$. On the other hand, since $\W$ stays unchanged, $\nu_{\e_n}$ defined in \eqref{nu n define} is also tight.
	\end{proof}

	\subsubsection{Stochastic compactness}  On the basis of Lemma \ref{Tightness} and the weak stochastic compactness theory we can now characterize the 
	convergence of the sequence $\{X_\e\}$ obtaining global-in-time results.
	\begin{Lemma}\label{convergence}
		Let $R>1$, $T>0$. The sequence 
		$\{\nu_{\e_n}\}$ defined in Lemma \ref{Tightness} has a weakly convergent subsequence, still denoted by $\{\nu_\e\}$, with limit measure $\nu$.
		There is a probability space $\left(\widetilde{\Omega}, \widetilde{\mathcal{F}},\widetilde{\p}\right)$ on which there is a sequence of  random variables $\left(\widetilde{X_{\e}},\widetilde{\W_{\e}}\right)$ and a pair $\left(\widetilde{X},\widetilde{\W}\right)$ such that we have 
		\begin{align}\label{same distribution}
		\widetilde{\p}\left\{\big(\widetilde{X_{\e}},\widetilde{\W_{\e}}\big)\in\cdot\right\}
		=\nu_{n}(\cdot),\ \ \widetilde{\p}\left\{\big(\widetilde{X},\widetilde{\W}\big)\in\cdot\right\}=\nu(\cdot),	
		\end{align}
		and
		\begin{align}
		\widetilde{X_{\e}}\rightarrow \widetilde{X}\  &{\rm in}\ C\left([0,T];\Z\right) \ \ \text{and}\ \ \
		\widetilde{\W_{\e}}\rightarrow \widetilde{\W}\  {\rm in}\ C\left([0,T];\U_0\right) \ \ \widetilde{\p}-a.s.\label{a.s. convergence}
		\end{align}
		Moreover, for $t\in [0,T]$, the following results hold.
		\begin{itemize}
			\item[(i)] $\widetilde{\W_{\e}}$ is a cylindrical Wiener process with respect to
			$\widetilde{\mathcal{F}^{\e}_t}
			=\sigma\left\{\widetilde{X_{\e}}(\tau),\widetilde{\W_{\e}}(\tau)\right\}_{\tau\in[0,t]}$.
			\item[(ii)] $\widetilde{\W}$ is a cylindrical Wiener process with respect to $\widetilde{\mathcal{F}_t}
			=\sigma\left\{\widetilde{X}(\tau),\widetilde{\W}(\tau)\right\}_{\tau\in[0,t]}$.
			\item[(iii)] On $\bigg(\widetilde{\Omega}, \widetilde{\mathcal{F}},\widetilde{\p},
			\left\{\widetilde{\mathcal{F}^{\e}_t}\right\}_{t\geq0}\bigg)$,
			$\widetilde{\p}-a.s.$ we have
			\begin{equation} \label{new approximate problem}
			\widetilde{X_\e}(t)-\widetilde{X_\e}(0)
			=\int_0^t	\chi^2_R(\|\widetilde{X_\e}\|_{\V})\left[b(t',\widetilde{X_\e})+g_\e(t',\widetilde{X_\e})\right] {\rm d}t'+\int_0^t \chi_R(\|\widetilde{X_\e}\|_{\V})h_\e(t',\widetilde{X_\e}) \,{\rm d}\widetilde{\W_{\e}}.
			\end{equation}
		\end{itemize}
	\end{Lemma}
	
	\begin{proof}
		The existence of the sequence $\big(\widetilde{X_{\e}},\widetilde{\W_{\e}}\big)$ satisfying \eqref{a.s. convergence} is a consequence of Lemma \ref{Tightness} and Theorems \ref{Prokhorov Theorem} and \ref{Skorokhod Theorem}.
		Besides,
		\cite[Theorem 2.1.35 and Corollary 2.1.36]{Breit-Feireisl-Hofmanova-2018-Book} imply that 
		$\widetilde{\W_{\e}}$ and $\widetilde{\W}$ are cylindrical Wiener processes relative to
		$\widetilde{\mathcal{F}^{\e}_t}
		=\sigma\big\{\widetilde{X_\e}(\tau),\widetilde{\W_{\e}}(\tau)\big\}_{\tau\in[0,t]}$ and
		$\widetilde{\mathcal{F}_t}
		=\sigma
		\big\{\widetilde{X}(\tau),\widetilde{\W}(\tau)\big\}_{\tau\in[0,t]}$, respectively.   As in \cite[page 282]{Bensoussan-1995-AAM} or \cite[Theorem 2.9.1]{Breit-Feireisl-Hofmanova-2018-Book}
		one can find that $\big(\widetilde{X_{\e}},\widetilde{\W_{\e}}\big)$ relative to  $\big\{\widetilde{\mathcal{F}^{\e}_t}\big\}_{t\geq0}$ satisfies  \eqref{new approximate problem} $\widetilde{\p}-a.s.$
	\end{proof}
	
	\subsubsection{Concluding the  proof of \ref{Abstract irregular SDE: M} in Theorem \ref{Abstract irregular SDE: T}}\label{prove M solution}
	To begin with, we  notice that the embedding $\X\hookrightarrow\Z$ is continuous, which means there exist    continuous maps $\pi_m: \Z\to \X,\ m\ge 1$ such that
	$$ \|\pi_m x\|_\X \le \|x\|_\X,\ \  \lim_{m\rightarrow\infty} \|\pi_m x\|_\X = \|x\|_\X,\ \ x\in \Z,$$  where $\|x\|_\X:=\infty$ if $x\notin\X.$ 
	This,  together with \eqref{Xn uniform bound},  \eqref{same distribution} and Fatou's lemma, yields 
	\begin{align}\label{solution M estimate}
	\widetilde{\E}\sup_{t\in[0,T]} \|\widetilde{X}\|_{\X}^2
	\leq\,&
	\liminf_{m\rightarrow \infty}
	\widetilde{\E}\sup_{t\in[0,T]} \|\pi_m\widetilde{X}\|_{\X}^2\notag\\
	\leq\,&
	\liminf_{m\rightarrow \infty}
	\liminf_{\e\rightarrow 0}
	\widetilde{\E}\sup_{t\in[0,T]} \|\pi_m\widetilde{X_\e}\|_{\X}^2\notag\\
	\leq\,& \liminf_{m\rightarrow \infty}
	\liminf_{\e\rightarrow 0} \E\sup_{t\in[0,T]} \|X_\e\|_{\X}^2
	<C(R,X_0,T).
	\end{align} 	
	Using \eqref{a.s. convergence}, \eqref{solution M estimate}, $\X\hookrightarrow\V$, \ref{gn hn condition} and Lemma \ref{stochastic integral convergence} (up to further subsequence) in \eqref{new approximate problem},  we obtain that
%	 for any $t\in[0,T] $ the convergence 
	$$\int_0^t \chi_R\big(\|\widetilde{X_\e}\|_{\V}\big)h_\e(t,\widetilde{X_\e}) \,{\rm d}\widetilde{\W_{\e}}
	\xrightarrow{\e\rightarrow 0}
	\int_0^t \chi_R\big(\|\widetilde{X}\|_{\V}\big)h(t,\widetilde{X}) \,{\rm d}\widetilde{\W}\ 
	\text { in } L^{2}(0, T ; \Z)\ \  \p-a.s.$$
	As before, it follows from \eqref{a.s. convergence}, \eqref{solution M estimate}, $\X\hookrightarrow\V$ and \ref{gn hn condition}  that for any $t\in[0,T]$ and $\phi\in\Z^*$,
	$$\int_0^t\chi^2_R\big(\|\widetilde{X}\|_{\V}\big)\tensor[_\Z]{\left\langle b(s,\widetilde{X_\e}(s))-b(s,\widetilde{X}(s))+ g_\e(s,\widetilde{X_\e}(s))-g(s,\widetilde{X}(s)),\phi\right\rangle}{_{\Z^*}}\dd s\xrightarrow{\e\rightarrow0}0,
	\ \ \p-a.s.$$
	Therefore we derive   that for all $\phi\in\Z^*$ and ${\rm d}t\otimes\widetilde{\p}-a.s.$,
	\begin{align*}
	&\hspace*{-0.5cm}\tensor[_\Z]{\left\langle\widetilde{X}(t),\phi\right\rangle}{_{\Z^*}}-
	\tensor[_\Z]{\left\langle\widetilde{X}(0),\phi\right\rangle}{_{\Z^*}}\\
	=\,&\int_0^t\chi^2_R\big(\|\widetilde{X}\|_{\V}\big)\tensor[_\Z]{\left\langle b(s,\widetilde{X}(s)) +g(s,\widetilde{X_\e}(s)),\phi\right\rangle}{_{\Z^*}}\dd s
	+\tensor[_\Z]{\left\langle\int_0^t \chi_R\big(\|\widetilde{X}\|_{\V}\big)h(t,\widetilde{X})\, {\rm d}\widetilde{\W},\phi\right\rangle}{_{\Z^*}}.
	\end{align*}
	Due to \eqref{solution M estimate}, \ref{b in X} and \ref{gn hn condition}, we see that 
	$t\mapsto\int_0^t\chi_R\big(\|\widetilde{X}\|_{\V}\big)h(t',\widetilde{X}(t')) \, {\rm d}\widetilde{\W}$ is a local continuous martingale  on $\Y\subset\Z$, and that $t\mapsto\int_0^t\chi^2_R\big(\|\widetilde{X}\|_{\V}\big)\left[b(t',\widetilde{X}(t'))+g(t',\widetilde{X}(t'))\right]\, {\rm d}t'$ is a continuous process on $\Z$ as well. Hence, we obtain that $\widetilde{X}$ is a global martingale solution to \eqref{cut-off problem}.
	Moreover, \eqref{a.s. convergence} and \eqref{solution M estimate} imply  that $\widetilde{X}\in L^2\big(\widetilde{\Omega};L^\infty(0,T;\X)\cap C([0,T];\Z)\big)$ holds.  Define
	$$\widetilde{\tau}=\inf\left\{t\ge0:\|\widetilde{X}(t)\|_{\V}>R\right\},$$
	then we see that $\left(\widetilde{\s},\widetilde{X},\widetilde{\tau}\right)$ is a local martingale solution to \eqref{Abstact irregular SDE}, where $\widetilde{\s}=\left(\widetilde{\Omega}, \widetilde{\mathcal{F}},\widetilde{\p},
	\{\widetilde{\mathcal{F}_t}\}_{t\geq0},\widetilde{\W}\right)$ with $\left\{\widetilde{\mathcal{F}_t}\right\}_{t\geq0}
	=\sigma\left\{\widetilde{X}(\tau),\widetilde{\W}(\tau)\right\}_{\tau\in[0,t]}$. We have  finished the proof.

	\subsection{Proof of \ref{Abstract irregular SDE: P} in Theorem \ref{Abstract irregular SDE: T}}\label{proof of Thm-P}
	
	To obtain a pathwise solution to \eqref{Abstact irregular SDE}, we will use \ref{Abstract irregular SDE: M} in Theorem \ref{Abstract irregular SDE: T} and the Gy\"{o}ngy-Krylov Lemma, cf. Lemma \ref{Gyongy convergence}. The proof can naturally be broken down into several subsections.
	
	\subsubsection{Pathwise uniqueness of the cut-off problem}

	We first  state the following result which indicates that for $L^\infty(\Omega)$-initial values, the solution map is time 
	locally Lipschitz in the  less regular space $\Z$.
	
	\begin{Lemma}\label{local time Lipschitz}
		Let $\s=\left(\Omega, \mathcal{F},\p,\{\mathcal{F}_t\}_{t\geq0},\W\right)$ be a fixed stochastic basis and let \ref{uniqueness requirement} hold. Let $M>0$ be a constant. Assume that $X_0$ and $Y_0$ are  two $\X$-valued $\mathcal{F}_0$-measurable random variables satisfying $\|X_0\|_{\X},\|Y_0\|_{\X}<M$ almost surely. \\
		Let $(\s,X,\tau_1)$ and $(\s,Y,\tau_2)$ be two local pathwise solutions to \eqref{Abstact irregular SDE}  such that $X(0)=X_0$, $Y(0)=Y_0$ almost surely,  and $X(\cdot\wedge \tau_1),Y(\cdot\wedge \tau_2)\in L^2\left(\Omega;C([0,\infty);\X)\right)$ for $i=1,2$.\\
		Then, for any $T>0$, there exists a constant $C(M,T)>0$ such that
		\begin{equation}\label{local time Lip estimate}
		\displaystyle\E\sup_{t\in[0,\tau^T_{X,Y}]}\|X(t)-Y(t)\|^2_{\Z}\leq C(M,T)\E\|X_0-Y_0\|^2_{\Z}.
		\end{equation}
		In \eqref{local time Lip estimate} we used 
		\begin{align}
		\tau^T_{X}:=\inf\left\{t\geq0: \|X(t)\|_{\X}>M+2\right\}\wedge T,\ \ 
		\tau^T_{Y}:=\inf\left\{t\geq0: \|Y(t)\|_{\X}>M+2\right\}\wedge T,\label{stopping time local Lip}
		\end{align}
		and
		$
		\tau^T_{X,Y}=\tau^T_{X}\wedge \tau^T_{Y}.
		$
	\end{Lemma}
	
	\begin{proof}
		Let
		$Z=X-Y$. Then $Z$ satisfies the stochastic differential equation
		\begin{align*}
		{\rm d}\|Z\|^2_{\Z}
		=\,&2 \left(\left[h(t,X)- h(t,Y)\right]{\rm d}\W, Z\right)_{\Z}
		+2\left(b(t,X)- b(t,Y), Z\right)_{\Z}\,{\rm d}t\nonumber\\
		&+2\left(g(t,X)- g(t,Y),Z\right)_{\Z}\,{\rm d}t
		+ \|h(t,X)-h(t,X)\|_{\LL_2(\U; \Z)}^2\,{\rm d}t.
		\end{align*}
		By \ref{b in X}, \ref{uniqueness requirement},  It\^o's formula (which holds true on the entire space $\Z$),  and the BDG inequality, we find
		for some $C>0$ depending on $b, g, h$ the estimate	\begin{align*}
		&\hspace*{-0.6cm}\E\sup_{t\in[0,\tau^T_{X,Y}]}\|Z(t)\|^2_{\Z} -\E\|Z(0)\|^2_{\Z} \notag\\
		\leq\,&
		C\E\left(\int_0^{\tau^T_{X,Y}}\|h(t,X)- h(t,Y)\|_{\LL_2(\U,\Z)}^2\|Z\|_{\Z}^2\,{\rm d}t\right)^{\frac12}
		+\E\int_0^{\tau^T_{X,Y}}q(M+2)k(t)\|Z(t)\|^2_{\Z}\,{\rm d}t\notag\\
		\leq\,&
		C q(M+2)\E
		\left(\sup_{t\in[0,\tau^T_{X,Y}]}\|Z\|_{\Z}^2\cdot
		\int_0^{\tau^T_{X,Y}}k^2(t)\|Z\|_{\Z}^2
		{\rm d}t\right)^{\frac12}
		+Cq(M+2)\int_0^{T}k(t)\E\sup_{t'\in[0,\tau^t_{X,Y}]}\|Z(t')\|^2_{\Z}\,{\rm d}t\notag\\
		\leq &\frac12\E\sup_{t\in[0,\tau^T_{X,Y}]}\|Z\|_{\Z}^2
		+C(M,T)\int_0^{T}
		\E\sup_{t'\in[0,\tau^t_{X,Y}]}\|Z(t')\|^2_{\Z}\,{\rm d}t.
		\end{align*}
		If we apply  Gr\"{o}nwall's inequality to the  estimate above, we get \eqref{local time Lip estimate}.
		%		$$\displaystyle\E\sup_{t\in[0,\tau^T_{X,Y}]}\|Z(t)\|^2_{\Z}
		%		\leq C(M,T)\E\|Z(0)\|^2_{\Z},$$ which is 
	\end{proof}
	
	\begin{Lemma}\label{pathwise uniqueness original}
		Let
		$\s=\left(\Omega, \mathcal{F},\p,\{\mathcal{F}_t\}_{t\geq0},\W\right)$ be a fixed stochastic basis and let \ref{uniqueness requirement} hold. Let $X_0$ be an $\X$-valued $\mathcal{F}_0$-measurable random variable satisfying $\E\|X_0\|^2_{\X}<\infty$. If $(\s,X_1,\tau_1)$ and $(\s,X_2,\tau_2)$ are two local pathwise solutions to \eqref{Abstact irregular SDE}  %obtained in Proposition \ref{global martingale solution cut-off}
		satisfying $X_i(\cdot\wedge \tau_i)\in L^2\left(\Omega;C([0,\infty);\X)\right)$ for $i=1,2$ and
		$\p\{X_1(0)=X_2(0)=X_0\}=1$, then 
		\begin{equation*}
		\p\left\{X_1=X_2 , \ \forall\
		t\in[0,\tau_1\wedge\tau_2]\right\}=1.
		\end{equation*}
	\end{Lemma}
	
	\begin{proof}
		We first assume that $\|X_0\|_{\X}<M$ $\p-a.s.$ for some deterministic $M>0$.  For any $K>2M$ and $T>0$,  we define 
		\begin{align*}
		\tau^T_{K}:=\inf\left\{t\geq0: \|X_1(t)\|_{\X}+\|X_2(t)\|_{\X}>K\right\}\wedge T.
		%\label{stopping time uniqueness}
		\end{align*}
		Then one can repeat all steps in the proof of \eqref{local time Lip estimate} by using $\tau^T_{K}$ instead of $\tau^T_{X,Y}$ to find  
		$$\displaystyle\E\sup_{t\in[0,\tau^T_{K}]}\|X_1(t)-X_2(t)\|^2_{\Z}\leq C(K,T)\E\|X_1(0)-X_2(0)\|^2_{\Z}=0.$$ 
		It is easy to see that
		\begin{align}
		\p\{\liminf_{K\rightarrow\infty}\tau^T_{K}\geq\tau_1\wedge\tau_2\wedge T\}=1.
		\label{tau-k>tau1 tau2}
		\end{align}
		Sending $K\rightarrow\infty$, using the monotone convergence theorem and \eqref{tau-k>tau1 tau2} with noticing $T>0$ is arbitrary, we obtain the desired result for $X_0$ being almost surely bounded.
		
		It remains to remove this restriction. 
		Motivated by \cite{GlattHoltz-Ziane-2009-ADE,GlattHoltz-Vicol-2014-AP}, for general $\X$-valued $\mathcal{F}_0$-measurable initial data such that $\E\|X_0\|^2_{\X}<\infty$ holds, we define 
		$\Omega_k=\{k-1\leq\|X_0\|_{\X}<k\}$, $k\geq1$. Then we see that $\Omega_k\bigcap\Omega_{k'}=\emptyset$ for $k\neq k'$, and $\bigcup_{k\geq1}\Omega_k$ is a set of full measure. Consider
		\begin{align*}
		X_0(\omega)&
		=\sum_{k\geq1}X_0(\omega,x)\textbf{1}_{k-1\leq\|X_0\|_{\X}<k}=:\sum_{k\geq1}X_{0,k}(\omega)\ \  \p-a.s.
		\end{align*}
		Notice that
		\begin{align*}
		&\hspace*{-0.5cm}\textbf{1}_{\Omega_k}X_1(t\wedge \tau_1)-\textbf{1}_{\Omega_k}X(0)\\
		=\,&\textbf{1}_{\Omega_k}\int_0^{t\wedge \tau_1} b(t',X_1)\,{\rm d}t'+\textbf{1}_{\Omega_k}\int_0^{t\wedge \tau_1}g(t',X_1)\,{\rm d}t'
		+\textbf{1}_{\Omega_k}\int_0^{t\wedge \tau_1} h(t,X_1)\,{\rm d}\W\\
		=\,&\int_0^{t\wedge \textbf{1}_{\Omega_k}\tau_1}\textbf{1}_{\Omega_k}b(t',X_1)\,{\rm d}t'
		+\int_0^{t\wedge \textbf{1}_{\Omega_k}\tau_1}\textbf{1}_{\Omega_k}g(t',X_1)\,{\rm d}t'
		+\int_0^{t\wedge \textbf{1}_{\Omega_k}\tau_1}\textbf{1}_{\Omega_k} h(t',X_1)\,{\rm d}\W.
		\end{align*}
		Due to $\textbf{1}_{\Omega_k} F(t,X_1)=F(t,\textbf{1}_{\Omega_k} X_1)-\textbf{1}_{\Omega^C_k}F(t,\mathbf{0})$ for $F\in\{b,g,h\}$, and \ref{b in X}, \ref{gn hn condition} we get  $\|b(t,\mathbf{0})\|_{\X}$, $\|g(t,\mathbf{0})\|_{\Z}$, $\|h(t,\mathbf{0})\|_{\LL_2(\U;\Y)}<\infty$. Then we can proceed
		with
		%	\begin{align*}
		%	\int_0^{t\wedge \tau_1} \textbf{1}_{\Omega_k}h(t',X_1)\,{\rm d}\W
		%	=\int_0^{t\wedge \textbf{1}_{\Omega_k}\tau_1}
		%	\left(h(t',\textbf{1}_{\Omega_k} X_1)-\textbf{1}_{\Omega^C_k}h(t',\mathbf{0})\right)\,{\rm d}\W
		%	=\int_0^{t\wedge \textbf{1}_{\Omega_k}\tau_1}
		%	h(t',\textbf{1}_{\Omega_k} X_1)\,{\rm d}\W\ \ \p-a.s.
		%	\end{align*}
		%	Similar equations can be obtained for $b$ and $g$. Therefore we have
		%%	\begin{equation*}
		%%		\textbf{1}_{\Omega_k}\int_0^{t\wedge \tau_1} b(t',X_1)\,{\rm d}\W
		%%	\end{equation*}
		%%	we see that
		\begin{align*}
		&\hspace*{-0.5cm}\textbf{1}_{\Omega_k}X_1(t\wedge \textbf{1}_{\Omega_k}\tau_1)-X_{0,k}\\
		=\,&\int_0^{t\wedge \textbf{1}_{\Omega_k}\tau_1}b(t',\textbf{1}_{\Omega_k}X_1)\,{\rm d}t'+\int_0^{t\wedge \textbf{1}_{\Omega_k}\tau_1}g(t',\textbf{1}_{\Omega_k}X_1)\,{\rm d}t'
		+\int_0^{t\wedge \textbf{1}_{\Omega_k}\tau_1}
		h(t',\textbf{1}_{\Omega_k} X_1)\,{\rm d}\W\ \ \p-a.s.,
		\end{align*}
		which means that $(\textbf{1}_{\Omega_k}X_1,\textbf{1}_{\Omega_k}\tau_1)$ is a solution to \eqref{Abstact irregular SDE}  with
		initial data $X_{0,k}$.	
		%	
		%	
		%	
		%	
		%	Since $h(t,0)=0$, $g(t,0)=0$ and $b(t,0)=0$, we have  $h(t,\textbf{1}_{\Omega_k},X_1)=\textbf{1}_{\Omega_k}h(t,X_1)$, and similarly $g$ and $b$ also enjoy this property. Therefore we find that $(X_1\textbf{1}_{\Omega_k},\tau_1\textbf{1}_{\Omega_k})$ is a solution to \eqref{Abstact irregular SDE} with initial data $X_{0,k}$. 
		Similarly, $(\textbf{1}_{\Omega_k}X_2,\textbf{1}_{\Omega_k}\tau_2)$  is also a solution to \eqref{Abstact irregular SDE} with initial data $X_{0,k}$. 
		Altogether we obtain  $\textbf{1}_{\Omega_k}X_1=\textbf{1}_{\Omega_k}X_{2}$ on $[0,\textbf{1}_{\Omega_k}\tau_1\wedge \textbf{1}_{\Omega_k}\tau_2]$ almost surely.   Because 	$X_i
		=\sum_{k\geq1}X_i\textbf{1}_{\Omega_k}$ and  $\tau_i=\sum_{k\geq1}\tau_i\textbf{1}_{\Omega_k}$  almost surely for $i=1,2$,  $\Omega_k\bigcap\Omega_{k'}=\emptyset$ for $k\neq k'$ and $\bigcup_{k\ge1}\Omega_k$ is a set of full measure, 
		we have
		$$\p\big\{X_1=X_2 \,  \forall\ t\in[0,\tau_1\wedge \tau_2]\big\}\geq\p\big\{\cup_{ k\geq1}\Omega_k\big\}
		=1,$$
		which completes the proof.
	\end{proof}
	
	For the cut-off problem \eqref{cut-off problem}, we also have
	pathwise uniqueness. Indeed,  since $\Z\hookrightarrow\V$,  the additional terms coming from the cut-off function $\chi_R(\cdot)$ can
	be handled by the mean value theorem as 
	$$\left|\chi_R\big(\|X_1\|_{\V}\big)-\chi_R\big(\|X_2\|_{\V}\big)\right|\leq C\|X_1-X_2\|_{\V}\leq C \|X_1-X_2\|_{\Z}.$$ 
	Then one can modify the proof of  Lemma \ref{pathwise uniqueness original} in a straightforward way  to get
	
	\begin{Lemma}\label{pathwise uniqueness cut-off}
		Let $T>0$ and
		$\s=\left(\Omega, \mathcal{F},\p,\{\mathcal{F}_t\}_{t\geq0},\W\right)$ be a fixed stochastic basis. Let \ref{uniqueness requirement} hold  and  let 
		$X_0$ be an $\X$-valued $\mathcal{F}_0$-measurable random variable satisfying $\E\|X_0\|^2_{\X}<\infty$. \\ 
		If  $(\s,X_1,T)$ and $(\s,X_2,T)$ are two solutions, on the same basis $\s$, of \eqref{cut-off problem} such that $\p\{X_1(0)=X_2(0)=X_0\}=1$ and $X_i\in L^2\left(\Omega;C([0,T);\X)\right)$ for $i=1,2$, then
		\begin{equation*}
		\p\left\{X_1=X_2 \, \forall\
		t\in[0,T]\right\}=1.
		\end{equation*}
	\end{Lemma}

	\subsubsection{Pathwise solution to the cut-off problem}
	
	Now we prove the existence and uniqueness of a pathwise solution to \eqref{cut-off problem}. To be more precise, we are going to show the following result.
	
	\begin{Lemma}\label{global pathwise solution cut-off}
		Let $\s=(\Omega, \mathcal{F},\p,\{\mathcal{F}_t\}_{t\geq0}, \W)$ be a fixed stochastic basis.  Let $X_0\in L^2(\Omega;\X)$ be an $\mathcal{F}_0$-measurable random variable. \\
		If Assumptions \ref{b in X}-\ref{continuity requirement} hold, then \eqref{cut-off problem} has a unique global pathwise solution $X$ which satisfies  for any $T>0$ 
		\begin{equation}\label{global_cutoof}
		X\in L^2\left(\Omega;C([0,T];\X)\right).
		\end{equation} 
	\end{Lemma}
	
	\begin{proof} 
		Uniqueness is a direct consequence of Lemma \ref{pathwise uniqueness cut-off}. 
		The proof  of the other assertions is divided into two steps.
		
		\textit{Step 1: Existence.} 	
		Let $\s=(\Omega, \mathcal{F},\p,\{\mathcal{F}_t\}_{t\geq0}, \W)$ be given and  let $X_\e$  be the global pathwise solution to \eqref{approximation cut problem}. We define sequences of measures
		$\nu_{\e^{(1)},\e^{(2)}}$ and $\mu_{\e^{(1)},\e^{(2)}}$ as
		\begin{align*}
		\nu_{\e^{(1)},\e^{(2)}}(\cdot)=\p\left\{(X_{\e^{(1)}},X_{\e^{(2)}})\in\cdot\right\}\ \mathrm{on}\ C([0,T];\Z)\times C([0,T];\Z),
		\end{align*}
		and
		\begin{align*}
		\mu_{\e^{(1)},\e^{(2)}}(\cdot)=\p\left\{(X_{\e^{(1)}},X_{\e^{(2)}},\W)\in\cdot\right\}\ \mathrm{on}\ C([0,T];\Z)\times C([0,T];\Z)\times C([0,T];\U_0).
		\end{align*}
		Let $\left\{\nu_{\e_k^{(1)},\e_k^{(2)}}\right\}_{k\in\N}$ be an arbitrary subsequence of $\left\{\nu_{\e^{(1)},\e^{(2)}}\right\}_{n\in\N}$  such that $\e_k^{(1)},\e_k^{(2)}\rightarrow0$ as $k\rightarrow \infty$.
		With minor modifications in the proof of Lemma \ref{Tightness}, the tightness of $\left\{\nu_{\e_k^{(1)},\e_k^{(2)}}\right\}_{k\in\N}$ can be  obtained. Similar to Lemma \ref{convergence}, one can find a probability space $\left(\widetilde{\Omega}, \widetilde{\mathcal{F}},\widetilde{\p}\right)$ on which there is a sequence of random variables $\left(\underline{X_{\e^{(1)}_k}},\overline{X_{\e^{(2)}_k}},\widetilde{\W_{k}}\right)$ and a random variable $\left(\underline{X},\overline{X},\widetilde{\W}\right)$ such that
		$$\left(\underline{X_{\e^{(1)}_k}},\overline{X_{\e^{(2)}_k}},\widetilde{\W_{k}}\right)\xrightarrow[k\rightarrow\infty]{} \left(\underline{X},\overline{X},\widetilde{\W}\right) \ \mathrm{in}\ C([0,T];\Z)\times C([0,T];\Z)\times C([0,T];\U_0)\ \ \widetilde{\p}-a.s.$$
		Then, $\nu_{\e^{(1)},\e^{(2)}}$ converges weakly to a measure $\nu$ on $C([0,T];\Z)\times C([0,T];\Z)$ defined by
		$$\nu(\cdot)
		=\widetilde{\p}\left\{\left(\underline{X},\overline{X}\right)\in\cdot\right\}.$$
		Going along the lines as in Section \ref{prove M solution}, we see that both $\left(\widetilde{\s},\underline{X},T\right)$ and $\left(\widetilde{\s},\overline{X},T\right)$ are martingale solutions to \eqref{cut-off problem} such that $\underline{X},\overline{X}\in L^2\left(\widetilde{\Omega};L^\infty(0,T;\X)\cap C([0,T];\Z)\right)$. Moreover, since $X_\e(0)\equiv X_0$ for all $n$, we have that $\underline{X}(0)=\overline{X}(0)$ almost surely in $\widetilde{\Omega}$. Then we use Lemma \ref{pathwise uniqueness cut-off} to see $$\nu\left(
		\left\{\left(\underline{X},\overline{X}\right)\in
		C([0,T];\Z)\times C([0,T];\Z),\underline{X}=\overline{X}\right\}
		\right)
		=1.$$
		Lemma \ref{Gyongy convergence} implies that the original sequence $\{X_\e\}$ defined on the initial probability space $(\Omega, \mathcal{F},\p)$ has a subsequence (still labeled in the same way) satisfying
		\begin{equation}
		X_\e\rightarrow X\ {\rm in}\ C\left([0,T];\Z\right) \label{convergence Xn}
		\end{equation}
		for some $X$ in $C([0,T];\Z)$. Similar to \eqref{solution M estimate}, we have
		\begin{align}\label{solution P estimate}
		\E\sup_{t\in[0,T]} \|X\|_{\X}^2
		\leq\,&
		\liminf_{m\rightarrow \infty}
		\E\sup_{t\in[0,T]} \|\pi_mX\|_{\X}^2\notag\\
		\leq\,&
		\liminf_{m\rightarrow \infty}
		\liminf_{\e\rightarrow 0}
		\E\sup_{t\in[0,T]} \|\pi_m X_\e\|_{\X}^2\notag\\
		\leq\,& \liminf_{m\rightarrow \infty}
		\liminf_{\e\rightarrow 0} \E\sup_{t\in[0,T]} \|X_\e\|_{\X}^2
		<C(R,X_0,T).
		\end{align} 
		Therefore $X\in L^2\left(\Omega;L^\infty(0,T;\X)\cap C([0,T];\Z)\right)$.
		Since for each $n$, $X_\e$ is $\{\mathcal{F}_t\}_{t\geq0}$ progressive measurable, so is $X$. 
		Using \eqref{convergence Xn} and the embedding $\Z\hookrightarrow\V$, we obtain a global pathwise solution to \eqref{cut-off problem}.

		\textit{Step 2: Time continuity.}  As	$X\in L^2\left(\Omega;L^\infty(0,T;\X)\cap C([0,T];\Z)\right)$, now we only need to prove that $X(t)$ is continuous in $\X$. Since $\X\hookrightarrow\Z$ is dense, we see that $X$ is weakly continuous in $\X$ (cf. \cite[page 263, Lemma 1.4]{Temam-1977-book}). It suffices to  prove the continuity of $[0,T]\ni t\mapsto \|X(t)\|_\X$. 
		The difficulty here is that the problem \eqref{Abstact irregular SDE}
		is singular,  i.e., $g(t,X)$ is only a $\Z$-valued process and $h(t,X)$ is only an $\LL_2(\U;\Y)$-valued process, hence the  products $\left(g(t,X),X\right)_{\X}$ and $\left(h(t,X)e_i,X\right)_{\X}$ might not exist %Here $\{e_i\}_{i\in\N}$ is a complete orthonormal basis of $\U$. 
		and the  classical It\^{o} formula in the Hilbert space $\X$ (see \cite[Theorem 4.32]{Prato-Zabczyk-2014-Cambridge} or \cite[Theorem 2.10]{Gawarecki-Mandrekar-2010-Springer}) can not be used directly here. 
		At this point  the regularization operator $T_\e$ from \ref{continuity requirement}  is invoked to consider the 
		It\^{o} formula for $\|T_\e  X\|^2_{\X}$ instead.  Then we have
		\begin{align}
		{\rm d}\|T_\e X\|^2_{\X}
		=\,&2\chi_R\big(\|X\|_{\V}\big)\left(T_\e h(t,X)\,{\rm d}\W,T_\e X\right)_{\X}
		+2\chi^2_R\big(\|X\|_{\V}\big)\left(T_\e b(t,X),T_\e X\right)_{\X}\,{\rm d}t\nonumber\\
		&+2\chi^2_R\big(\|X\|_{\V}\big)\left(T_\e g(t,X),T_\e X\right)_{\X}\,{\rm d}t+
		\chi^2_R\big(\|X\|_{\V}\big)\|T_\e h(t,X)\|_{\LL_2(\U;\X)}^2\,{\rm d}t.\label{TnX 2}
		\end{align}
		By \eqref{solution P estimate}, 
		\begin{equation}\label{stopping time X}
		\tau_N=\inf\{t\ge0:\|X\|_{\X}>N\}\rightarrow \infty\ \text{as}\ N\rightarrow\infty\ \ \p-a.s.
		\end{equation}
		Thus, we only need to prove the continuity up to time $\tau_N\wedge T$ for each $N\ge 1$. 
		Using \ref{continuity requirement}, \ref{b in X} and the bound
		$\chi_R\big(\|X\|_{\V}\big)\leq 1$, we have for $[t_2,t_1]\subset[0,T]$ with $t_1-t_2<1$ the estimate
		\begin{align*}
		\E\left[\left(\|T_\e  X(t_1\wedge\tau_N)\|^2_{\X}-\|T_\e  X(t_2\wedge\tau_N)\|^2_{\X}\right)^4\right]
		\leq\,& C(N,T)|t_1-t_2|^{2}.
		\end{align*}
		Using Fatou's lemma, we arrive at
		\begin{align*}
		\E\left[\left(\|X(t_1\wedge\tau_N)\|^2_{\X}-\|X(t_2\wedge\tau_N)\|^2_{\X}\right)^4\right]\leq C(N,T)|t_1-t_2|^{2},
		\end{align*}
		which together with Kolmogorov's continuity theorem ensures the continuity of $t\mapsto\|X(t\wedge\tau_N)\|_{\X}$.
	\end{proof}

	With Lemma \ref{global pathwise solution cut-off} at hand, we are in the position to finish the proof of \ref{Abstract irregular SDE: P} in Theorem \ref{Abstract irregular SDE: T}.
	
	\subsubsection{Concluding the proof of  \ref{Abstract irregular SDE: P} in Theorem \ref{Abstract irregular SDE: T}}\label{prove P solution}

	%
	%we only need to prove that the solution is actually continuous in $t$ in $\X$, remove the cut-off function in \eqref{cut-off problem} and then  obtain the blow-up criterion. 

	Similar to Lemma \ref{pathwise uniqueness original}, for $X_0(\omega,x)\in L^2(\Omega; \X)$, we let 
	$$
	\Omega_k=\{k-1\leq\|X_0\|_{\X}<k\},\ \ k\geq1.
	$$
	Since $\E\|\X_0\|^2_{\X}<\infty$, we have $1=\sum_{k\geq1}\textbf{1}_{\Omega_k}$ $\p-a.s.$, which means that
	\begin{align*}
	X_0(\omega)&
	=\sum_{k\geq1}X_0(\omega,x)\textbf{1}_{k-1\leq\|X_0\|_{\X}<k}:=\sum_{k\geq1}X_{0,k}(\omega)\ \  \p-a.s.
	\end{align*}
	On account of Lemma \ref{global pathwise solution cut-off}, we let $X_{k,R}$ be the global  pathwise solution to the cut-off problem \eqref{cut-off problem} with initial value $X_{0,k}$ and cut-off function $\chi_R(\cdot)$. Define
	\begin{equation}\label{Remove tau k}
	\tau_{k,R}=\inf\left\{t>0:\sup_{t'\in[0,t]}\|X_{k,R}(t')\|^2_{\X}>\|X_{0,k}\|^2_{\X}+2\right\}.
	\end{equation}
	Since $X_{k,R}$ is continuous in time (cf. Lemma \ref{global pathwise solution cut-off}), 
	for any $R>0$, we have $\p\{\tau_{k,R}>0,\ \forall k\ge1\}=1$. Now we let $R=R_k$ be discrete and then denote $(X_k,\tau_k)=(X_{k,R_k},\tau_{k,R_k})$. If $R^2_k>k^2+2$, then 
	$\p\{\tau_k>0,\ \forall k\ge1\}=1$
	and
	$$\p\left\{\|X_k\|^2_{\V}\leq\|X_k\|^2_{\X}\leq\|X_{0,k}\|^2_{\X}+2<R^2_k,\ \forall t\in[0,\tau_{k}],\ \forall k\ge1\right\}=1,$$
	which means $$\p\left\{\chi_{R_k}(\|X_k\|_{\V})=1,\ \forall t\in[0,\tau_k],\ \forall k\ge1\right\}=1.$$
	Therefore $(X_k,\tau_k)$ is the pathwise solution  to \eqref{Abstact irregular SDE} with initial value $X_{0,k}$.
	%Because $h(t,0)=0$, we have  $h(t,\textbf{1}_{\Omega_k},X_k)=\textbf{1}_{\Omega_k}h(t,X_k)$. On account of $\Omega_k\bigcap\Omega_{k'}=\emptyset$ with $k\neq k'$, we find
	%$$\sum_{k\geq1}\textbf{1}_{\Omega_k}h(t,X_k)
	%=\sum_{k\geq1}h(t,\textbf{1}_{\Omega_k},X_k)
	%=h\left(t,\sum_{k\geq1}\textbf{1}_{\Omega_k},X_k\right),\ t\geq0.$$
	%Similarly, $b$ and $g$ also satisfy the above property. 
	As has been shown in Lemma \ref{pathwise uniqueness original}, $\textbf{1}_{\Omega_k}X_{k}$ also solves   \eqref{Abstact irregular SDE} with initial value $X_{0,k}$ on $[0,\textbf{1}_{\Omega_k}\tau_{k}]$. Then uniqueness means $X_k=\textbf{1}_{\Omega_k}X_{k}$ on $[0,\textbf{1}_{\Omega_k}\tau_{k}]$  $\p-a.s.$ Therefore we infer from
	$\p\{\bigcup_{k\geq1}\Omega_k\}=1$ that
	the pair
	\begin{equation*}
	\bigg(X=\sum_{k\geq1}\textbf{1}_{\Omega_k}X_k,\ \
	\tau=\sum_{k\geq1}\textbf{1}_{\Omega_k}\tau_k\bigg)
	\end{equation*}
	is a pathwise solution to \eqref{Abstact irregular SDE} corresponding to the initial condition $X_0$.  Since for each $k$, $X_k$ is continuous in time (cf. Lemma \ref{global pathwise solution cut-off}), 
	so is $X$.  Then we have
	\begin{align*}
	\sup_{t\in[0,\tau]}\|X\|_{H^s}^2
	=\sum_{k\geq1}\textbf{1}_{k-1\leq\|X_0\|_{\X}<k}\sup_{t\in[0,\tau_k]}\|X_k\|_{H^s}^2
	\leq\,&\sum_{k\geq1}\textbf{1}_{k-1\leq\|X_0\|_{\X}<k}\left( \|X_{0,k}\|^2_{H^s}+2\right)
	\leq2\|X_{0}\|^2_{H^s}+4.
	\end{align*}
	Taking expectation gives rise to \eqref{L2 moment bound} and we have finished the proof of \ref{Abstract irregular SDE: P} in Theorem \ref{Abstract irregular SDE: T}.
	
	%Besides, using \eqref{Remove tau k}, we have
	%\begin{align*}
	%\sup_{t\in[0,\tau]}\|X\|_{H^s}^2
	%=\sum_{k\geq1}\textbf{1}_{k-1\leq\|X_0\|_{\X}<k}\sup_{t\in[0,\tau_k]}\|X_k\|_{H^s}^2
	%\leq\,&\sum_{k\geq1}\textbf{1}_{k-1\leq\|X_0\|_{\X}<k}\left( \|X_{0,k}\|^2_{H^s}+2\right)
	%\leq2\|X_{0}\|^2_{H^s}+4.
	%\end{align*}
	%Taking expectation gives rise to \eqref{L2 moment bound}. Uniqueness has been obtained in  Lemma \ref{pathwise uniqueness original}. And, the passage, from $(X,\tau)$ to a maximal pathwise solution in the sense of Definition \ref{pathwise solution definition}, may be carried out as in  \cite{Crisan-Flandoli-Holm-2018-JNS,GlattHoltz-Vicol-2014-AP,GlattHoltz-Ziane-2009-ADE,Breit-Feireisl-Hofmanova-2018-Book}.  We omit the details for brevity.	
	
	\subsection{Proof of  \ref{Abstract irregular SDE: P:time and blow-up} in Theorem \ref{Abstract irregular SDE: T}}
	
	To complete the proof of Theorem \ref{Abstract irregular SDE: T}, it suffices to prove the blow-up criterion 
	\eqref{Blow-up criterion} when additionally\ref{blow-up requirement} holds true.  To show it, we define
	\begin{align*}
	\tau_{1,m}:=\inf\left\{t\geq0: \|X(t)\|_{\X}\geq m\right\},\ \ \
	\tau_{2,l}:=\inf\left\{t\geq0: \|X(t)\|_{\V}\geq l\right\},\ \ m,l\in\N,
	\end{align*}
	where $\inf\emptyset=\infty$.
	Denote  $\displaystyle\tau_1=\lim_{m\rightarrow\infty}\tau_{1,m}$ and $\displaystyle\tau_2=\lim_{l\rightarrow\infty}\tau_{2,l}$. Then,  \eqref{Blow-up criterion} is just a direct consequence  of the  statement
	\begin{align}\label{tau1=tau2}
	\tau_{1}=\tau_{2}\ \ \p-a.s.
	\end{align}  
	Hence it suffices to prove \eqref{tau1=tau2}. Because $\X\hookrightarrow\V$,  it is obvious that 
	$\tau_{1}\leq \tau_2$ $\p-a.s.$
	Therefore, the proof reduces further  to  checking only  $\tau_{1}\geq \tau_2$ $\p-a.s.$
	We first notice that for all $M,l\in\N$,
	\begin{align*}
	\left\{\sup_{t\in[0,\tau_{2,l}\wedge M]}\|X(t)\|_{\X}<\infty\right\}
	=\bigcup_{m\in\N}\left\{\sup_{t\in[0,\tau_{2,l}\wedge M]}\|X(t)\|_{\X}<m\right\}
	\subset\bigcup_{m\in\N}\left\{\tau_{2,l}\wedge M\leq\tau_{1,m}\right\}.
	\end{align*}
	Because $$\bigcup_{m\in\N}\left\{\tau_{2,l}\wedge M\leq\tau_{1,m}\right\}\subset\left\{\tau_{2,l}\wedge M\leq\tau_{1}\right\},$$
	as long as
	\begin{align}
	\p\left( \sup_{t\in[0,\tau_{2,l}\wedge M]}\|X(t)\|_{\X}<\infty\right)=1,
	\ \ \forall\ M,l\in\N,\label{tau2<tau1 condition}
	\end{align}
	we have $\p\left(\tau_{2,l}\wedge M\leq\tau_{1}\right) =1$  for all $M,l\in\N$, and
	\begin{align*}
	\p\left(\tau_2\leq\tau_1\right)
	=\p\left(\bigcap_{l\in\N}\left\{\tau_{2,l}\leq \tau_{1}\right\}\right)
	=\p\left(\bigcap_{M,l\in\N}\left\{\tau_{2,l}\wedge M\leq \tau_{1}\right\}\right)=1.\label{tau2<tau1}
	\end{align*}
	As a result, it remains to prove \eqref{tau2<tau1 condition}.  However,  as mentioned before, we can not directly apply the It\^{o} formula to $\|X\|^2_{\X}$ to get control of $\E\|X(t)\|_{\X}^2$. As in \eqref{TnX 2}, but now  with $Q_\e $, we use It\^{o} formula for  $\|Q_\e X\|^2_{\X}$, apply  the BDG inequality, 
	\ref{b in X} and \ref{blow-up requirement} to find constants $C_1>0$ and $C_2=C_2(l)>0$ such that
	\begin{align*}
	&\hspace*{-1cm}\E\sup_{t\in[0,\tau_{2,l}\wedge M]}\|Q_\e X\|^2_{\X}-\E\|Q_\e X_0\|^2_{\X}\\
	\leq\,&C_1 \E\left(\int_0^{\tau_{2,l}\wedge M}k(t)
	f\big(\|X\|_{\V}\big)\|X\|^2_{\X}\|Q_\e X\|^2_{\X}\dd t\right)^{\frac12}+C_1\E\int_0^{\tau_{2,l}\wedge M}k(t)f\big(\|X\|_{\V}\big)\|X\|^2_{\X}\,{\rm d}t\nonumber\\
	\leq\,&C_{2}\E\left(\sup_{t\in[0,\tau_{2,l}\wedge M]}\|Q_\e X\|_{\X}^{2}\int_0^{\tau_{2,l}\wedge M}k(t)
	\|X\|^2_{\X}\dd t\right)^{\frac12}+C_{2}\E\int_0^{\tau_{2,l}\wedge M}k(t)\|X\|_{\X}^{2}\,{\rm d}t\notag\\
	\leq\,&\frac{1}{2}\E\sup_{t\in[0,\tau_{2,l}\wedge M]}\|Q_\e X\|_{\X}^{2}
	+C_{2}\int_0^{M}k(t)\E\sup_{t'\in[0,t\wedge \tau_{2,l}]}\|X(t')\|_{\X}^{2}\,{\rm d}t.
	\end{align*} 
	%where $C_2$ depends on $l$ through $f\big(\|X\|_{\V}\big)$ and the definition of $\tau_{2,l}$.
	%(cf. \ref{b in X} and Assumption \ref{Assum-C}). 
	This, together with \ref{blow-up requirement}, yields
	\begin{align*}
	\E\sup_{t\in[0,\tau_{2,l}\wedge M]}\|Q_\e X\|^2_{\X}\leq
	2\E\|X_0\|^2_{\X}
	+C_{2}\int_0^{M}k(t)\E\sup_{t'\in[0,t\wedge \tau_{2,l}]}\|X(t')\|_{\X}^{2}\,{\rm d}t.
	\end{align*}

	Since the right hand side of the inequality above does not depend on $\e$, and since $Q_\e $ satisfies \eqref{B21}, we can send $\e\rightarrow0 $ to find
	\begin{align*}
	\E\sup_{t\in[0,\tau_{2,l}\wedge M]}\|X\|^2_{\X}\leq
	2\E\|X_0\|^2_{\X}
	+C_{2}\int_0^{M}k(t)\E\sup_{t'\in[0,t\wedge \tau_{2,l}]}\|X(t')\|_{\X}^{2}\,{\rm d}t.
	\end{align*} 
	Then Gr\"{o}nwall's inequality  shows that for each $l,M\in\N$,
	$$\E\sup_{t\in[0,\tau_{2,l}\wedge M]}\|X(t)\|^2_{\X}
	\leq 2\E\|X_0\|^2_{\X}\exp\left\{C_2\int_0^Mk(t)\,{\rm d}t\right\}<\infty,$$ 
	which gives \eqref{tau2<tau1 condition}.  We conclude the proof of \ref{Abstract irregular SDE: P:time and blow-up} in Theorem \ref{Abstract irregular SDE: T}.

	\section{Applications to nonlinear ideal fluid models with transport noise}\label{sec:application}

	\subsection{Stochastic advection by Lie transport  in  fluid dynamics}\label{Application to SALT}
	Starting with the pioneering works  \cite{Fedrizzi-Flandoli-2013-JFA,Flandoli-Gubinelli-Priola-2010-Invention} for  linear scalar transport equations, many achievements have been made in recent years  for stochastic fluid equations with noise of \textit{transport type}.  Transport-type noise refers to noise  depending linearly on the gradient of the solution. In \cite{Holm-2015-ProcA},  stochastic equations governing the dynamics of some ideal fluid regimes have been derived by employing a novel variational principle for stochastic Lagrangian particle dynamics. Later, the same stochastic evolution equations were rediscovered in \cite{Cotter-etal-2017} using a multi-scale decomposition of the deterministic Lagrangian flow map into a slow large-scale mean, and a rapidly fluctuating small-scale map. 
	In \cite{Holm-2015-ProcA}, the extension of geometric mechanics to include stochasticity in nonlinear fluid theories was accomplished by using Hamilton's variational principle. This extension motivates us to study  stochastic Lagrangian fluid trajectories, denoted as $X_{t}(x, t)$,  arising from the stochastic Eulerian vector field with a noise in the Stratonovich  sense, i.e.,
	\begin{equation}
	\dd X_{t}(x, t) := u(x, t) dt + \displaystyle\sum_{k=1}^{M}\xi_{k}(x)\circ \dd W_{k}.
	\label{SALT}
	\end{equation}
	In \eqref{SALT} $u(x,t)$ means the drift velocity, $\{W_k=W_k(t)\}_{k=1,2,\cdots,M}$ is a family of standard 1-D independent Brownian motions, and $M$ can be determined via the amount of variance required from a principal component analysis, or via empirical orthogonal function analysis.
	
	Deriving continuum-scale  equations  taking into account noise as in \eqref{SALT} is known as  the Stochastic Advection by Lie Transport (SALT) 
	approach, see \cite{Cotter-etal-2019} and the references therein. The SALT approach combines stochasticity in the velocity of the fluid material loop in Kelvin's circulation theorem with ensemble forecasting and meets the important challenge of incorporating stochastic parameterisation at the fundamental level, see for example \cite{Berner-Jung-Palmer-12,Leslie-Quarini-79,Zidikheri-F-10}.

	Many subsequent investigations of the properties of the equations of fluid dynamics with the SALT modification have appeared in the literature recently. For example, local existence in Sobolev spaces and a Beale-Kato-Majda type blow-up criterion were derived in \cite{Crisan-Flandoli-Holm-2018-JNS,Flandoli-Luo-2019} for the incompressible 3-D SALT Euler equations. For the 2-D version, global existence of solutions has been shown in \cite{Crisan-Lang-2019}. In \cite{Alonso-Bethencourt-2020-JNS}, the authors provide a local existence result for the incompressible 2-D SALT Boussinesq equations. For a simpler but still nonlinear equation as the SALT Burgers equation, we refer to \cite{Alonso-etal-2019-NODEA,Flandoli-2011-book}.

	\subsubsection{The two-component CH system with transport noise}
	The  Camassa-Holm (CH) equation
	\begin{equation}\label{CH:equation}
	u_{t}-u_{xxt}+3uu_{x}=2u_{x}u_{xx}+uu_{xxx}
	\end{equation}
	was proposed independently by Fokas and Fuchssteiner in \cite{Fuchssteiner-Fokas-1981-PhyD} and by Camassa and Holm in \cite{Camassa-Holm-1993-PRL}. 
	In \cite{Fuchssteiner-Fokas-1981-PhyD}, it was proposed to consider some completely integrable generalizations of the Korteweg-de-Vries equation with bi-Hamiltonian structures, and in \cite{Camassa-Holm-1993-PRL}, it was derived to describe the unidirectional propagation of shallow water waves over a flat bottom. Solutions of 
	equation \eqref{CH:equation} exhibit %both phenomena of (peaked) soliton interaction and
	the  wave-breaking phenomenon, i.e., smooth global existence may fail \cite{Constantin-Escher-1998,Constantin-Escher-1998-2}.
	Global conservative solutions to the CH equation \eqref{CH:equation} were obtained in \cite{Bresan-Constantin-2007,Holden-Raynaud-2007}. 
	Different stochastic versions of the CH equation have been studied including additive noise \cite{Chen-Gao-Guo-2012-JDE} and multiplicative noise \cite{Albeverio-etal-2019-Arxiv,Rohde-Tang-2020-Arxiv,Tang-2018-SIMA,Tang-2020-Arxiv}. Following the approach in \cite{Holm-2015-ProcA}, the corresponding stochastic version of the CH equation with transport noise was introduced in \cite{Bendall-Cotter-Holm-2019,Crisan-Holm-2018}. Transforming the equation into
	a partial differential equation with random coefficients, the well-posedness of the stochastic CH equation with some special transport noise has been studied in \cite{Albeverio-etal-2019-Arxiv}.     We can extend this result to a far  more complex system:  the  stochastic two-component CH system  which has been  derived in \cite{Holm-Erwin-2019-Arxiv}, i.e., 
	\begin{equation}\label{SCH2-S-integral}
	\left\{\begin{aligned}
	&\dd m+(m\partial_{x}+\partial_{x} m) \dd\chi_{t}+\eta\partial_{x}\eta \dd t=0,\\
	&\dd\eta + (\eta \dd\chi_{t})_{x}=0,\\ 
	&m=u-u_{xx}.
	\end{aligned} \right.
	\end{equation} 
	In \eqref{SCH2-S-integral}
	$u$ is the fluid velocity and $\eta$ denotes the depth of the flow. As in \eqref{SALT}, the noise structure  in \eqref{SCH2-S-integral} is
	$$\dd\chi_{t} =u(t,x) \dd t +\sum_{k=1}^{M} \xi_{k}(x) \circ \dd W_k.$$
	The functions $\xi_1,\ldots, \xi_M$  represent spatial velocity-velocity correlations up to order $M$.\\ 
	Note that the   system  \eqref{SCH2-S-integral} reduces to the scalar CH equation from \cite{Albeverio-etal-2019-Arxiv} if we put $\eta$ to be zero.
	Here we consider $M=\infty$ and  rewrite \eqref{SCH2-S-integral} as
	\begin{equation}\label{SCH2 transport noise}
	\left\{\begin{aligned}
	&\dd m+\left[(mu)_{x} +\eta\eta_x\right] \dd t + \displaystyle\sum_{k=1}^{\infty} \mathcal{L}_{\xi_{k}}m\circ \dd W_k = 0, \\
	&\dd\eta + (\eta u)_{x} \dd t + \displaystyle\sum_{k=1}^{\infty} \mathcal{L}_{\xi_{k}}\eta\circ \dd W_k = 0.
	\end{aligned} \right.
	\end{equation}
	The differential operator $\mathcal{L}_{\xi}$  is given by
	\begin{equation}\label{define Lxi}
	\mathcal{L}_{\xi_k}= \partial_{x}\xi_k + \xi_k\partial_{x}.
	\end{equation}  
	We use the notation $\mathcal{L}_{\xi_k}$ since it coincides with the Lie derivative operator acting on one-forms. However, our analysis is valid for  general  linear differential operators with suitable coefficients.\\
	Calculating the cross-variation term in  the  general transformation formula
	$$\int_0^t f \circ d W   = \int_0^t f \dd W+ \frac{1}{2} \left\langle f, W\right\rangle_t, $$
	%	where $\left\langle \cdot, \cdot\right\rangle_t$ represents the cross-variation between two stochastic processes,
	we obtain  the corresponding It\^o formulation of \eqref{SCH2 transport noise}, given by 
	\begin{equation}\label{SCH2 transport noise 2}
	\left\{\begin{aligned}
	&\dd m+\left[(mu)_{x} +\eta\eta_x\right] \dd t -\frac{1}{2}\displaystyle\sum_{k=1}^{\infty}\mathcal{L}^2_{\xi_{k}}m \dd{t}  = -\displaystyle\sum_{k=1}^{\infty}\mathcal{L}_{\xi_{k}}m  \dd W_k, \\
	&\dd\eta + (\eta u)_{x} \dd t -\frac{1}{2}\displaystyle\sum_{k=1}^{\infty} \mathcal{L}^2_{\xi_{k}}\eta \dd{t} = -\displaystyle\sum_{k=1}^{\infty}\mathcal{L}_{\xi_{k}}\eta  \dd W_k.
	\end{aligned} \right.
	\end{equation}
	Note that the   operator $\mathcal{L}^2_{\xi_{k}}$  in \eqref{SCH2 transport noise 2} is the  second-order operator 
	$$\mathcal{L}^2_{\xi_{k}} f=\mathcal{L}_{\xi_{k}}(\mathcal{L}_{\xi_{k}} f)=  \xi_k^2\partial^2_{xx}f+ 3\xi_{k}\partial_{x}\xi_k\partial_{x}f + (\xi_k\partial^2_{xx}\xi_{_k}+(\partial_x\xi_k)^2)f. $$
	In this paper, we will consider \eqref{SCH2 transport noise 2} on the periodic torus $\mathbb T=\R/2\pi\mathbb Z$ in terms of the unknowns $(u,\eta)$.
	Therefore, for any real number $s$, we define $D^s=(I-\Delta)^{s/2}$
	as
	$\widehat{D^sf}(k)=(1+|k|^2)^{s/2}\widehat{f}(k)$. Then we apply $(1-\partial^{2}_{xx})^{-1}=D^{-2}$ to \eqref{SCH2 transport noise 2} and consider 
	for  $(u,\eta)$ the nonlocal Cauchy problem 
	\begin{equation}\label{SCH2 Ito form}
	\left\{\begin{aligned}
	&\dd u
	+\left[uu_x+\partial_xD^{-2}\left(\frac{1}{2}u^2+u_x^2+\frac{1}{2}\eta^2\right)-\frac{1}{2}D^{-2}\sum_{k=1}^{\infty}\mathcal{L}^2_{\xi_{k}}D^2u \right] \dd t=-D^{-2}\sum_{k=1}^{\infty}\mathcal{L}_{\xi_{k}}D^2u \dd W_k, \\
	&\dd\eta + \left(u\eta_x+\eta u_x\right)\ \dd t 
	- \frac{1}{2}\sum_{k=1}^{\infty} \mathcal{L}^2_{\xi_{k}}\eta\ \dd{t} = -\sum_{k=1}^{\infty}\mathcal{L}_{\xi_{k}}\eta  \dd W_k,\\
	&(u(0),\eta(0))=(u_0,\eta_0).
	\end{aligned} \right.
	\end{equation}
	Here we remark that in \eqref{SCH2 Ito form}, $f=D^{-2}g=(I-\partial^2_{xx})^{-1}g$ means $f=\mathcal{G}\star g$, where $\mathcal{G}$ is the Green function of the Helmholtz operator $(I-\partial^2_{xx})$ and $\star$ stands for the convolution. The local theory for \eqref{SCH2 Ito form} is stated in Theorem \ref{SCH2 results} below.
	
	\subsubsection{The CCF model with transport noise}
	As the second application of the abstract framework, we will consider a stochastic transport equation with non-local velocity on the periodic torus $\T$. In the deterministic case,  it reads
	\begin{equation}\label{CCF:model}
	\theta_{t}+(\mathcal{H}\theta) \theta_{x}=0,
	\end{equation}
	where
	$\mathcal{H}$ is the periodic Hilbert transform defined by
	\begin{equation}\label{define Hilbert transform}
	(\mathcal{H}f)(x)=\frac{1}{2\pi}\ {\rm p.v.}\int_0^{2\pi}f(t)\cot\left(\frac{x-t}{2}\right)\dd t.
	\end{equation}  
	Equation \eqref{CCF:model} was proposed by C\'ordoba, C\'ordoba and Fontelos in \cite{Cordoba-etal-2005-Annals} to consider advective transport  with 
	non-local velocity.
	%Now it is known as the C\'ordoba-C\'ordoba-Fontelos model 
	It is deeply connected to the 2-D SQG equation and hence with the 3-D Euler equations (cf.~\cite{Bae-Granero-2015-AM} and the references therein). Notice that, if we replace the non-local Hilbert transform by the identity operator we recover the classical Burgers equation. 
	%The non-locality  makes the equation much more complicated but also a much better realistic model for the 3-D Euler equations, cf. \cite{Bertozzi-Majda-2002}). 
	In \cite{Cordoba-etal-2005-Annals},  the breakdown of classical solutions to \eqref{CCF:model} for a generic
	class of smooth initial data was discovered.\\ 
	%When dissipation is added into the equation \eqref{CCF:model}, the solution may be global or may blows up in finite time, depending on the strength of the %dissipation, see \cite{Li-Rodrigo-2008,Dong-2008-JFA}.
	To the best of our knowledge, the stochastic counterpart of the CCF model \eqref{CCF:model} has not been studied yet. In this paper, we will consider the stochastic CCF model with transport noise, i.e.,
	\begin{equation}\label{Transport:nonloca:eq}
	\dd\theta + \left(\mathcal{H}\theta\right)\partial_{x}\theta \ dt + \displaystyle\sum_{k=1}^{\infty} \mathcal{L}_{\xi_{k}}\theta\circ \dd W_k=0,
	\end{equation}
	where
	$\{W_k=W_k(t)\}_{k\in\N}$ is a sequence of standard 1-D independent Brownian motions and $\mathcal{L}_{\xi_{k}}$ is given  as in \eqref{define Lxi}. 
	Using the corresponding It\^{o} formulation, we  are led to the Cauchy problem
	\begin{equation}\label{Transport:nonloca:eq:ito}
	\left\{\begin{aligned}
	&\dd\theta + (\mathcal{H}\theta)\partial_{x}\theta \ dt -\frac{1}{2}\displaystyle\sum_{k=1}^{\infty} 
	\mathcal{L}^{2}_{\xi_{k}}\theta \ \dd t= -\displaystyle\sum_{k=1}^{\infty} 
	\mathcal{L}_{\xi_{k}}\theta \dd W_k,\\
	&\theta(0)=\theta_0.
	\end{aligned}\right.
	\end{equation}
	A local theory for \eqref{Transport:nonloca:eq:ito} is stated in Theorem \ref{transport:nonlocal results} below.
	
	\subsection{Notations, assumptions and main results}\label{sect:applications:main results}
	
	To state the main results for \eqref{SCH2 Ito form} and \eqref{Transport:nonloca:eq:ito}, we  introduce some function spaces.
	For   $d\in\N$ and $1\leq p<\infty$ we denote by  $L^p(\T^d;\R)$ the standard Lebesgue space of measurable $p$-integrable $\R$-valued functions with domain $\T^d=(\R/2\pi\mathbb Z)^d$ and by  $L^\infty(\T^d;\R)$ the space of essentially bounded functions. Particularly, $L^2(\T^d;\R)$ is equipped with the inner product 
	$
	(f,g)_{L^2}=\int_{\T^d}f\cdot\overline{g}\,{\rm d}x,
	$
	where $\overline{g}$ denotes the complex conjugate of $g$.  The Fourier
	transform and inverse Fourier transform of $f(x)\in L^2(\T^d;\R)$ are defined by
	$\widehat{f}(\xi)=\int_{\T^d}f(x){\rm e}^{-{\rm i}x\cdot \xi}\,{\rm d}x$ and 
	$f(x)=\frac{1}{(2\pi)^d}\sum_{k\in{\mathbb Z}^d}\widehat{f}(k){\rm e}^{{\rm i}x\cdot k}$, respectively. Recalling that for any $s\in\R$,  
	$\widehat{D^sf}(k)=(1+|k|^2)^{s/2}\widehat{f}(k)$, we define the Sobolev space $H^s$ on $\T^d$ with values in $\R$ as
	\begin{align*}
	H^s(\T^d;\R):=\left\{f\in L^2(\T^d;\R):\|f\|_{H^s(\T^d;\R)}^2
	=\sum_{k\in{\mathbb Z}^d}|\widehat{D^sf}(k)|^2<+\infty\right\}.
	\end{align*}
	For $u=(u_j)_{1\leq j\leq n}$: $\T^d\mapsto\R^n$, we define
	$  
	\|u\|^2_{H^s(\T^d;\R^n)}:=\sum_{j=1}^n\|u_j\|^2_{H^s(\T^d;\R)}.
	$
	For the sake of simplicity, we omit the parentheses in the above notations from now on if there is no ambiguity.
	Similarly, for two spaces $H^{s_1}$ and $H^{s_2}$ ($s_1,s_2>0$) and $(f,g)\in H^{s_1}\times H^{s_2}$,  we define $\|(f,g)\|^2_{H^{s_1}\times H^{s_2}}:=\|f\|^2_{H^{s_1}}+\|g\|^2_{H^{s_2}}$. The  commutator  for two operators $P,Q$ is denoted by
	$[P,Q]:= PQ-QP.$ 
	The space of linear operators from  $\U$ to some separable Hilbert space $\mathbb X$ is denoted by $\LL(\U;\mathbb X)$.

	To obtain a local theory for \eqref{SCH2 Ito form} and \eqref{Transport:nonloca:eq:ito}, we have to impose natural regularity assumptions on $\{\xi_{k}(x)\}_{k\in\N}$ to give a reasonable meaning to the stochastic integral and to show certain estimates. For this reason, we make the following assumption:
	\begin{Assumption}\label{Assum-xi}
		$\sum_{k\in\N}\|\xi_k\|_{H^s}<\infty$ for any $s\geq0$. 
		
	\end{Assumption}
	
	\begin{Remark}\label{Remark xi with f}
		It follows from Assumption \ref{Assum-xi} that
		there is a $C>0$ such that	for all $f\in H^{s+2}$ with $s>\frac{1}{2}$, we have
		\begin{equation*}
		\sum_{k=1}^\infty\norm{\mathcal{L}_{\xi_k}f}_{H^{s}} \leq C \norm{f}_{H^{s+1}}\ \ \text{and}\ \ \ 
		\sum_{k=1}^\infty\norm{\mathcal{L}^2_{\xi_k}f}_{H^{s}} \leq C \norm{f}_{H^{s+2}}.
		\end{equation*}
		Besides, we do not require that $\{\xi_k\}_{k\in\N}$ is an orthogonal system.
		%\footnote{what is the reason for this discussion?}  In this paper, we use  two auxiliary spaces $\U$ and $\U_0$ to describe the cylindrical Wiener %process $\W$, cf. \eqref{define W}.  In applications, we will show the summation of noise can be viewed as a stochastic integral for such $\W$ %involving a Hilbert-Schmidt operator, see Lemmas \ref{Lemma for g h} and \ref{Lemma for g h nlt} below.  This formulation covers the case that the %noise has only finite many terms,  i.e., $\xi_k=0$ for all $k>M$ for some $M\in\N$. We notice that in the finite many noises case, %$\{\xi_k\}_{k=1,2,\cdots,M}$ is not necessary  an orthogonal set.
	\end{Remark}

	The main results for \eqref{SCH2 Ito form} and \eqref{Transport:nonloca:eq:ito} are the following:
	\begin{Theorem}\label{SCH2 results}
		Let	$s>\frac{11}{2}$ and $\s=(\Omega, \mathcal{F},\p,\{\mathcal{F}_t\}_{t\geq0}, \W)$ be a  stochastic basis fixed in advance. Let Assumption \ref{Assum-xi} hold.  If
		$(u_0,\eta_0)\in L^2(\Omega;H^s\times H^{s-1})$ is an $\mathcal{F}_0$-measurable random variable, then \eqref{SCH2 Ito form} has a local unique pathwise solution $((u,\eta),\tau)$ such that 
		\begin{equation}\label{SCH2 solution bound}
		(u,\eta)(\cdot\wedge \tau)\in L^2\left(\Omega; C\left([0,\infty);H^s\times H^{s-1}\right)\right).
		\end{equation} 
		Moreover, the maximal solution $((u,\eta),\tau^*)$ to \eqref{SCH2 Ito form} satisfies
		\begin{equation*}
		\textbf{1}_{\left\{\limsup_{t\rightarrow \tau^*}\|(u,\eta)(t)\|_{H^s\times H^{s-1}}=\infty\right\}}=\textbf{1}_{\left\{\limsup_{t\rightarrow \tau^*}\|(u,\eta)(t)\|_{W^{1,\infty}\times W^{1,\infty}}=\infty\right\}}\ \ \p-a.s.
		\end{equation*}
	\end{Theorem}

	\begin{Theorem}\label{transport:nonlocal results}
		Let	$s>\frac{7}{2}$ and $\s=(\Omega, \mathcal{F},\p,\{\mathcal{F}_t\}_{t\geq0}, \W)$ be a  stochastic basis fixed in advance. Let Assumption \ref{Assum-xi} hold.  If
		$\theta_0\in L^2(\Omega;H^s)$ is an $\mathcal{F}_0$-measurable random variable, then \eqref{Transport:nonloca:eq:ito} has a local unique pathwise solution $(\theta,\tau)$ such that 
		\begin{equation}\label{Trans:nonlocal:bound}
		\theta(\cdot\wedge \tau)\in L^2\left(\Omega; C\left([0,\infty);H^s\right)\right).
		\end{equation} 
		Moreover, the maximal solution $(\theta,\tau^*)$ to \eqref{Transport:nonloca:eq:ito} satisfies 
		\begin{equation}\label{STNL blowup criterion}
		\textbf{1}_{\left\{\limsup_{t\rightarrow \tau^*}\|\theta(t)\|_{H^s}=\infty\right\}}
		=\textbf{1}_{\left\{\limsup_{t\rightarrow \tau^*}\|\theta_x(t)\|_{L^{\infty}}+\|(\mathcal{H}\theta_x)(t)\|_{L^{\infty}}=\infty\right\}}\ \ \p-a.s.
		\end{equation}
	\end{Theorem}
	
	\begin{Remark}
		We require $s>11/2$ in Theorem \ref{SCH2 results}. This is because, if $(u,\eta)\in H^s\times H^{s-1}$,  then  $\left(-\frac{1}{2}D^{-2}\sum_{k=1}^{\infty}\mathcal{L}^2_{\xi_{k}}D^2u,-\frac{1}{2}\sum_{k=1}^{\infty}\mathcal{L}^2_{\xi_{k}}\eta\right)\in H^{s-2}\times H^{s-3}$. 
		As one can see $\left|\left(uu_x,u\right)_{H^s}\right|+\left|\left(u\eta_x,\eta\right)_{H^{s-1}}\right|\lesssim\, \|(u,\eta)\|_{W^{1,\infty}\times W^{1,\infty}}\|(u,\eta)\|^2_{H^s\times H^{s-1}}$. To apply Theorem \ref{Abstract irregular SDE: T} to \eqref{SCH2 Ito form} with $\X=H^s\times H^{s-1}$, we have to  verify \eqref{B12} with using Lemma \ref{Liecancellations}.  Therefore $s-4>\frac{3}{2}$, which means $s>11/2$. Similarly, $s>7/2$ is needed in Theorem \ref{transport:nonlocal results}.\\
		As mentioned before, the scalar stochastic CH equation with  transport noise  has been analyzed in  \cite{Albeverio-etal-2019-Arxiv} with a completely
		different approach. The authors obtain the local existence of pathwise solutions  in  a less regular space but without a blow-up criterion.  We note 
		that our approach can be also applied to this equation to give local existence, uniqueness and the blow-up criterion. 
	\end{Remark}
	
	\begin{Remark}\label{blow-up remark}
		Notice that in the deterministic case, one can use the estimate
		\begin{equation}\label{Dong estimate}
		\|\mathcal{H}\theta_x\|_{L^\infty}\lesssim\, 
		\left(1+\|\theta_x\|_{L^\infty}\log\left({\rm e}+\|\theta_x\|_{H^1}\right)+\|\theta_x\|_{L^2}\right)
		\end{equation}
		to improve  the blow-up criterion \eqref{STNL blowup criterion} into (cf.~\cite{Dong-2008-JFA})
		$$ \limsup_{t\rightarrow \tau^*}\|\theta(t)\|_{H^s}=\infty
		\Longleftrightarrow  \limsup_{t\rightarrow \tau^*}\|\theta_x(t)\|_{L^{\infty}}=\infty.$$
		To achieve this in the stochastic setting, we have an essential difficulty in closing the $H^s$-estimate. That is, one has to split the expectation $\E\|\mathcal{H}\theta_x\|_{L^\infty}\|\theta\|^2_{H^s}$. If we use \eqref{Dong estimate}, so far we have not known how to close the estimate for 
		$\E\|\theta\|^2_{H^s}$, where $\E\left[\left(1+\|\theta_x\|_{L^\infty}\log\left({\rm e}+\|\theta_x\|_{H^1}\right)+\|\theta_x\|_{L^2}\right)\|\theta\|_{H^s}\right]$ is involved.
	\end{Remark}

	\subsection{The stochastic two-component CH system: Proof of Theorem \ref{SCH2 results}}\label{Sec:SCH2}
	Now we consider \eqref{SCH2 Ito form} on the periodic torus $\mathbb T$, and  we will apply the abstract framework developed in Section \ref{sect: Abstact irregular SDE}  to obtain Theorem \ref{SCH2 results}. 
	To put \eqref{SCH2 Ito form} into the abstract framework, we define
	\begin{equation*}
	X=(u,\eta),\ G(u,\eta)=\partial_xD^{-2}\left(\frac{1}{2}u^2+u_x^2+\frac{1}{2}\eta^2\right),
	\end{equation*}
	and we set
	\begin{equation}\label{Define drift}
	%\left\{
	\begin{array}{rcccl}
	b(t,X) &=&b(X)&=&\left(-G(u,\eta),-\eta u_x\right),\\[1.0ex]
	g(t,X)&=&g(X)&=&  \left(-uu_x+\frac{1}{2}D^{-2}\sum_{k=1}^{\infty}\mathcal{L}^2_{\xi_{k}}D^2u,
	-u\eta_x+\frac{1}{2}\sum_{k=1}^{\infty} \mathcal{L}^2_{\xi_{k}}\eta\right),\\[2.0ex]
	h^k(t,X)&=&h^k(X)&=&\displaystyle \left(-D^{-2}\mathcal{L}_{\xi_{k}}D^2u,-\mathcal{L}_{\xi_{k}}\eta\right),\ \ k\in\N.
	\end{array} 
	%\right.
	\end{equation}
	%	where the operator $D^s$  is defined in Section \ref{section:notations and definitions}. 
	
	Now we recall that $\U$ is a fixed separable Hilbert space and  $\{e_i\}_{i\in\N}$ is a complete orthonormal basis of $\U$ such that the cylindrical Wiener process $\W$ is defined as in \eqref{define W}. Then we  define $h(X)\in \LL(\U;H^s\times H^{s-1})$ such that
	\begin{align}\label{Define diffusion}
	h(X)(e_k)=h^{k}(X)=\left(-D^{-2}\mathcal{L}_{\xi_{k}}D^2u,-\mathcal{L}_{\xi_{k}}\eta\right),\ \ k\in\N.
	\end{align}
	Altogether we can  rewrite the problem \eqref{SCH2 Ito form} as
	\begin{equation}\label{Abstarct SCH2}
	\left\{\begin{aligned}
	&\dd X=\left(b(X)+g(X)\right)\dd t+h(X)\dd\W,\\
	&X(0)=X_0=(u_0,\eta_0).
	\end{aligned}\right.
	\end{equation}
	
	In order to prove Theorem \ref{SCH2 results} by applying Theorem \ref{Abstract irregular SDE: T}, we  need to check that Assumption \ref{Assum-A}  is satisfied.  To ease notation, we 
	define  
	\begin{equation}\label{U Xs V SCH2}
	\X^s=H^s\times H^{s-1}
	\end{equation}
	and make the following choice for the spaces $\X \subset \Y \subset \Z$  and $\Z\subset\V$,
	\begin{equation}\label{XYZCS}
	\X=\X^s, \,  \Y=\X^{s-1}, \, \Z=\X^{s-2}, \, \V=W^{1,\infty}\times W^{1,\infty}.
	\end{equation}

	\subsubsection{Estimates on nonlinear terms\label{subsubest}}
	In this preparatory part,  some basic Sobolev estimates to deal with $b,g,h$ from \eqref{Define drift}, \eqref{Define diffusion} are introduced.

	\begin{Lemma}\label{Lemma for b}
		Let $s>5/2$. Then $b$ is regular in $\X$ and  for $X=(u,\eta)\in \X^s$, $Y=(v,\rho)\in \X^s$, we have  
		\begin{align*}
		\|b(X)\|_{\X^s} &\lesssim\, 
		\|X\|_{\V}\|X\|_{\X^s}. \\
		\|b(X)-b(Y)\|_{\X^s} &\lesssim\, \left(\|X\|_{\X^s}+\|Y\|_{\X^s}\right)\|X-Y\|_{\X^s}.
		\end{align*}
	\end{Lemma}
	\begin{proof}
		Since $\partial_x(1-\partial^2_{xx})^{-1}$ is a bounded map from $H^s$ to $H^{s+1}$, the first estimate follows from 
		\begin{align*}
		\|b(X)\|^2_{\X^s}
		= \,&\|G(u,\eta)\|^2_{H^s}+\|u_x\eta\|^2_{H^{s-1}}\\
		\lesssim\, &\|u^2+u^2_x+\eta\|^2_{H^{s-1}}+\|u_x\|^2_{L^{\infty}}\|\eta\|^2_{H^{s-1}}
		+\|u_x\|^2_{H^{s-1}}\|\eta\|^2_{L^{\infty}}\\
		\lesssim\, &\|u\|^2_{W^{1,\infty}}\|u\|^2_{H^{s}}+\|\eta\|^2_{L^{\infty}}\|\eta\|^2_{H^{s-1}}
		+
		\|u\|^2_{W^{1,\infty}}\|\eta\|^2_{H^{s-1}}
		+\|u\|^2_{H^{s}}\|\eta\|^2_{L^{\infty}}\\
		\lesssim\, &\|(u,\eta)\|^2_{W^{1,\infty}\times L^{\infty}}\|(u,\eta)\|^2_{H^s\times H^{s-1}}.
		\end{align*} 
		Using the fact that $H^{s-1}$ is an algebra, we can infer that
		\begin{align*}
		&\hspace*{-0.5cm}\|b(X)-b(Y)\|^2_{\X^s}\\
		\lesssim\,&
		\|G(u,\eta)-G(v,\rho)\|^2_{H^s}+\|u_x\eta-v_x\rho\|^2_{H^{s-1}}\\
		\lesssim\,& \|u^2-v^2+u_x^2-v_x^2+\eta^2-\rho^2\|^2_{H^{s-1}}
		+\|u_x(\eta-\rho)+\rho(u_x-v_x)\|^2_{H^{s-1}}\\
		\lesssim\,&  \|u+v\|_{H^s}^2\|u-v\|^2_{H^s}
		+\|\eta+\rho\|_{H^{s-1}}^2\|\eta-\rho\|^2_{H^{s-1}}
		+\|u\|^2_{H^s}\|\eta-\rho\|^2_{H^{s-1}}
		+\|\rho\|^2_{H^{s-1}}\|u-v\|^2_{H^s}\\
		\lesssim\,&  \left(\|(u,\eta)\|^2_{H^{s}\times H^{s-1}}+\|(v,\rho)\|^2_{H^{s}\times H^{s-1}}\right)\|(u-v,\eta-\rho)\|^2_{H^{s}\times H^{s-1}},
		\end{align*}
		which gives the second estimate.
	\end{proof}
	
	\begin{Lemma}\label{Lemma for g h}
		Let Assumption \ref{Assum-xi} hold true and  $s>7/2$. If $X=(u,\eta)\in \X^s$, then  $g:\X^s\rightarrow \X^{s-2}$ and $h:\X^s\rightarrow \LL_2(\U;\X^{s-1})$ obey
		\begin{align*}
		\|g(X)\|_{\X^{s-2}}&\lesssim\, 1+\|X\|^2_{\X^{s}} 
		\end{align*}
		and
		\begin{align*}
		\|h(X)\|_{\LL_{2}(\U;\X^{s-1})}&\lesssim\,\|X\|_{\X^{s}}.
		\end{align*}
	\end{Lemma}
	\begin{proof}
		Using  $H^{s-3}\hookrightarrow L^{\infty}$, we derive
		\begin{align*}
		\|g(X)\|^2_{\X^{s-2}}
		=\,&\left\|-uu_x+\frac{1}{2}D^{-2}\sum_{k=1}^{\infty}\mathcal{L}^2_{\xi_{k}}D^2u\right\|^2_{H^{s-2}}
		+\norm{-u\eta_x+\frac{1}{2}\displaystyle\sum_{k=1}^{\infty} \mathcal{L}^2_{\xi_{k}}\eta}^2_{H^{s-3}}\\
		\lesssim\,& \norm{u}^4_{H^s}+\norm{D^2u}^2_{H^{s-2}}+\norm{\eta}^2_{H^{s-1}}\norm{u}^2_{H^{s}}+\norm{\eta}^2_{H^{s-1}}\\
		\lesssim\,&
		\left(1+\norm{u}^2_{H^s}+\norm{\eta}^2_{H^{s-1}}\right)^2,
		\end{align*}
		which implies the first estimate. 
		Similarly, from the definition of $h$ in  \eqref{Define diffusion}, and the definition of $\LL_{\xi}$ in \eqref{define Lxi}, one has
		\begin{align*}
		\sum_{k=1}^{\infty}\|h(X)e_k\|^2_{\X^{s-1}}=\,&\sum_{k=1}^{\infty}
		\left(\norm{D^{-2} \mathcal{L}_{\xi_{k}}D^2u}^2_{H^{s-1}}
		+\norm{ \mathcal{L}_{\xi_{k}}\eta}^2_{H^{s-2}}\right)\\
		\lesssim\,&\sum_{k=1}^{\infty}
		\left(\norm{ \mathcal{L}_{\xi_{k}}D^2u}^2_{H^{s-3}}
		+\norm{ \mathcal{L}_{\xi_{k}}\eta}^2_{H^{s-2}}\right)\lesssim\, \|u\|^2_{H^s}+\|\eta\|^2_{H^{s-1}},
		\end{align*}
		which gives 
		the second estimate.
	\end{proof}

	\begin{Lemma}\label{Difference cancellation}
		Let $s>\frac{11}{2}$, $X=(u,\eta)\in \X^s$ and $Y=(v,\rho)\in \X^s$. Then we have
		\begin{align*}
		2\left(g(t,X)-g(t,Y),X-Y\right)_{\X^{s-2}}+\|h(t,X)-h(t,Y)\|^2_{\LL_2(\U;\X^{s-2})}\lesssim\,& \left(1+\|X\|^2_{\X^{s}}+\|Y\|^2_{\X^{s}}\right)\|X-Y\|^2_{\X^{s-2}}.
		\end{align*}
	\end{Lemma}

	\begin{proof}
		Recalling \eqref{Define drift} and \eqref{Define diffusion}, we have
		\begin{align*}
		&\hspace*{-0.5cm}2\left(g(X)-g(Y),X-Y\right)_{\X^{s-2}}
		+\|h(t,X)-h(t,Y)\|^2_{\LL_2(\U;\X^{s-2})}\\
		=\,& 2(vv_x-uu_x,u-v)_{H^{s-2}}+2(v\rho_x-u\eta_x,\eta-\rho)_{H^{s-3}}\\
		&+\left(D^{-2}\sum_{k=1}^{\infty}\mathcal{L}^2_{\xi_{k}}D^2(u-v),u-v\right)_{H^{s-2}}
		+\sum_{k=1}^{\infty}
		\left(D^{-2}\mathcal{L}_{\xi_{k}}D^2(u-v),D^{-2} \mathcal{L}_{\xi_{k}}D^2(u-v)\right)_{H^{s-2}}\\
		&+\left(\sum_{k=1}^{\infty} \mathcal{L}^2_{\xi_{k}}(\eta-\rho),\eta-\rho\right)_{H^{s-3}}
		+\sum_{k=1}^{\infty}\left(\mathcal{L}_{\xi_{k}}(\eta-\rho), \mathcal{L}_{\xi_{k}}(\eta-\rho)\right)_{H^{s-3}}\\
		=:\,&\sum_{i=1}^{6}I_i.
		\end{align*}
		Because $H^{s-2}\hookrightarrow W^{1,\infty}$, we can use Lemma \ref{KP  commutator estimate} and integration by parts to arrive at
		\begin{align*}
		\left|I_1\right|
		\lesssim\,& \left|\left(D^{s-2}v(u-v)_x,D^{s-2}(u-v)\right)_{L^2}\right|
		+\left|\left(D^{s-2}(u-v)u_x,D^{s-2}(u-v)\right)_{L^2}\right|\\
		\lesssim\,& \norm{[D^{s-2},v](u-v)_x}_{L^2}\norm{u-v}_{H^{s-2}}
		+\norm{u_x}_{L^{\infty}}\norm{u-v}^2_{H^{s-2}}\\
		\lesssim\, &\left(\norm{v}_{H^s}+\norm{u}_{H^s}\right)\norm{u-v}^2_{H^{s-2}}.
		\end{align*}
		Similarly, we have
		\begin{align*}
		\left|I_2\right|
		\lesssim\,& \left|\left(D^{s-3}v(\eta-\rho)_x,D^{s-3}(\eta-\rho)\right)_{L^2}\right|
		+\left|\left(D^{s-3}(u-v)\eta_x,D^{s-3}(\eta-\rho)\right)_{L^2}\right|\\
		\lesssim\,& \norm{[D^{s-3},v](\eta-\rho)_x}_{L^2}\norm{\eta-\rho}_{H^{s-3}}
		+\norm{\eta_x}_{H^{s-3}}\norm{u-v}_{H^{s-3}}\norm{\eta-\rho}_{H^{s-3}}\\
		\lesssim\, &\norm{v}_{H^s}\norm{\eta-\rho}^2_{H^{s-3}}+\norm{\eta}^2_{H^{s-1}}\norm{u-v}^2_{H^{s-3}}+\norm{\eta-\rho}^2_{H^{s-3}}.
		\end{align*}
		Therefore,
		\begin{align*}
		|I_1|+|I_2|
		\lesssim\,&  \left(\norm{\eta}^2_{H^{s-1}}+\norm{v}_{H^s}+\norm{u}_{H^s}\right)
		\norm{u-v}^2_{H^{s-2}}
		+\left(1+\norm{v}_{H^s}\right)\norm{\eta-\rho}^2_{H^{s-3}}\\
		\lesssim\, & \left(1+\|X\|^2_{\X^{s}}+\|Y\|^2_{\X^{s}}\right)\|X-Y\|^2_{\X^{s-2}}.
		\end{align*}
		Observe that $D^{s-2}D^{-2}=D^{s-4}$.
		Since $s-4>3/2$,  we can invoke Lemma \ref{Liecancellations} 
		%		with $f=D^2(u-v)$, $\mathcal{P}=D^{s-4}$ 
		to obtain
		\begin{align*}
		&\hspace*{-0.5cm}I_3+I_4\\
		=\,&\left(D^{s-4}\sum_{k=1}^{\infty}\mathcal{L}^2_{\xi_{k}}D^2(u-v),
		D^{s-4}D^{2}(u-v)\right)_{L^2}
		+\sum_{k=1}^{\infty}
		\left(D^{s-4}\mathcal{L}_{\xi_{k}}D^2(u-v),D^{s-4}\mathcal{L}_{\xi_{k}}D^2(u-v)\right)_{L^2}\\
		\lesssim\,& \|D^2(u-v)\|^2_{H^{s-4}}\lesssim\, \|u-v\|^2_{H^{s-2}}.
		\end{align*}
		In the same way, we have
		\begin{align*}
		&\hspace*{-0.5cm}I_5+I_6\\
		=\,&\left(D^{s-3}\sum_{k=1}^{\infty}\mathcal{L}^2_{\xi_{k}}(\eta-\rho),
		D^{s-3}(\eta-\rho)\right)_{L^2}
		+\sum_{k=1}^{\infty}
		\left(D^{s-2}\mathcal{L}_{\xi_{k}}(\eta-\rho),
		D^{s-2}\mathcal{L}_{\xi_{k}}(\eta-\rho)\right)_{L^2}\\
		\lesssim\,& \|\eta-\rho\|^2_{H^{s-3}}.
		\end{align*}
		Collecting the above estimates, we obtain the desired result.
	\end{proof}
	
	\subsubsection{Proof of Theorem \ref{SCH2 results}}
	Now we will  prove that all the requirements in Assumption \ref{Assum-A} hold true. 
	We first fix  regular mappings  $g_\e$ and $h_\e $ using the mollification operators from   \eqref{Define Je} and \eqref{Define Te} in the  Appendix
	\ref{sec:appendix} by
	\begin{align}\label{family:gn}
	g_\e(X)&
	=\left(-J_\e [J_\e uJ_\e u_x]
	+\frac{1}{2} J^3_\e D^{-2}\sum_{k=1}^{\infty}\mathcal{L}^2_{\xi_{k}}D^2J_\e u, -J_\e [J_\e uJ_\e \eta_x]+\frac{1}{2} J^3_\e\sum_{k=1}^{\infty} \mathcal{L}^2_{\xi_{k}}J_\e \eta\right).
	\end{align}
	Let
	\begin{align}\label{family: h-n,k}
	h_{\e}^{k}(X)&=\left(-J_\e D^{-2}\mathcal{L}_{\xi_{k}}D^2J_\e u, -J_\e\mathcal{L}_{\xi_{k}}J_\e \eta \right).
	\end{align}
	Similar to \eqref{Define diffusion}, here we define $h_\e(X)\in \LL(\U;\X^s)$   such that
	\begin{align}\label{family:hn}
	h_{\e}(X)(e_k)=h_{\e}^{k}(X),\ \ k\in\N.
	\end{align}
	%	Due to  Lemmas \ref{Lemma for b} and \ref{Lemma for g h} and \eqref{Je r>s}, $h_\e(\X^s)\in \LL_2(\U;\X^s)$. 
	%	
	We choose functions    $k(\cdot)\equiv1$, $ f(\cdot)=C(1+\cdot)$, $q(\cdot)=C(1+\cdot^5)$ for some $C>1$ large enough depending only on $b,g,h$.
	Finally we let  $T_\e =Q_\e= {\tilde J}_\e$,   where ${\tilde J}_\e $ is given in \eqref{Define Te}.
	%\begin{itemize}
	%	\item $\X:=\X^s$, $\Y:=\X^{s-1}$ and $\Z:=\X^{s-2}$, where $\X^s$ and $\V$ are given in \eqref{U Xs V SCH2};
	%	\item $b$, $g$, $h$, $g_\e$ and $h_\e$ are given in \eqref{Define drift}, \eqref{Define diffusion}, \eqref{family:gn} and \eqref{family:hn}, respectively;
	%	\item 
	%	\item $T_\e =Q_\e =T_\e $, 		where $T_\e $ is given in \eqref{Define Te}.
	%	
	%\end{itemize}
	
	Let $s>11/2$.
	Obviously, $\X\hookrightarrow\Y\hookrightarrow\hookrightarrow\Z\hookrightarrow \V$.  Then
	Lemma \ref{Lemma for b} shows $b:\X^s\rightarrow\X^s$, and Lemma \ref{Lemma for g h} implies $g: \X^s\rightarrow\X^{s-2}$ and $h: \X^s\rightarrow \LL_2(\U;\X^{s-1})$. Hence the stochastic integral in \eqref{Abstarct SCH2} is a well defined $\X^{s-1}$-valued local martingale. It is  straightforward to verify that all of them are continuous in $X\in\X^s$.

	\textbf{Checking \ref{b in X}:} Lemma \ref{Lemma for b} implies \ref{b in X}.
	
	\textbf{Checking \ref{gn hn condition}:} By the construction of $g_\e(\cdot)$ and $h_\e(\cdot)$,  \eqref{Je r>s}, Lemma \ref{Lemma for g h} and Assumption \ref{Assum-xi}, it is easy to check that  \ref{gn hn condition} is satisfied.

	\textbf{Checking \ref{growth gn hn}:} 
	We first verify \eqref{A31}. By \eqref{family: h-n,k} and \eqref{Je Te self-adjoint}, we have
	\begin{align*}
	\left(h_{\e}^k(X),X\right)_{\X}
	=\,&\left(-J_\e D^{-2}\mathcal{L}_{\xi_{k}}D^2J_\e u,u\right)_{H^s}
	+\left(-J_\e \mathcal{L}_{\xi_{k}}J_\e \eta,\eta\right)_{H^{s-1}}\\
	=\,&-\left(D^{-2}\mathcal{L}_{\xi_{k}}D^2J_\e u,J_\e u\right)_{H^s}
	-\left(\mathcal{L}_{\xi_{k}}J_\e \eta,J_\e\eta\right)_{H^{s-1}}\\
	=\,&-\left(D^{s-2}\mathcal{L}_{\xi_{k}}D^2J_\e u,D^{s-2}D^2J_\e u\right)_{L^2}
	-\left(D^{s-1}\mathcal{L}_{\xi_{k}}J_\e \eta,D^{s-1}J_\e\eta\right)_{L^{2}}.
	\end{align*}
	Let $v=D^{2}J_\e u$.
	From the definition of the operator $\LL_{\xi}$ in \eqref{define Lxi}, we have
	\begin{align*}
	\left(D^{s-2}\mathcal{L}_{\xi_{k}}v,D^{s-2}v\right)_{L^2}
	=\,&\left(D^{s-2}\left(v\partial_x\xi_k\right),D^{s-2}v\right)_{L^2}
	+\left(D^{s-2}\left(\partial_xv\xi_k\right),D^{s-2}v\right)_{L^2}\\
	=\,&\left([D^{s-2},v]\partial_x\xi_k,D^{s-2}v\right)_{L^2}
	+\left(vD^{s-2}\partial_x\xi_k,D^{s-2}v\right)_{L^2}\\
	&+\left([D^{s-2},\xi_k]\partial_xv,D^{s-2}v\right)_{L^2}
	+\left(\xi_kD^{s-2}\partial_xv,D^{s-2}v\right)_{L^2}.
	\end{align*}
	By Lemma \ref{KP  commutator estimate}, $H^{s-2}\hookrightarrow W^{1,\infty}$ and integration by parts, we arrive at
	\begin{align*}
	\left([D^{s-2},v]\partial_x\xi_k,D^{s-2}v\right)_{L^2}
	+\left(vD^{s-2}\partial_x\xi_k,D^{s-2}v\right)_{L^2}
	\lesssim\, &\|v\|^2_{H^{s-2}}\|\xi_k\|_{H^{s-1}}
	\end{align*}
	and
	\begin{align*}
	\left([D^{s-2},\xi_k]\partial_xv,D^{s-2}v\right)_{L^2}
	+\left(\xi_kD^{s-2}\partial_xv,D^{s-2}v\right)_{L^2}
	\lesssim\, &\|v\|^2_{H^{s-2}}\|\xi_k\|_{H^{s-2}}.
	\end{align*}
	Combining the above estimates and using \eqref{Je Te Hs}, we have that
	\begin{equation*}
	\left(D^{s-2}\mathcal{L}_{\xi_{k}}D^2J_\e u,D^{s-2}D^2J_\e u\right)_{L^2}
	\lesssim\,\|v\|^2_{H^{s-2}}\|\xi_k\|_{H^{s}}\leq\|u\|^2_{H^{s}}\|\xi_k\|_{H^{s}}.
	\end{equation*}
	Similarly,
	\begin{equation*}
	\left(D^{s-1}\mathcal{L}_{\xi_{k}}J_\e \eta,D^{s-1}J_\e\eta\right)_{L^{2}}
	\lesssim\,\|J_\e\eta\|^2_{H^{s-1}}\|\xi_k\|_{H^{s}}\leq\|\eta\|^2_{H^{s-1}}\|\xi_k\|_{H^{s}}.
	\end{equation*}
	Therefore, by using \eqref{family: h-n,k}, \eqref{family:hn}, Assumption \ref{Assum-xi} and \eqref{Je Te Hs}, we conclude that
	\begin{align*}
	\sum_{k=1}^{\infty}\left|\left(h_\e(X)\xi_k,X\right)_{\X}\right|^2
	=\sum_{k=1}^{\infty}\left|\left(h_{\e}^k(X),X\right)_{\X}\right|^2
	\lesssim\, \sum_{k=1}^{\infty}\|\xi_k\|^2_{H^{s}}\left(\|u\|^2_{H^s}+\|\eta\|^2_{H^{s-1}}\right)^2\leq C\|X\|^4_{\X},
	\end{align*} 
	which yields \eqref{A31}. 
	
	Now we prove \eqref{A32}. For all $X=(u,\eta)\in \X^s$, we have
	\begin{align*}
	&\hspace*{-0.5cm}2\left(g_\e(X),X\right)_{\X^{s}}+\|h_\e(X)\|^2_{\LL_2(\U;\X^{s})}\\
	=\,& -2(J_\e [J_\e uJ_\e u_x],u)_{H^{s}}
	-2\left(J_\e[J_\e uJ_\e \eta_x], \eta\right)_{H^{s-1}}\\
	&+\left(D^sJ^3_\e D^{-2}\sum_{k=1}^{\infty}\mathcal{L}^2_{\xi_{k}}D^2J_\e u, D^s u\right)_{L^2}
	+\sum_{k=1}^{\infty}
	\left(D^sJ_\e D^{-2}\mathcal{L}_{\xi_{k}}D^2J_\e u, D^sJ_\e D^{-2}\mathcal{L}_{\xi_{k}}D^2J_\e u\right)_{L^2}\\
	&+\left(D^{s-1}J^3_\e \sum_{k=1}^{\infty}\mathcal{L}^2_{\xi_{k}}J_\e \eta, D^{s-1} \eta\right)_{L^2}
	+\sum_{k=1}^{\infty}
	\left(D^{s-1}J_\e\mathcal{L}_{\xi_{k}}J_\e \eta, D^{s-1}J_\e\mathcal{L}_{\xi_{k}}J_\e \eta\right)_{L^2} \\
	=:\,&\sum_{i=1}^{6}E_i.
	\end{align*}

	It follows from \eqref{Je Te Ds}, \eqref{Je Te Hs}, Lemma \ref{KP  commutator estimate} and integration by parts that
	\begin{equation*}
	|E_1|=2\left|\left([D^s,J_\e u]J_\e u_x, D^sJ_\e u\right)_{L^2}
	+\left(J_\e uD^sJ_\e u_x, D^sJ_\e u\right)_{L^2}\right|
	\lesssim\,\|u_x\|_{L^{\infty}}\|u\|^2_{H^s}
	\end{equation*}
	and
	\begin{align*}
	|E_2|
	=\,&2\left|\left([D^{s-1},J_\e u]J_\e \eta_x, D^{s-1}J_\e \eta\right)_{L^2}+\left(J_\e uD^{s-1}J_\e \eta_x, D^{s-1}J_\e \eta\right)_{L^2}\right|\\
	\lesssim\,&(\|u_x\|_{L^\infty}+\|\eta_x\|_{L^\infty})
	\left(\|u\|^2_{H^s}+\|\eta\|^2_{H^{s-1}}\right).
	\end{align*}
	By $\eqref{Je Te Ds}$, $\eqref{Je Te self-adjoint}$ and the fact that $D^{s-2}=D^sD^{-2}$, we obtain
	\begin{align*}
	&\hspace*{-0.5cm}E_3+E_4\\
	=\,&\left(D^{s-2}J_\e\sum_{k=1}^{\infty}\mathcal{L}^2_{\xi_{k}}D^2J_\e u,  D^{s-2}J_\e D^{2}J_\e u\right)_{L^2}
	+\sum_{k=1}^{\infty}
	\left(D^{s-2}J_\e\mathcal{L}_{\xi_{k}}D^2J_\e u, D^{s-2}J_\e\mathcal{L}_{\xi_{k}}  D^2J_\e u\right)_{L^2}.
	\end{align*}
	Since $\mathcal{P}=D^{s-2}J_\e \in \mbox{OPS}^{s-2}_{1,0}$ (cf. Lemma \ref{lemma;pseudo}), we apply Lemma \ref{Liecancellations} to arrive at
	\begin{align*}
	E_3+E_4
	\lesssim\, \norm{D^{2}J_{\e} u}^2_{H^{s-2}} \leq C \norm{u}^2_{H^{s}},
	\end{align*}
	where we have used \eqref{Je Te Hs} in the last inequality.
	Similarly,
	\begin{align*}
	&\hspace*{-0.5cm}E_5+E_6\\
	=\,& \left(D^{s-1}J_\e\sum_{k=1}^{\infty}\mathcal{L}^2_{\xi_{k}}J_\e  \eta,  D^{s-1}J_\e J_\e\eta\right)_{L^2}
	+\sum_{k=1}^{\infty}
	\left(D^{s-1}J_\e\mathcal{L}_{\xi_{k}}J_\e \eta, D^{s-1}J_\e\mathcal{L}_{\xi_{k}}J_\e \eta\right)_{L^2} \\
	\leq &C \|J_\e\eta\|^2_{H^{s-1}}\leq C \|\eta\|^2_{H^{s-1}}.
	\end{align*}
	Combining the above estimates, we arrive at
	\begin{align*}
	2\left(g_\e(X),X\right)_{\X^s}+\|h_\e(X)\|^2_{\LL_2(\U;\X^s)}
	\lesssim\, (1+\|u_x\|_{L^\infty}+\|\eta_x\|_{L^\infty})\left(\|u\|^2_{H^s}+\|\eta\|^2_{H^{s-1}}\right)\leq f\big(\|X\|_{\V}\big)\|X\|^2_{\X^s},
	\end{align*}
	which implies \eqref{A32} with $k(t)\equiv 1$.
	
	\textbf{Checking \ref{uniqueness requirement}:} It is clear that $\X=\X^s$ is dense in $\Z=\X^{s-2}$. Since $s-2>\frac{5}{2}$, inequality \eqref{B11} follows directly from Lemma \ref{Lemma for b}. Applying Lemma \ref{Difference cancellation} yields \eqref{B12}.
	
	\textbf{Checking \ref{continuity requirement}:} Recall that  $\tilde{J}_\e =(1-\e^{2}\Delta)^{-1}$. Due to \eqref{Je Te Hs} and $T_\e =Q_\e =\tilde{J}_\e $, \ref{continuity requirement} is a direct consequence of \ref{blow-up requirement}, which will be checked below.
	
	\textbf{Checking \ref{blow-up requirement}:}
	It is easy to prove \eqref{B21} and we omit the details here. Then we notice that
	\begin{align*}
	&\hspace*{-0.5cm}2\left(T_\e  g(X),T_\e  X\right)_{\X^{s}} +\|T_\e  h(X)\|^2_{\LL_2(\U;\X^{s})} \\
	= &-2(T_\e  [uu_x],T_\e  u)_{H^{s}}
	-2\left(T_\e [u\eta_x], T_\e  \eta\right)_{H^{s-1}}\\
	&+\left(D^sT_\e  D^{-2}\sum_{k=1}^{\infty}\mathcal{L}^2_{\xi_{k}}D^2 u, D^s T_\e  u\right)_{L^2}
	+\sum_{k=1}^{\infty}
	\left(D^sT_\e  D^{-2}\mathcal{L}_{\xi_{k}}D^2 u, D^sT_\e  D^{-2}\mathcal{L}_{\xi_{k}}D^2 u\right)_{L^2}\\
	&+\left(D^{s-1}T_\e  \sum_{k=1}^{\infty}\mathcal{L}^2_{\xi_{k}} \eta, D^{s-1} T_\e \eta\right)_{L^2}
	+\sum_{k=1}^{\infty}
	\left(D^{s-1}T_\e \mathcal{L}_{\xi_{k}} \eta, D^{s-1}T_\e \mathcal{L}_{\xi_{k}}\eta\right)_{L^2}=\sum_{i=1}^{6}R_i.
	\end{align*}
	For the first term we have that
	\begin{align*}
	|R_{1}| =\,& 2\left|\left(D^sT_\e  
	\left[uu_x\right],D^sT_\e  u\right)_{L^2}\right|\notag\\
	\leq\,&2\left|\left(\left[D^s,
	u\right]u_x,D^sT^2_\e u\right)_{L^2}+
	\left([T_\e ,u]D^su_x, D^sT_\e  u\right)_{L^2}
	+\left(uD^sT_\e  u_x, D^sT_\e  u\right)_{L^2}\right|\notag\\
	\leq\,&  C\|u_x\|_{L^{\infty}}\|u\|_{H^s}\|T_\e  u\|_{H^s}
	+C\|u_x\|_{L^{\infty}}\|T_\e  u\|^2_{H^s},
	\end{align*}
	where we have used Lemmas \ref{Te commutator} and \ref{KP  commutator estimate}, integration by parts, embedding $H^{s-1}\hookrightarrow W^{1,\infty}$, \eqref{Je Te self-adjoint} and \eqref{Je Te Hs}.
	Similarly, we can show that
	\begin{align*}
	|R_{2}| =\,& 2\left|\left(D^{s-1}T_\e  
	\left[u\eta_x\right],D^{s-1}T_\e  \eta\right)_{L^2}\right|\notag\\
	=\,&2\left|\left(\left[D^{s-1},u\right]\eta_x,D^{s-1}T^2_\e \eta\right)_{L^2}
	+\left([T_\e ,u]D^{s-1}\eta_x, D^{s-1}T_\e  \eta\right)_{L^2}
	+\left(uD^{s-1}T_\e  \eta_x, D^{s-1}T_\e  \eta\right)_{L^2}\notag\right|\\
	\leq\,&  C\left(\|u_x\|_{L^{\infty}}\|\eta\|_{H^{s-1}}\|T_\e  \eta\|_{H^{s-1}}
	+\|\eta_x\|_{L^{\infty}}\|u\|_{H^s}\|T_\e  \eta\|_{H^{s-1}}\right)
	+C\|u_x\|_{L^{\infty}}\|T_\e  \eta\|^2_{H^s}\\
	\lesssim\,& \|u_x\|_{L^{\infty}}\|\eta\|_{H^{s-1}}\|T_\e  \eta\|_{H^{s-1}}
	+\|\eta_x\|_{L^{\infty}}\|u\|_{H^s}\|T_\e  \eta\|_{H^{s-1}}.
	\end{align*}
	Using Lemma \ref{Liecancellations} yields
	%	with $\mathcal{P}= D^{s-2}T_{\e}\in \mbox{OPS}^{s-2}_{1,0}$ (cf. Lemma \ref{lemma;pseudo}) 
	%	and $f=D^2 u$,
	%	 we have that
	\begin{align*}
	R_3+R_4=\,&\left(D^{s-2}T_{\e}\sum_{k=1}^{\infty}\mathcal{L}^2_{\xi_{k}}D^2 u, D^{s-2}T_{\e}D^2 u\right)_{L^2}
	+\sum_{k=1}^{\infty}
	\left(-D^{s-2}T_{\e}\mathcal{L}_{\xi_{k}}D^2 u, -D^{s-2}T_{\e}\mathcal{L}_{\xi_{k}}D^2 u\right)_{L^2} \\
	\lesssim\,&  \|D^2u\|^2_{H^{s-2}}\leq \|u\|^{2}_{H^s}. 
	\end{align*}
	and analogously 
	\begin{align*}
	R_5+R_6=\,&\left(D^{s-1}T_\e  \sum_{k=1}^{\infty}\mathcal{L}^2_{\xi_{k}} \eta, D^{s-1} T_\e \eta\right)_{L^2}
	+\sum_{k=1}^{\infty}
	\left(D^{s-1}T_\e \mathcal{L}_{\xi_{k}} \eta, D^{s-1}T_\e \mathcal{L}_{\xi_{k}}\eta\right)_{L^2} \lesssim\,  \|\eta\|^2_{H^{s-1}}.
	\end{align*}
	Gathering together the above estimates and noticing \eqref{Je Te Hs}, we get 
	\begin{align*}
	&\hspace*{-0.5cm}2\left(T_\e  g(X),T_\e  X\right)_{\X^{s}} +\|T_\e  h(X)\|^2_{\LL_2(\U;\X^{s})} \notag\\
	\lesssim\,& (1+\|u_x\|_{L^\infty}+\|\eta_x\|_{L^\infty})\left(\|u\|^2_{H^s}+\|\eta\|^2_{H^{s-1}}\right)
	\leq f\big(\|X\|_{\V}\big)\|X\|^2_{\X^s},
	\end{align*}
	which gives \eqref{blow-up requirement 1}. We are just left to show \eqref{blow-up requirement 2} to conclude the proof of Theorem \ref{SCH2 results}. To this end, we recall \eqref{Define drift} and consider
	\begin{align*}
	-\left(T_\e  h_k(X),T_\e  X\right)_{\X^s}=\,&
	\left(T_{\e}D^{s-2}\mathcal{L}_{\xi_{k}}D^2 u,T_{\e}D^{s}u\right)_{L^2}
	+\left(T_{\e}D^{s-1}\mathcal{L}_{\xi_{k}}\eta,T_{\e}D^{s-1}\eta\right)_{L^2}\\ 
	=\,& \left(\mathcal{P}_{1}\mathcal{L}_{\xi_{k}}D^2 u,\mathcal{P}_{1}D^2 u\right)_{L^2}  +\left(\mathcal{P}_{2}\mathcal{L}_{\xi_{k}}\eta,\mathcal{P}_{2}\eta\right)_{L^2} \\
	=\,& \left(\mathcal{T}_{1}D^2 u,\mathcal{P}_{1}D^2 u\right)_{L^2} 
	+\left(\mathcal{L}_{\xi_{k}}\mathcal{P}_{1}D^2 u,\mathcal{P}_{1}D^2 u\right)_{L^2}\\
	& \quad  +\left(\mathcal{T}_{2}\eta,\mathcal{P}_{2}\eta\right)_{L^2}
	+\left(\mathcal{L}_{\xi_{k}}\mathcal{P}_{2}\eta,\mathcal{P}_{2}\eta\right)_{L^2}\\
	=:\,&\displaystyle\sum_{i=1}^4 J_i,
	\end{align*}
	where $\mathcal{P}_{1}:=T_{\e}D^{s-2}\in \mbox{OPS}^{s-2}_{1,0}$, $\mathcal{P}_{2}:=T_{\e}D^{s-1}\in \mbox{OPS}^{s-1}_{1,0}$ (cf. Lemma \ref{lemma;pseudo}),  and 
	$\mathcal{T}_{1}=\left[\mathcal{P}_{1}, \mathcal{L}_{\xi_{k}} \right]$, $\mathcal{T}_{2}=\left[\mathcal{P}_{2}, \mathcal{L}_{\xi_{k}} \right]$.
	Using integration by parts, \eqref{define Lxi} and \eqref{Je Te Ds}, we have that
	\begin{align*}
	|J_{2}|+|J_{4}|&
	\lesssim\, \norm{\partial_{x}\xi_{k}}_{L^\infty}\left(\norm{\mathcal{P}_{1}D^2 u}^2_{L^2}+\norm{\mathcal{P}_{2} \eta}^2_{L^2}\right)
	\lesssim\,\norm{\partial_{x}\xi_{k}}_{L^\infty}\|T_\e  X\|^2_{\X^s}.
	\end{align*}
	Using \eqref{Je Te self-adjoint} and \eqref{Je Te Ds}, we have
	\begin{align*}
	J_3=\,&\left(\mathcal{T}_{2} \eta,\mathcal{P}_{2} \eta\right)_{L^2}\\
	=\,&\left(D^{s-1}\mathcal{L}_{\xi_{k}}\eta,D^{s-1}T^2_\e \eta\right)_{L^2} 
	-\left(\mathcal{L}_{\xi_{k}}D^{s-1}T_\e \eta,D^{s-1}T_\e  \eta\right)_{L^2} \\
	=\,&\left(D^{s-1}\xi_k\partial_x\eta,D^{s-1}T^2_\e \eta\right)_{L^2} +
	\left(D^{s-1}\eta\partial_x\xi_k,D^{s-1}T^2_\e \eta\right)_{L^2}-
	\left(\mathcal{L}_{\xi_{k}}D^{s-1}T_\e \eta,D^{s-1} T_\e \eta\right)_{L^2}\\
	=\,&\left(\left[D^{s-1},\xi_k\right]\partial_x\eta,D^{s-1}T^2_\e \eta\right)_{L^2} +
	\left(T_\e \xi_kD^{s-1}\partial_x\eta,D^{s-1}T_\e  \eta\right)_{L^2}\\
	&+\left(D^{s-1}\eta\partial_x\xi_k,D^{s-1}T^2_\e \eta\right)_{L^2}-
	\left(\mathcal{L}_{\xi_{k}}D^{s-1}T_\e \eta,D^{s-1} T_\e \eta\right)_{L^2}\\
	=:\,& \sum_{i=1}^4 K_i.
	\end{align*}
	On account of $H^{s-1}\hookrightarrow W^{1,\infty}$ and integration by parts, 
	it holds that
	\begin{equation*}
	|K_3|\lesssim\, \|\eta\partial_x\xi_k\|_{H^{s-1}}\|T_\e \eta\|_{H^{s-1}}\leq \|\xi_k\|_{H^s}\|\eta\|_{H^{s-1}}\|T_\e \eta\|_{H^{s-1}},
	\end{equation*}
	and
	\begin{equation*}
	|K_4|\lesssim\, \norm{\partial_{x}\xi_{k}}_{L^\infty}\norm{T_\e \eta}^2_{H^{s-1}}
	\lesssim\, \norm{\partial_{x}\xi_{k}}_{L^\infty}\|\eta\|_{H^{s-1}}\|T_\e \eta\|_{H^{s-1}}.
	\end{equation*}
	Then we apply Lemma \ref{KP  commutator estimate} to $K_1$ to find
	\begin{equation*}
	|K_1|\lesssim\, \norm{\xi_{k}}_{H^{s}}\|\eta\|_{H^{s-1}}\|T_\e \eta\|_{H^{s-1}}.
	\end{equation*}
	For $K_2$, we use Lemma \ref{Te commutator} and integration by parts to derive
	\begin{align*}
	|K_2|
	\lesssim\,& 
	\left|\left(\left[T_\e ,\xi_k\right]\partial_xD^{s-1}\eta,D^{s-1}T_\e  \eta\right)\right|
	+\left|\left(\xi_k\partial_xD^{s-1}T_\e \eta,D^{s-1}T_\e  \eta\right)\right|
	\lesssim\, \norm{\partial_{x}\xi_{k}}_{L^\infty}\|\eta\|_{H^{s-1}}\|T_\e \eta\|_{H^{s-1}}.
	\end{align*}
	Therefore, 
	$$|J_3|=\left|\left(\mathcal{T}_{2} \eta,\mathcal{P}_{2} \eta\right)_{L^2}\right|\lesssim\, \|\xi_k\|_{H^s}\|\eta\|_{H^{s-1}}\|T_\e \eta\|_{H^{s-1}}.$$
	
	The form  $J_4=\left(\mathcal{T}_{1}D^2 u,\mathcal{P}_{1}D^2 u\right)_{L^2}$ can be handled in the same way  
	using $H^{s-2}\hookrightarrow W^{1,\infty}$. Hence we have
	$$|J_3|=\left|\left(\mathcal{T}_{1} f,\mathcal{P}_{1} f\right)_{L^2}\right|\lesssim\, \|\xi_k\|_{H^s}\|f\|_{H^{s-2}}\|T_\e  f\|_{H^{s-2}}
	\lesssim\, \|\xi_k\|_{H^s}\|u\|_{H^{s}}\|T_\e  u\|_{H^{s}}.$$
	Now we summarize the above estimates, and use \eqref{Define diffusion} and
	Assumption \ref{Assum-xi}  to arrive at
	\begin{align}\label{final sch2}
	\sum_{k=1}^{\infty}\left|\left(T_\e  h(X)e_k,T_\e  X\right)_{\X^s}\right|^2\lesssim\, \sum_{k=1}^{\infty}\norm{\xi_{k}}^2_{H^s}\|X\|^2_{\X^s}\|T_\e  X\|^2_{\X^s}
	\leq C\|X\|^2_{\X^s}\|T_\e  X\|^2_{\X^s}.
	\end{align}
	Hence we obtain inequality \eqref{blow-up requirement 2} 
	and complete the proof.

	\subsection{Stochastic CCF model: Proof of Theorem \ref{transport:nonlocal results}}\label{Sec:STNL} 
	In this section we will apply Theorem \ref{Abstract irregular SDE: T} to  \eqref{Transport:nonloca:eq:ito} with $x\in\T$ to obtain Theorem \ref{transport:nonlocal results}. To that purpose, we set $X=\theta$ and 
	\begin{equation}\label{def:b:g:h:nonlocal}
	\begin{array}{rcccl}
	b(t,X)&=&b(X)&=&0, \\[1.2ex]
	g(t,X)&=&g(X)&=& \displaystyle -(\mathcal{H}\theta)\partial_{x}\theta+\frac{1}{2}\sum_{k=1}^{\infty} \mathcal{L}^2_{\xi_{k}}\theta,\\[2.2ex]
	h^k(t,X)&=&h^k(X)&=&-\mathcal{L}_{\xi_{k}}\theta,\ \ k\in\N.
	\end{array} 
	\end{equation}
	As in \eqref{Define diffusion}, we  define $h(X)\in \LL(\U;H^s)$ such that
	\begin{align}\label{Define diff:nonlocal}
	h(X)(e_k)=h^{k}(X),\ \ k\in\N.
	\end{align}
	With the above notations,  we reformulate \eqref{Transport:nonloca:eq:ito} in the abstract form, i.e.,
	\begin{equation}\label{Abstarct nlp}
	\left\{\begin{aligned}
	&\dd X=\left(b(X)+g(X)\right) \dd t+h(X)\dd\W,\\
	&X(0)=\theta_0.
	\end{aligned}\right.
	\end{equation}
	
	To prove Theorem \ref{transport:nonlocal results}, we would like to invoke Theorem \ref{Abstract irregular SDE: T} to this setting. To do that, we just need to check the Assumption \ref{Assum-A}. Now we let $r\in({3}/{2},s-2)$, and then let
	\begin{equation}\label{U Xs V nlt}
	\X^s=H^s \ {\rm and} \ \V=H^r.
	\end{equation}
	
	%	From now on, we will frequently use the fact that the Hilbert transform $\mathcal{H}$ is a continuous operator from $H^s$ to $H^s$ for any $s\ge0$, and from $L^\infty$ to $BMO$. 
	
	\subsubsection{Estimates on nonlinear terms}
	Analogously to Section \ref{subsubest} we will need the following auxiliary lemmas.
	\begin{Lemma}\label{Lemma for g h nlt}
		Let Assumption \ref{Assum-xi} hold true and  $s>5/2$. If $X=\theta \in \X^s$, 
		then  $g:\X^s\rightarrow \X^{s-2}$ and $h:\X^s\rightarrow \LL_2(\U;\X^{s-1})$ such that
		\begin{align*}
		\|g(X)\|_{\X^{s-2}}&\lesssim\, 1+\|X\|^2_{\X^{s}},
		\end{align*}
		and 
		\begin{align*}
		\|h(X)\|_{\LL_{2}(\U;\X^{s-1})}&\lesssim\, \|X\|_{\X^{s}}.
		\end{align*}
	\end{Lemma}
	\begin{proof}
		Using $H^{s-2}\hookrightarrow W^{1,\infty}$, the continuity of the Hilbert transform for $s\geq0$ and  Remark \ref{Remark xi with f}, one can prove the above estimates directly.  We omit the details for exposition clearness.
	\end{proof}

	\begin{Lemma}\label{Difference cancellation nlt}
		Let $X=\theta\in \X^s$ and $Y=\rho \in \X^s$. Then we have  that for $s>7/2$,
		\begin{align*}
		2\left(g(t,X)-g(t,Y),X-Y\right)_{\X^{s-2}}+\|h(t,X)-h(t,Y)\|^2_{\LL_2(\U;\X^{s-2})}\lesssim\,& \left(1+\|X\|^2_{\X^{s}}+\|Y\|^2_{\X^{s}}\right)\|X-Y\|^2_{\X^{s-2}}.
		\end{align*}
	\end{Lemma}

	\begin{proof}
		Recalling \eqref{def:b:g:h:nonlocal} and \eqref{Define diff:nonlocal}, we have
		\begin{align*}
		&\hspace*{-0.5cm}2\left(g(X)-g(Y),X-Y\right)_{\X^{s-2}}+\|h(t,X)-h(t,Y)\|^2_{\LL_2(\U;\X^{s-2})}\\
		=\,& 2((\mathcal{H}\rho)\rho_x-(\mathcal{H}\theta)\theta_x,\theta-\rho)_{H^{s-2}}
		+\left(\sum_{k=1}^{\infty} \mathcal{L}^2_{\xi_{k}}(\theta-\rho),\theta-\rho\right)_{H^{s-2}}
		+\sum_{k=1}^{\infty}
		\left(\mathcal{L}_{\xi_{k}}(\theta-\rho),\mathcal{L}_{\xi_{k}}(\theta-\rho)\right)_{H^{s-2}}.
		\end{align*}
		Because $H^{s-2}\hookrightarrow W^{1,\infty}$, we use Remark \ref{Remark xi with f}, Lemma \ref{KP  commutator estimate}, the continuity of the Hilbert transform and integration by parts to bound the first term as
		\begin{align*}
		&\hspace*{-0.5cm}((\mathcal{H}\rho)\rho_x-(\mathcal{H}\theta)\theta_x,\theta-\rho)_{H^{s-2}}\\
		\lesssim\,& \left|\left(D^{s-2}(\mathcal{H}\rho)(\theta-\rho)_x,D^{s-2}(\theta-\rho)\right)_{L^2}\right|
		+\left|\left(D^{s-2}(\mathcal{H}(\theta-\rho))\theta_x,D^{s-2}(\theta-\rho)\right)_{L^2}\right|\\
		\lesssim\,& \norm{[D^{s-2},\mathcal{H}\rho](\theta-\rho)_x}_{L^2}\norm{\theta-\rho}_{H^{s-2}}
		+\norm{\partial_x\mathcal{H}\rho}_{L^{\infty}}\norm{\theta-\rho}^2_{H^{s-2}}\\
		&+\norm{[D^{s-2},\mathcal{H}(\theta-\rho)]\theta_x}_{L^2}\norm{\theta-\rho}_{H^{s-2}}
		+\norm{\partial_x\mathcal{H}(\theta-\rho)}_{L^{\infty}}\norm{\theta-\rho}^2_{H^{s-2}}\\
		\lesssim\, &\left(\norm{\rho}_{H^s}+\norm{\theta}_{H^s}\right)\norm{\theta-\rho}^2_{H^{s-2}}.
		\end{align*}
		The last two terms can be bounded by invoking Lemma
		\ref{Liecancellations} 
		%		with $f=(\theta-\rho)$, $\mathcal{P}=D^{s-2}$ 
		to obtain
		\begin{align*}
		&\left(D^{s-2}\sum_{k=1}^{\infty}\mathcal{L}^2_{\xi_{k}}(\theta-\rho),
		D^{s-2}(\theta-\rho)\right)_{L^2}
		+\sum_{k=1}^{\infty}
		\left(D^{s-2}\mathcal{L}_{\xi_{k}}(\theta-\rho), D^{s-2}\mathcal{L}_{\xi_{k}}(\theta-\rho)\right)_{L^2}
		\lesssim\, \|\theta-\rho\|^2_{H^{s-2}}.
		\end{align*}
		Collecting the above estimates, we obtain the desired result.
	\end{proof}

	\subsubsection{Proof of Theorem \ref{transport:nonlocal results}}
	To avoid unnecessary repetition,
	we just sketch the main points of the proof since it is similar to the proof of Theorem \ref{SCH2 results}. Recalling \eqref{Define Je}, we  define
	\begin{align}\label{family:gn:nlt}
	g_\e(X)&
	=-J_\e [\left(\mathcal{H}J_\e \theta \right)\partial_{x}J_\e \theta]+\frac{1}{2} J^3_\e\sum_{k=1}^{\infty} \mathcal{L}^2_{\xi_{k}}J_\e \theta.
	\end{align}
	Let
	\begin{align}\label{family: h-n,k:nlt}
	h_{\e}^k(X)&= -J_\e\mathcal{L}_{\xi_{k}}J_\e \theta.
	\end{align}
	Similar to \eqref{Define diffusion},  we define $h_\e(X)\in \LL(\U;\X^s)$ such that
	\begin{align}\label{family:hn:nlt}
	h_{\e}(X)(e_k)=h_{\e}^k(X),\ \ k\in\N.
	\end{align}
	We now prove that all the estimates in Assumption \ref{Assum-A}  hold true  for 
	\begin{itemize}
		\item $\X=\X^s$, $\Y=\X^{s-1}$ and $\Z=\X^{s-2}$, where $\X^s$ and $\V$ are given in \eqref{U Xs V nlt},
		\item $b$, $g$, $h$, $g_\e$ and $h_\e$ are given in \eqref{def:b:g:h:nonlocal}, \eqref{family:gn:nlt}, \eqref{family: h-n,k:nlt} and \eqref{family:hn:nlt}, respectively,
		\item $k(\cdot)\equiv1$, $ f(\cdot)=C(1+\cdot)$, $q(\cdot)=C(1+\cdot^5)$ for some $C>1$ large enough,
		\item $T_\e =Q_\e ={\tilde J}_\e $, 		where ${\tilde J}_\e $ is given in \eqref{Define Te}.
		
	\end{itemize}
	
	Let $s>7/2$.
	Obviously, $\X\hookrightarrow\Y\hookrightarrow\hookrightarrow\Z\hookrightarrow \V$.  	Moreover, Lemma \ref{Lemma for g h nlt} implies $g: \X^s\mapsto\X^{s-2}$ and $h: \X^s\mapsto \LL_2(\U;\X^{s-1})$.  Hence the stochastic integral in \eqref{Abstarct nlp} is a well defined $\X^{s-1}$-valued local martingale. It is  easy  to check that $g$ and $h$ are continuous in $X\in\X^s$.

	\textbf{Checking \ref{b in X}:} Trivial, since $b(t,X)\equiv 0$. 
	
	\textbf{Checking \ref{gn hn condition}:} By the construction of $g_\e(X)$ and $h_\e(X)$,  \eqref{Je r>s}, Lemma \ref{Lemma for g h nlt} and Assumption \ref{Assum-xi},  \ref{gn hn condition} is verified.
	
	\textbf{Checking \ref{growth gn hn}:} 
	Since \eqref{family: h-n,k:nlt} enjoys similar  estimates as we established for  \eqref{family: h-n,k}, the first part  \eqref{A31} can be proved as before. Therefore, we just need to show \eqref{A32}.  For all $X=\theta \in \X^s$, we have
	\begin{align*}
	2\left(g_\e(X),X\right)_{\X^{s}}+\|h_\e(X)\|^2_{\LL_2(\U;\X^{s})}=\,& -2(D^sJ_\e [\mathcal{H} J_\e \theta \partial_{x}J_\e \theta], D^s\theta)_{L^2} \\
	&+\left(D^{s}J^3_\e \sum_{k=1}^{\infty}\mathcal{L}^2_{\xi_{k}}J_\e \theta, D^{s} \theta\right)_{L^2}
	+\sum_{k=1}^{\infty}
	\left(D^{s}J_\e\mathcal{L}_{\xi_{k}}J_\e \theta, D^{s}J_\e\mathcal{L}_{\xi_{k}}J_\e \theta\right)_{L^2} \\
	=:\,&\sum_{i=1}^{3}E_i.
	\end{align*}
	Invoking Lemma  \ref{Liecancellations}  with $\mathcal{P}=D^{s}J_\e \in \mbox{OPS}^{s}_{1,0}$ (cf. Lemma \ref{lemma;pseudo}), we have that
	\begin{align*}
	E_2+E_3=\,& \left(D^{s}J_\e\sum_{k=1}^{\infty}\mathcal{L}^2_{\xi_{k}}J_\e  \theta,  D^{s}J_\e J_\e\theta\right)_{L^2}
	+\sum_{k=1}^{\infty}
	\left(D^{s}J_\e\mathcal{L}_{\xi_{k}}J_\e \theta, D^{s}J_\e\mathcal{L}_{\xi_{k}}J_\e \theta\right)_{L^2}  \\
	\leq\, &C \|J_\e\theta\|^2_{H^{s}}\leq C \|\theta\|^2_{H^{s}}.
	\end{align*}
	To bound the first term, we notice that $H^r\hookrightarrow W^{1,\infty}$, then we use Lemma \ref{KP  commutator estimate}, integration by parts, \eqref{Je Te Hs} and \eqref{H Ds} to find
	\begin{align*}
	|E_{1}|=\,&2\left|\left(\mathcal{H}J_\e \theta \partial_{x}J_\e D^{s} \theta, D^{s}J_{\e}\theta\right)_{L^2} + 2\left(\left[D^{s}, \mathcal{H}J_{\e}\theta \right]\partial_{x}J_{\e}\theta, D^{s}J_{\e}\theta \right)_{L^2}\right| \\
	\lesssim\, & \norm{\mathcal{H}\partial_{x}\theta}_{L^{\infty}}\norm{D^{s}J_{\e}\theta}^{2}_{L^2}+\norm{\left[D^{s}, \mathcal{H}J_{\e}\theta \right]\partial_{x}J_{\e}\theta}_{L^2}\norm{D^{s}J_{\e}\theta}_{L^2} \\
	\lesssim\, & \norm{\mathcal{H}\partial_{x}J_{\e}\theta}_{L^{\infty}}\norm{D^{s}J_{\e}\theta}^{2}_{L^2} + \norm{\partial_{x}\mathcal{H}J_{\e}\theta}_{L^\infty}\norm{D^{s-1}\partial_{x}J_{\e}\theta}_{L^{2}}+\norm{D^{s}\mathcal{H}J_{\e}\theta}_{L^2}\norm{\partial_{x}J_{\e}\theta}_{L^{\infty}} \\
	\lesssim\, & \norm{\mathcal{H}\partial_{x}\theta}_{L^\infty}\norm{\theta}^{2}_{H^{s}}.
	\end{align*}
	Combining the above estimates, we arrive at
	\begin{align*}
	2\left(g_\e(X),X\right)_{\X^s}+\|h_\e(t,X)\|^2_{\LL_2(\U;\X^s)}
	\lesssim\, (1+\norm{\mathcal{H}\partial_{x}\theta}_{L^\infty})\|\theta\|^2_{H^{s}}\leq f\big(\|X\|_{\V}\big)\|X\|^2_{\X^s},
	\end{align*}
	which implies \eqref{A32}. 
	
	\textbf{Checking \ref{uniqueness requirement}:} The dense embedding $\X=\X^s\hookrightarrow\Z=\X^{s-2}$ and \eqref{B11} is clear.   Applying Lemma \ref{Difference cancellation nlt}, we infer \eqref{B12}.
	
	\textbf{Checking \ref{continuity requirement}:}  As before, this is a direct consequence of \ref{blow-up requirement}, which will be shown next. 
	
	\textbf{Checking \ref{blow-up requirement}:} Following the same way as we  proved \eqref{final sch2}, we 
	have that for some $C>1$,
	\begin{align}\label{STNL:refined estimate 1}
	\sum_{k=1}^{\infty}\left|\left(T_\e  h(\theta)\xi_k,T_\e  \theta\right)_{H^s}\right|^2\leq C\|\theta\|^2_{H^s}\|T_\e  \theta\|^2_{H^s}.
	\end{align}
	Hence \eqref{blow-up requirement 2} holds. Now
	we just need to prove \eqref{blow-up requirement 1}. Indeed,
	\begin{align*}
	&\hspace*{-0.5cm}2\left(T_\e  g(X),T_\e  X\right)_{\X^{s}} +\|T_\e  h(X)\|^2_{\LL_2(\U;\X^{s})} \\
	=\,& -2(T_\e  [\mathcal{H}\theta \theta_x],T_\e  \theta)_{H^{s}} 
	+ \left(D^{s}T_\e  \sum_{k=1}^{\infty}\mathcal{L}^2_{\xi_{k}} \theta, D^{s} T_\e \theta\right)_{L^2}   +\sum_{k=1}^{\infty}
	\left(D^{s}T_\e \mathcal{L}_{\xi_{k}} \theta, D^{s}T_\e  \mathcal{L}_{\xi_{k}}\theta\right)_{L^2}\\
	=:&\sum_{i=1}^{3}R_i.
	\end{align*}
	
	Using Lemma \ref{KP  commutator estimate}, \eqref{H Ds}, \eqref{H Hs}, integration by parts, Lemma \ref{Te commutator},  and \eqref{Je Te Hs}, we have
	\begin{align*}
	|R_{1}| \leq\,&2\left|\left(\left[D^s,
	\mathcal{H}\theta\right]\theta_x,D^sT^2_\e \theta \right)_{L^2}+
	\left([T_\e ,\mathcal{H}\theta]D^s\theta_x, D^sT_\e  \theta\right)_{L^2}
	+\left(\mathcal{H}\theta D^sT_\e  \theta_x, D^sT_\e  \theta\right)_{L^2}\right|\notag\\
	\leq\,&  C\|\theta_x\|_{L^{\infty}}\|\theta\|^2_{H^s}
	+ C\|\mathcal{H}\theta_x\|_{L^{\infty}}\|\theta\|^2_{H^s}
	\lesssim\, (\|\theta_x\|_{L^{\infty}}+\|\mathcal{H}\theta_x\|_{L^{\infty}})\|\theta\|^2_{H^s}.\label{}
	\end{align*}
	Using Lemma \eqref{Liecancellations} with $\mathcal{P}= D^{s}T_{\e}\in \mbox{OPS}^{s}_{1,0}$ (cf. Lemma \ref{lemma;pseudo}), we have that
	\begin{align*}
	R_{2}+R_{3}=\left(D^{s}T_\e  \sum_{k=1}^{\infty}\mathcal{L}^2_{\xi_{k}} \theta, D^{s} T_\e \theta\right)_{L^2}
	+\sum_{k=1}^{\infty}
	\left(D^{s}T_\e  \mathcal{L}_{\xi_{k}} \theta, D^{s}T_\e \mathcal{L}_{\xi_{k}}\theta\right)_{L^2} \lesssim\,  \|\theta\|^2_{H^{s}}.
	\end{align*}
	Combining the above estimates,
	we find some $C>1$ such that,
	\begin{equation}
	2\left(T_\e  g(X),T_\e  X\right)_{\X^{s}} +\|T_\e  h(X)\|^2_{\LL_2(\U;\X^{s})} 
	\leq C (1+\|\theta_x\|_{L^{\infty}}+\|\mathcal{H}\theta_x\|_{L^{\infty}})\|\theta\|^2_{H^s}.
	\label{STNL:refined estimate 2}
	\end{equation}
	Due to $\V=H^r\hookrightarrow W^{1,\infty}$ and \eqref{H Hs},  \eqref{blow-up requirement 1} holds true.
	Therefore, we can apply Theorem \ref{Abstract irregular SDE: T} to obtain the existence, uniqueness of pathwise solutions, together with the blow-up criterion
	\begin{equation*}
	\textbf{1}_{\left\{\limsup_{t\rightarrow \tau^*}\|\theta(t)\|_{H^s}=\infty\right\}}
	=\textbf{1}_{\left\{\limsup_{t\rightarrow \tau^*}\|\theta(t)\|_{H^{r}}=\infty\right\}}\ \ \p-a.s.,
	\end{equation*}
	where $r\in(3/2,s-2)$ is arbitrary.
	Now we only need to improve the above blow-up criterion to \eqref{STNL blowup criterion}. To this end, we proceed as in the proof of \eqref{Blow-up criterion} (cf. \eqref{tau1=tau2}).
	For $m,l\in\N$, we define 
	\begin{align*}
	\sigma_{1,m}=\inf\left\{t\geq0: \|\theta(t)\|_{H^s}\geq m\right\},\ \ \
	\sigma_{2,l}=\inf\left\{t\geq0: \|\theta_x(t)\|_{L^{\infty}}+\|\mathcal{H}\theta_x\|_{L^{\infty}}\geq l\right\},
	\end{align*}
	where $\inf\emptyset=\infty$.
	Denote  $\sigma_1=\lim_{m\rightarrow\infty}\sigma_{1,m}$ and $\sigma_2=\lim_{l\rightarrow\infty}\sigma_{2,l}$. 
	Now we fix a $r\in(3/2,s-2)$. Then
	$$\|\theta_x(t)\|_{L^{\infty}}+\|\mathcal{H}\theta_x\|_{L^{\infty}}\lesssim\, \|\theta(t)\|_{H^r}\lesssim\, \|\theta(t)\|_{H^s}.$$
	From this, it is obvious that $\sigma_1\leq \sigma_2$ $\p-a.s.$
	To prove $\sigma_1=\sigma_2$ $\p-a.s.$, we need to prove $\sigma_{1}\geq \sigma_2$ $\p-a.s.$ In the same way as we prove \eqref{tau1=tau2}, we only need to prove
	\begin{align}
	\p\left\{\sup_{t\in[0,\sigma_{2,l}\wedge N]}\|\theta(t)\|_{H^s}<\infty\right\}=1
	\ \ \ \forall\ N,l\in\N.\label{s2<s1 condition}
	\end{align}
	%Recalling 
	%the following estimate (cf.\cite{})
	%	\begin{equation*}
	%\|\mathcal{H}\theta_x\|_{L^\infty}\lesssim\, 
	%\left(1+\|\theta_x\|_{L^\infty}\log({\rm e}+\|\theta\|_{H^2})+\|\theta_x\|_{L^2}\right).
	%\end{equation*}
	
	It follows from \eqref{STNL:refined estimate 1} and \eqref{STNL:refined estimate 2} that
	\begin{align*}
	&\hspace*{-0.5cm}E\sup_{t\in[0,\sigma_{2,l}\wedge N]}\|T_\e  \theta\|^2_{H^s}-\E\|T_\e  \theta_0\|^2_{H^s}\\
	\leq\,&C\E\left(\int_0^{\sigma_{2,l}\wedge N}
	\|\theta\|^2_{H^s}\|T_\e  \theta\|^2_{H^s}\dd t\right)^{\frac12}
	+C\E\int_0^{\sigma_{2,l}\wedge N}  (1+\|\theta_x\|_{L^{\infty}}+\|\mathcal{H}\theta_x\|_{L^{\infty}})\|\theta\|^2_{H^s}\,{\rm d}t\nonumber\\
	\leq\,&C\E\left(\sup_{t\in[0,\sigma_{2,l}\wedge N]}\|T_\e  \theta\|_{H^s}^{2}\int_0^{\sigma_{2,l}\wedge N} 
	\|\theta\|^2_{H^s}\dd t\right)^{\frac12}
	+C_{l}\E\int_0^{\sigma_{2,l}\wedge N} \|\theta\|_{H^s}^{2}\,{\rm d}t\notag\\
	\leq\,&\frac{1}{2}\E\sup_{t\in[0,\sigma_{2,l}\wedge N]}\|T_\e  \theta\|_{H^s}^{2}
	+C_{l}\int_0^{M} \E\sup_{t'\in[0,t\wedge \sigma_{2,l}]}\|\theta(t')\|_{H^s}^{2}\,{\rm d}t,
	\end{align*} 
	where $C_l=C(1+l)$ for some $C>1$ large enough. Therefore we arrive at
	\begin{align*}
	\E\sup_{t\in[0,\sigma_{2,l}\wedge N]}\|T_\e  \theta\|^2_{H^s}-2\E\|T_\e  \theta_0\|^2_{H^s} 
	\leq\,& C_{l}\int_0^{M} \E\sup_{t'\in[0,t\wedge \sigma_{2,l}]}\|\theta(t')\|_{H^s}^{2}\,{\rm d}t.
	\end{align*} 
	Hence one can send $\e\rightarrow0$ and then use  Gr\"{o}nwall's inequality to derive that for each $l,N\in\N$,
	$$\E\sup_{t\in[0,\sigma_{2,l}\wedge N]}\|\theta(t)\|^2_{H^s}
	\leq C\E\|\theta_0\|^2_{H^s}\exp\left(C_lN\right)<\infty,$$ 
	which is \eqref{s2<s1 condition}. Hence  we obtain \eqref{STNL blowup criterion} and finish the proof.
	
	\subsection{Further examples}\label{Sect:further examples}
	Actually, the abstract framework for \eqref{Abstact irregular SDE} can be applied to show the local existence theory to a broader class of fluid dynamics equations. For instance, consider the SALT surface quasi-geostrophic (SQG) equation:
	\begin{equation}\label{general CCF}
	\left\{\begin{aligned}
	&{\rm d}\theta + u\cdot\nabla\theta \ dt+ \sum_{k=1}^{\infty}( \xi_{k}\cdot\nabla\theta)\circ {\rm d}W_k=0,\ \ x\in\T^2, \\
	&u=\mathcal{R}^{\perp}\theta,
	\end{aligned} \right.
	\end{equation}
	where $\mathcal{R}$ is the 
	Riesz transform in $\T^2$, and 
	$\{W_t^k\}_{k\in\N}$ is a sequence of standard 1-D independent Brownian motions. The deterministic version of \eqref{general CCF} reduces to the SQG  equation describing
	the dynamics of sharp fronts
	between masses of hot and cold air (cf.~\cite{Constantin-Majda-Tabak-1994}). The SQG equations have been studied intensively, and we cannot survey the vast research literature
	here. However, the stochastic version with transport noise as in \eqref{general CCF} has not been studied yet as far as we know. 
	
	To apply Theorem \ref{Abstract irregular SDE: T} to \eqref{general CCF} to  get a local theory, we introduce some notations.
	For any real number $s$, $\Ls=(-\Delta)^{s/2}$ 
	are defined by
	$\widehat{\Ls f}(k)=|k|^s\widehat{f}(k)$. Then we let
	\begin{equation}
	\X^s=H^s\cap \left\{f:\int_{\T^2}f\ \dd x=0\right\}.\label{solution space SQG}
	\end{equation}
	We notice that with the mean-zero condition, $\X^s$ is Hilbert space for $s>0$ with inner product $(f,g)_{\X^s}=(\Ls f,\Ls g)_{L^2}$ and homogeneous Sobolev norm $\|f\|_{\X^s}=\|\Ls f\|_{L^2}$. However, it can be shown that if $f\in\X^s$ for $s>0$, then, cf. \cite{Bahouri-Chemin-Danchin-2011},
	\begin{equation}\label{Xs Hs mean-zero}
	\|f\|_{H^s}\lesssim\,\|f\|_{\X^s}\lesssim\,\|f\|_{H^s}.
	\end{equation}

	\begin{Assumption}\label{Assum-xi-div}
		For all $s>1$,
		$\{\xi_{k}(x):\T^2\rightarrow\R^2\}_{k\in\N}\subset H^s\cap \left\{f\in H^1:\nabla\cdot f=0\right\}$ and $\sum_{k\in\N}\|\xi_k\|_{H^s}<\infty$.
	\end{Assumption}
	
	Then we have the following local results for \eqref{general CCF}:
	\begin{Theorem}\label{SQG results}
		Let	$s>4$, $\s=(\Omega, \mathcal{F},\p,\{\mathcal{F}_t\}_{t\geq0}, \W)$ be a  stochastic basis fixed in advance and $\X^s$ be given in \eqref{solution space SQG}. Let Assumption \ref{Assum-xi-div} hold true.  If
		$\theta_0\in L^2(\Omega;\X^s)$ is an $\mathcal{F}_0$-measurable random variable, then \eqref{general CCF} has a local unique pathwise solution $\theta$ starting from $\theta_0$ such that 
		\begin{equation*} 
		\theta(\cdot\wedge \tau)\in L^2\left(\Omega; C\left([0,\infty);\X^s\right)\right).
		\end{equation*} 
		Moreover, the maximal solution $(\theta,\tau^*)$ to \eqref{general CCF} satisfies 
		\begin{equation*} 
		\textbf{1}_{\left\{\limsup_{t\rightarrow \tau^*}\|\theta(t)\|_{\X^s}=\infty\right\}}
		=\textbf{1}_{\left\{\limsup_{t\rightarrow \tau^*}\|\theta_x(t)\|_{L^{\infty}}+\|(\mathcal{R}\theta_x)(t)\|_{L^{\infty}}=\infty\right\}}\ \ \p-a.s.
		\end{equation*}
	\end{Theorem}
	
	\begin{proof}
		We only give a very quick sketch.  The  approximation of \eqref{general CCF} can be constructed as in the proof of Theorem \ref{transport:nonlocal results}. We only notice that if Assumption \ref{Assum-xi-div} is verified and $\theta_0$ has mean-zero, then the approximate solution $\theta_\e$ has also mean-zero. 
		Recalling that $\U$ is fixed in advance to define \eqref{define W},  we take
		$\X=\X^s$, $\Y=\X^{s-1}$, $\Z=\X^{s-2}$, $\V=\X^{r}$ with $2<r<s-2$ and $T_\e =Q_\e =T_\e $. One can basically go along the lines as in the proof of Theorem \ref{transport:nonlocal results} with using the $\Ls$-version of Lemma \ref{KP  commutator estimate} (see also in \cite{Kato-Ponce-1988-CPAM,Kenig-Ponce-Vega-1991-JAMS}) to estimate the nonlinear term.  For the noise term, after writing it into the It\^{o} form, one can use Lemma \ref{Liecancellations} and \eqref{Xs Hs mean-zero} to estimate the corresponding two terms. For the sake of brevity, we omit the  details.
	\end{proof} 
	
	\begin{Remark}
		If the relation $u=\mathcal{R}^{\perp}\theta$ in \eqref{general CCF} is replaced by $ u= \mathcal{R}^{\perp}\Lambda^{\alpha}u$ with $\alpha\in[-1,0]$, \eqref{general CCF} becomes a SALT 2-D Euler-$\alpha$ model in vorticity form, 
		which interpolates with the SALT 2-D Euler equations \cite{Crisan-Lang-2019} ($\alpha=-1$) and the SALT SQG equations ($\alpha=0$).  If $u=\mathcal{R}^{\perp}\mathcal{R}_{1}\theta$ in \eqref{general CCF}, then \eqref{general CCF} is the SALT incompressible porous medium equation, where $\theta$ is now explained as the density of the incompressible fluid moving through a homogeneous porous domain. For the deterministic incompressible porous medium equation, we refer to \cite{Castro-Cordoba-Gancedo-Orive-2009}. Both of them with SALT noise $\sum_{k=1}^{\infty} (\xi_{k}\cdot\nabla\theta)\circ {\rm d}W_k$ have not been studied. Similar to Theorem \ref{SCH2 results}, our general framework \ref{Abstract irregular SDE: P} is also applicable to them. 
	\end{Remark}
	
	\begin{Remark}\label{LA case remark}
		It is worthwhile remarking that, a new framework called Lagrangian-Averaged Stochastic
		Advection by Lie Transport (LA SALT) has been  developed for a class of stochastic partial differential equations in \cite{Alonso-etal-2020-JSP,Drivas-Leahy-Holm-2020}. For LA SALT  the velocity field is randomly transported by white-noise vector fields as well as by its own average over realizations of this noise.  For the even more general distribution-path dependent case of transport type equations, we refer to \cite{Ren-Tang-Wang-2020-Arxiv}. Generally speaking, the distribution of the solution is a global object on the path space, and it does not exist  for explosive stochastic processes whose paths are killed at the life time. For a local theory of distribution dependent SDEs/SPDEs, we have to either consider the non-explosive setting or modify the ``distribution" by a local notion (for example, conditional distribution given by solution does not blow up at present time). Here, we focus our attention to the abstract framework for SPDEs with SALT noise. The general case with LA SALT is left as future work. 
	\end{Remark}
	
	\section*{Acknowledgements}
	D.~Alonso-Or\'{a}n is deeply indebted to Antonio C\'{o}rdoba  for his helpful conversations about the theory of pseudo-differential operators. H. Tang benefited greatly from many insightful discussions with  Professor Feng-Yu Wang. 
	
	\appendix\section{Auxiliary results}\label{sec:appendix}	
	
	In this appendix we formulate and prove some estimates employed in the  proofs above. We start from mollifiers which can preserve periodicity. Let $j=j(x)$ be a Schwartz function such that $0\leq\widehat{j}(\xi)\leq1$ for all $\xi\in \R^d$ and $\widehat{j}(\xi)=1$ for any $|\xi|\leq1$. Define for $\e \in (0,1)$
	the mollifier 
	\begin{equation}\label{Define Je}
	J_{\varepsilon}g(x):=(j_{\varepsilon}\star g)(x),
	\end{equation}
	where $j_{\varepsilon}(x)=\frac{1}{(2\pi)^d}\sum_{k\in{\mathbb Z^d}}\widehat{j}(\varepsilon k){\rm e}^{{\rm i}x\cdot k}$.
	The following operator ${\tilde J}_\e $  is also fundamental for the approximation and  defined by
	\begin{equation}\label{Define Te}
	{\tilde J}_\e  g(x):=(1-\e^2 \Delta)^{-1}g(x)= \sum_{k\in\mathbb{Z}^d} \left(1+\e^2 |k|^2\right)^{-1}  \widehat{g}(k)\, {\rm e}^{{\rm i}x\cdot k}.
	\end{equation}
	
	For any $u,v\in H^s$, $J_\e$ and ${\tilde J}_\e $ satisfy, cf. \cite{Tang-2018-SIMA,Tang-2020-Arxiv},
	\begin{align}
	\|u-J_{\varepsilon}u\|_{H^r}&\sim o(\varepsilon^{s-r}),\ \ r\leq s,\label{Je r<s}\\
	\|J_{\varepsilon}u\|_{H^{r}}&\lesssim\,  \varepsilon^{s-r}\|u\|_{H^{s}},\ \ r> s,\label{Je r>s}\\
	[D^s,J_{\varepsilon}]=\,&[D^s,{\tilde J}_{\varepsilon}]=0,\label{Je Te Ds}\\
	(J_{\varepsilon}u, v)_{L^2}=(u, J_{\varepsilon}v)_{L^2},\ & ({\tilde J}_{\varepsilon}u, v)_{L^2}=(u, {\tilde J}_{\varepsilon}v)_{L^2},\label{Je Te self-adjoint} 
	\end{align}
	and
	\begin{align}
	\|J_{\varepsilon}u\|_{H^s},\|{\tilde J}_{\varepsilon}u\|_{H^s}&\leq \|u\|_{H^s}.\label{Je Te Hs} 
	\end{align}
	From the definition of the Hilbert transform $\mathcal{H}$ in \eqref{define Hilbert transform},  we have
	\begin{equation}\label{H Ds}
	[D^s,\mathcal{H}]=[\partial_x,\mathcal{H}]=[J_\e,\mathcal{H}]=0,
	\end{equation}
	and for any $s\geq 0$,
	\begin{equation}\label{H Hs}
	\|\mathcal{H} u\|_{H^s}\lesssim\, \|u\|_{H^s}.
	\end{equation}
	
	A pseudo-differential operator  $P(x,D)$ on the periodic torus $\mathbb{T}^d$ is an operator given by
	\begin{equation}\label{def:pseudodiff}
	p(x,D)f(x)=\frac{1}{(2\pi)^d}\displaystyle\sum_{k\in \mathbb{Z}^d} a(x,k) e^{ix\cdot k} \widehat{f}(k),
	\end{equation}
	where $P(x,D)$ belongs to a certain class and $a(x,k)$ is called the symbol of $P(x,D)$. For  $\rho,\delta\in [0,1]$, $s \in \mathbb{R}$, we define the H\"{o}rmander class of symbols $S^{m}_{\rho,\delta}$ to be the set of all symbols $a:\mathbb{T}^d\times \mathbb{Z}^d \to \mathbb{C}$ such that $a(\cdot,k)\in C^{\infty}(\mathbb{T}^d)$ for all $k\in \mathbb{Z}^d$ and for all $\alpha,\beta\in\mathbb{N}^{d}$, there exists a constant $C=C(\alpha,\beta)>0$ such that
	$$\abs{\Delta^{\alpha}_{k}\partial^{\beta}_{x} a(x,k)}\leq C \langle k \rangle^{s-\rho\abs{\alpha}+\delta\abs{\beta}},$$
	where $\langle k \rangle=(1+k^2)^{1/2}$ and for $g:\mathbb{Z}^{d}\to \mathbb{C}$,
	$$\Delta^{\alpha}_{k}g(k):=\displaystyle\sum_{\gamma\in\mathbb{N}^{d}, \gamma\leq \alpha} (-1)^{\abs{\alpha-\gamma}}\binom{\alpha}{\gamma}g(k+\gamma) $$
	is the finite difference operator of order $\alpha$ with step size one in each of the coordinates of the frequency variable $k$. In such a case we say the associated operator $p(x,D)$ defined by \eqref{def:pseudodiff} belongs to the class $\mbox{OPS}^{s}_{\rho,\delta}$.
	Then $J_\e$ and ${\tilde J}_\e $ also satisfy% the following properties as  pseudo-differential operators: 
	\begin{Lemma}[\cite{Hormander-1985-book-3,Taylor-1991-book}]\label{lemma;pseudo}
		Let $J_\e, {\tilde J}_\e $ be defined as in \eqref{Define Je} and \eqref{Define Te}, then the following properties hold true
		\begin{enumerate}
			\item \label{Prop1}  $J_{\e}\in \mbox{OPS}^{-\infty}_{1,0}$, ${\tilde J}_\e \in \mbox{OPS}^{-2}_{1,0}$ for every $\e\in(0,1)$;\\
			\item  \label{Prop2}   $\{J_{\e}\}_{0<\e <1}$ and $\left\{{\tilde J}_{\e}\right\}_{0<\e<1}$ are bounded subsets of  $\mbox{OPS}^{0}_{1,0}$;\\
			\item \label{Prop3}  If $p(x,D)\in \mbox{OPS}^{s}_{1,0}$, then  $p(x,D)J_{\e}\in\mbox{OPS}^{-\infty}_{1,0}$, $p(x,D){\tilde J}_{\e}\in\mbox{OPS}^{-\infty}_{1,0}$ for all $\e\in(0,1)$;\\
			\item \label{Prop4} If $p(x,D)\in \mbox{OPS}^{s}_{1,0}$, then
			$\left\{p(x,D)J_{\e}\right\}_{0<\e < 1}\subset\mbox{OPS}^{s}_{1,0}$ and $\left\{p(x,D){\tilde J}_{\e}\right\}_{0<\e< 1}\subset\mbox{OPS}^{s}_{1,0}$ are bounded.
		\end{enumerate}
	\end{Lemma}
	We also recall the following commutator estimates for two pseudo-differential operators.
	\begin{Lemma}[\cite{Hormander-1985-book-3,Taylor-1991-book}]\label{lemma:commuta}
		Let $\mathcal{P}\in \mbox{OPS}^{p}_{\rho,\delta}$ and $\mathcal{T}\in \mbox{OPS}^{q}_{\rho,\delta}$ with $p,q\in \mathbb{R}$, $0 \leq \delta<\rho, \leq 1$ then 
		$$ [\mathcal{P},\mathcal{T}]\in \mbox{OPS}^{p+q-(\rho-\delta)}_{\rho,\delta}.$$
	\end{Lemma}

	%For ${\tilde J}_\e $, we have the following commutator estimate:
	\begin{Lemma}[\cite{Tang-2020-Arxiv,Ren-Tang-Wang-2020-Arxiv}]\label{Te commutator} 
		Let $d\geq1$ and $f,g:\T^d\rightarrow\R^d$ such that $g\in W^{1,\infty}$ and $f\in L^2$. Then for some $C>0$,
		\begin{align*}
		\left\|\left[{\tilde J}_{\varepsilon}, (g\cdot \nabla)\right]f\right\|_{L^2}
		\leq C\|\nabla g\|_{L^\infty}\|f\|_{L^2}.
		\end{align*}
	\end{Lemma}
	
	Now we recall some useful estimates.
	
	\begin{Lemma}[\cite{Kato-Ponce-1988-CPAM,Kenig-Ponce-Vega-1991-JAMS}]\label{KP  commutator estimate}
		If $f,g\in H^s\bigcap W^{1,\infty}$ with $s>0$, then for $p,p_i\in(1,\infty)$ with $i=2,3$ and
		$\frac{1}{p}=\frac{1}{p_1}+\frac{1}{p_2}=\frac{1}{p_3}+\frac{1}{p_4}$, we have
		$$
		\|\left[D^s,f\right]g\|_{L^p}\leq C(\|\nabla f\|_{L^{p_1}}\|D^{s-1}g\|_{L^{p_2}}+\|D^sf\|_{L^{p_3}}\|g\|_{L^{p_4}}),$$
		and
		$$\|D^s(fg)\|_{L^p}\leq C(\|f\|_{L^{p_1}}\|D^s g\|_{L^{p_2}}+\|D^s f\|_{L^{p_3}}\|g\|_{L^{p_4}}).$$
	\end{Lemma}
	
	\begin{Lemma} \label{Liecancellations}
		Let $s>\frac{d}{2}+1$,  $f\in H^{s+2}$ be a scalar function,  $\xi_k$ be a $d$-$D$ vector and $\mathcal{P}\in \mbox{OPS}^{s}_{1,0}$.  Define
		$$\mathcal{L}_{\xi_k}f=\xi_k\cdot\nabla f+({\rm div}\xi_k)f.$$ 
		If	Assumption \ref{Assum-xi} holds, then we have
		\begin{equation}\label{Liecancellation-target}
		\left( \mathcal{P} \sum_{k=1}^{\infty}\mathcal{L}^2_{\xi_k} f, \mathcal{P} f \right)_{L^2} 
		+  
		\sum_{k=1}^{\infty}\left (\mathcal{P}  \mathcal{L}_{\xi_k} f, \mathcal{P}   \mathcal{L}_{\xi_k} f \right) _{L^2} \lesssim\, \|f\|_{H^s}^2. 
		\end{equation}
	\end{Lemma}
	\begin{proof}
		The essential part of the desired estimate lies in the following result in \cite{Alonso-Bethencourt-2020-JNS}:
		Let $\mathcal{Q}$ be a first-order linear operator with smooth coefficients and $\mathcal{P}\in \mbox{OPS}^{s}_{1,0}$. Then $f\in H^{s}$ with $s>\frac{d}{2}+1$ we have that
		\begin{equation*} 
		\left( \mathcal{P}\mathcal{Q}^{2}f, \mathcal{P} f \right)_{L^2} +  \left (\mathcal{P} \mathcal{Q} f, \mathcal{P} \mathcal{Q} f \right) _{L^2} \lesssim\,  \norm{f}_{H^s}^2. 
		\end{equation*}
		In particular, if we choose $\mathcal{Q}=\mathcal{L}_{\xi_{k}}$ we have that:
		\begin{equation}\label{eq:cancellation:Lie}
		\left( \mathcal{P}\mathcal{L}_{\xi_{k}}^{2}f, \mathcal{P} f \right)_{L^2} +  \left (\mathcal{P} \mathcal{L}_{\xi_{k}} f, \mathcal{P} \mathcal{L}_{\xi_{k}} f \right) _{L^2} \lesssim\,  \norm{f}_{H^s}^2. 
		\end{equation}
		Since we want to calculate this estimate for $\sum_{k=1}^{\infty}\mathcal{L}^2_{\xi_k} $, we need to precise the constant of the right hand side of
		\eqref{eq:cancellation:Lie}. To this end, mimicking the proof of \cite{Alonso-Bethencourt-2020-JNS} we can rewrite the left hand side of \eqref{eq:cancellation:Lie} as
		\begin{align*}
		\left( \mathcal{P}\mathcal{L}_{\xi_{k}}^{2}f, \mathcal{P} f \right)_{L^2} +  \left (\mathcal{P} \mathcal{L}_{\xi_{k}} f, \mathcal{P} \mathcal{L}_{\xi_{k}} f \right) _{L^2} =\,& \left( R_{2}f, \mathcal{P} f \right)_{L^2} + \left( R_{1}f, R_{1} f \right)_{L^2} + \left(\mathcal{P} f, ER_{1}f \right)_{L^2} \\
		&-\frac{1}{2}\left( \mathcal{P}f, R_{0} \mathcal{P}f \right)_{L^2}+\frac{1}{2}\left( \mathcal{P}f, E^{2} \mathcal{P}f \right)_{L^2} +\left( R_{1}f, E \mathcal{P}f \right)_{L^2} \\
		=:\,& \displaystyle\sum_{i=1}^{6} I_{i},
		\end{align*}
		where $E={\rm div}\xi_{k} \in\mbox{OPS}^{0}_{1,0}$, $R_{0}=[\mathcal{L}_{\xi_{k}},E]\in\mbox{OPS}^{1}_{1,0}$, $R_{1}=[\mathcal{P},\mathcal{L}_{\xi_{k}}]$ and $R_{2}=[R_{1},\mathcal{L}_{\xi_{k}}]$.  	By Lemma \ref{lemma:commuta}, we have $$R_{1},\ R_2,\  [R_{1},\partial_{x}]\in \mbox{OPS}^{s}_{1,0}.$$

		To derive \eqref{Liecancellation-target} we will invoke the following commutator estimates (see \cite[ (3.6.1) and (3.6.2)]{Taylor-1991-book}):
		\begin{itemize}
			\item If  $P\in \mbox{OPS}^{s}_{1,0}, s>0$, then there is a $C>0$ such that
			\begin{equation}\label{commutator:Taylor s>0}
			\norm{P(gu)-gPu}_{L^{2}} \leq C \left(\norm{g}_{W^{1,\infty}}\norm{u}_{H^{s-1}}+\norm{g}_{H^{s}}\norm{u}_{L^{\infty}} \right).
			\end{equation}
			
			\item If  $P\in \mbox{OPS}^{1}_{1,0}$, then there is a $C>0$ such that
			\begin{equation}\label{commutator:Taylor 1}
			\norm{P(gu)-gPu}_{L^{2}} \leq C \norm{g}_{W^{1,\infty}}\norm{u}_{H^{s-1}}.
			\end{equation}
		\end{itemize}
		
		For $I_1$, we have that
		\begin{align*}
		\abs{I_{1}} \leq \norm{R_{2}f}_{L^{2}}\norm{Pf}_{L^{2}}
		\leq & \norm{[R_{1},\mathcal{L}_{\xi_{k}}]f}_{L^2}\norm{f}_{H^s} \\
		=\,& \left(\norm{[R_{1},\xi_{k}\cdot\nabla]f}_{L^{2}}+\norm{[R_{1},{\rm div}\xi_{k}]f}_{L^2}\right)\norm{f}_{H^s} \\
		=\,&  \left(\norm{[R_{1},\xi_{k}\cdot]\nabla f}_{L^{2}}+\norm{\xi_{k}\cdot[R_{1},\nabla]f}_{L^{2}} +\norm{[R_{1},{\rm div}\xi_{k}]f}_{L^2}\right)\norm{f}_{H^s} \\
		=\,& \left(I_{1,1}+I_{1,2}+I_{1,3}\right)\norm{f}_{H^s}
		\end{align*}
		Applying \eqref{commutator:Taylor s>0} with $P=R_{1}$, $g=\xi_{k}$, $u=\nabla f$, and using $H^s\hookrightarrow W^{1,\infty}$, we arrive at
		$$ \abs{I_{1,1}}\leq \norm{\xi_{k}}_{W^{1,\infty}}\norm{\nabla f}_{H^{s-1}}+ \norm{\xi_{k}}_{H^{s}}\norm{\nabla f}_{L^{\infty}} \leq \norm{\xi_{k}}_{H^{s}}\norm{f}_{H^s}.$$
		For the second term, we have
		$$\abs{I_{1,2}}=\norm{\xi_{k}\cdot[R_{1},\nabla]f}_{L^{2}}\leq \norm{\xi_{k}}_{L^{\infty}}\norm{[R_{1},\nabla]f}_{L^{2}} \leq \norm{\xi_{k}}_{H^{s}} \norm{f}_{H^s}.$$
		Applying \eqref{commutator:Taylor s>0} with $P=R_{1}$, $g={\rm div}\xi_{k}$ and $u=f$  yields
		$$ \abs{I_{1,3}}\leq \norm{{\rm div}\xi_{k}}_{W^{1,\infty}}\norm{f}_{H^{s-1}}+ \norm{{\rm div}\xi_{k}}_{H^{s}}\norm{f}_{L^{\infty}} \leq \norm{\xi_{k}}_{H^{s+1}}\norm{f}_{H^s}.$$
		Hence, we have show that 
		$$\abs{I_{1}}\leq C \norm{\xi_{k}}_{H^{s+1}}\norm{f}^{2}_{H^{s}}.$$
		Repeat the above procedure as we estimate $\|R_2f\|_{L^2}=\norm{[R_{1},\LL_{\xi_{k}}]f}_{L^{2}}$ with replacing $R_1$ by $\mathcal{P}$, we have
		\begin{align*}
		\abs{I_{2}} \leq \norm{R_{1}f}^{2}_{L^{2}}\leq  \norm{[\mathcal{P},\mathcal{L}_{\xi_{k}}]f}^{2}_{L^2}=\,& \left(\norm{[\mathcal{P},\xi_{k}\cdot\nabla]f}_{L^{2}}+\norm{[\mathcal{P},{\rm div}{\xi_{k}}]f}_{L^2}\right)^{2} \\
		=\,&  \left(\norm{[\mathcal{P},\xi_{k}\cdot]\nabla f}_{L^{2}}+\norm{\xi_{k}\cdot[\mathcal{P},\nabla]f}_{L^{2}} +\norm{[\mathcal{P},{\rm div}\xi_{k}]f}_{L^2}\right)^{2} \\
		\leq & \norm{\xi_{k}}^{2}_{H^{s+1}}\norm{f}^{2}_{H^s},
		\end{align*}
		For the third term, using the Cauchy-Schwarz inequality and the fact that $E={\rm div}\xi_{k}\in\mbox{OPS}^{1}_{1,0}$ gives rise to
		\[ \abs{I_{3}} = \left(\mathcal{P} f, ER_{1}f \right)_{L^2} \leq \norm{\mathcal{P} f}_{L^2}\norm{{\rm div}\xi_{k}R_{1}f}_{L^2} \leq \norm{{\rm div}\xi_{k}}_{L^{\infty}} \norm{f}^{2}_{H^s}. \] 
		Similarly, 
		\begin{align*}
		\abs{I_{5}+I_{6}} =\,& \left|\frac{1}{2}\left(\mathcal{P} f, E^2 \mathcal{P}  f \right)_{L^2} + \left(R_{1}f, E\mathcal{P} f \right)_{L^2} \right|\\
		\leq & C\left(\norm{{\rm div}\xi_{k}}^{2}_{L^{\infty}} \norm{\mathcal{P}f}_{L^2}^{2}+\norm{R_{1}f}_{L^2}\norm{{\rm div}\xi_{k}}_{L^{\infty}}\norm{\mathcal{P} f }_{L^2}\right) \\
		\leq & C\left( \norm{{\rm div}\xi_{k}}_{L^{\infty}}+ \norm{{\rm div}\xi_{k}}^2_{L^{\infty}}\right)\norm{f}^{2}_{H^2}.
		\end{align*} 
		For $I_4$, we notice that $\mathcal{L}_{\xi_{k}}\in\mbox{OPS}^{1}_{1,0}.$ Hence it follows from \eqref{commutator:Taylor 1} with $P=\mathcal{L}_{\xi_{k}}$, $g={\rm div}\xi_{k}$ and $u=\mathcal{P}f$ that
		$$ \abs{I_{4}}\leq C\norm{\mathcal{P}f}_{L^2}\norm{[\LL_{\xi_k},{\rm div}\xi_{k}]\mathcal{P}f}_{L^2} \leq C \norm{f}_{H^s} \norm{{\rm div}\xi_{k}}_{W^{1,\infty}} \norm{f}_{H^{s-1}} \leq C  \norm{\xi_{k}}_{H^{s+1}}\norm{f}^{2}_{H^{s}}.$$ 
		Gathering all the above estimates implies that for some $C>0$,
		\begin{equation*}\label{eq:cancellation:Lie:final}
		\left( \mathcal{P}\mathcal{L}_{\xi_{k}}^{2}f, \mathcal{P} f \right)_{L^2} +  \left (\mathcal{P} \mathcal{L}_{\xi_{k}} f, \mathcal{P} \mathcal{L}_{\xi_{k}} f \right) _{L^2} \leq C\left( \norm{\xi_{k}}^{2}_{H^{s+1}}+\norm{\xi_{k}}_{H^{s+1}} \right)\norm{f}_{H^s}^2. 
		\end{equation*}
		Using Assumption \ref{Assum-xi} to the above estimates, we obtain  \eqref{Liecancellation-target}.
	\end{proof}

	We conclude this appendix with some useful tools in stochastic analysis.
	
	\begin{Lemma}[Prokhorov Theorem, \cite{Prato-Zabczyk-2014-Cambridge}]\label{Prokhorov Theorem}
		Let $\mathbb{X}$ be a complete and separable metric space.  A sequence of measures $\{\mu_n\}\subset\Pm{}(\mathbb{X})$ is tight if and only if it is relatively compact, i.e.,  there is a subsequence $\{\mu_{n_k}\}$  converging to a probability measure $\mu$ weakly.
	\end{Lemma}
	
	\begin{Lemma}[Skorokhod Theorem, \cite{Prato-Zabczyk-2014-Cambridge}]\label{Skorokhod Theorem}
		Let $\mathbb{X}$ be a complete and separable metric space. For an arbitrary sequence $\{\mu_n\}\subset\Pm{}(\mathbb{X})$ such that $\{\mu_n\}$ is tight on $(\mathbb{X},\B(\mathbb{X}))$, there exists a subsequence $\{\mu_{n_k}\}$  converging weakly to a probability measure $\mu$, and a probability space $(\Omega, \mathcal{F},\p)$ with $\mathbb{X}$-valued Borel measurable random variables $x_n$ and $x$, such that $\mu_n$ is the distribution of $x_n$, $\mu$ is the distribution of $x$, and
		$x_n\xrightarrow{n\rightarrow\infty} x$ $\p-a.s.$
	\end{Lemma}
	
	\begin{Lemma}[\cite{Breit-Feireisl-Hofmanova-2018-Book,Debussche-Glatt-Temam-2011-PhyD}]\label{stochastic integral convergence}
		Let $(\Omega, \mathcal{F},\p)$ be a complete probability space and $\mathbb{X}$  be a separable Hilbert space. Let $\s_n=\left(\Omega, \mathcal{F},\left\{\mathcal{F}_{t}^{n}\right\}_{t \geq 0}, \mathbb{P},\W_n\right)$ be a sequence of stochastic bases 
		such that for each $n\ge1$,  $\W^{n}$ is cylindrical Brownian motion $($over $\U$ with the canonical embedding $\U\hookrightarrow \U_0$ being Hilbert--Schmidt$)$ with respect to $\left\{\mathcal{F}_{t}^{n}\right\}_{t \geq 0}$. Let $G_{n}$ be an $\mathcal{F}_{t}^{n}$ predictable process ranging in $\LL_{2}(\U; \mathbb{X}).$ Finally consider  
		$\s=\left(\Omega, \mathcal{F},\p,\{\mathcal{F}_t\}_{t\geq0}, \W\right)$ and $G \in L^{2}\left(0, T; \LL_{2}(\U; \mathbb{X})\right)$, which is $\mathcal{F}_t$ predictable. 
		Suppose that in probability we have
		\begin{equation*}
		\W_{n} \rightarrow \W  \text { in } C\left([0, T] ;\U_{0}\right)\ \text {and}\ \ 
		G_{n} \rightarrow G  \text { in } L^{2}\left(0, T; \LL_{2}(\U; \mathbb{X})\right).
		\end{equation*}
		Then
		$$
		\int_{0}^\cdot G_{n} \mathrm{d} \W_{n} \rightarrow \int_{0}^\cdot G \mathrm{d} \W \quad 
		\text { in } L^{2}(0, T ; \mathbb{X}) \text { in probability. }
		$$

	\end{Lemma}
	
	\begin{Lemma}[Gy\"{o}ngy-Krylov Lemma, \cite{Gyongy-Krylov-1996-PTRF}]\label{Gyongy convergence}
		Let $\mathbb{X}$ be a Polish space equipped with the Borel sigma-algebra $\B(\mathbb{X})$. Let $\{Y_j\}_{j\geq0}$ be a sequence of $\mathbb{X}$-valued random variables. Let 
		$$\mu_{j,l}(\cdot):=\p(Y_j\times Y_l\in\cdot) \ \ \ \forall \cdot\in\B(\mathbb{X}\times \mathbb{X}).$$ Then $\{Y_j\}_{j\geq0}$ converges in probability if and only if for every subsequence of $\{\mu_{j_k,l_k}\}_{k\geq0}$, there exists a further subsequence which weakly converges to some $\mu\in \Pm{}(\mathbb{X}\times \mathbb{X})$ satisfying
		$$\mu\left(\left\{(u,v)\in \mathbb{X}\times \mathbb{X},\ u=v\right\}\right)=1.$$
	\end{Lemma}


\begin{thebibliography}{10}
		
		\bibitem{Albeverio-etal-2019-Arxiv}
		S.~Albeverio, Z.~Brzeźniak, and A.~Daletskii.
		\newblock Stochastic {C}amassa-{H}olm equation with convection type noise.
		\newblock {\em arXiv:1911.07077}, 2019.
		
		\bibitem{Alonso-Bethencourt-2020-JNS}
		D.~Alonso-Or\'{a}n and A.~Bethencourt~de Le\'{o}n.
		\newblock On the well-posedness of stochastic {B}oussinesq equations with
		transport noise.
		\newblock {\em J. Nonlinear Sci.}, 30(1):175--224, 2020.
		
		\bibitem{Alonso-etal-2019-NODEA}
		D.~Alonso-Or\'{a}n, A.~Bethencourt~de Le\'{o}n, and S.~Takao.
		\newblock The {B}urgers' equation with stochastic transport: shock formation,
		local and global existence of smooth solutions.
		\newblock {\em NoDEA Nonlinear Differential Equations Appl.}, 26(6):Paper No.
		57, 33, 2019.
		\bibitem{Alonso-etal-2020-JSP}
		D.~Alonso-Or\'{a}n, A.~Bethencourt~de Le\'{o}n, D.~D. Holm and S.~Takao.
		\newblock Modelling the Climate and Weather of a 2-D Lagrangian-Averaged Euler–Boussinesq Equation with Transport Noise.
		\newblock {\em Journal of Statistical Physics}, 179, 1267-–1303, 2020.
		
		
		
		
		\bibitem{Baker-Li-Morlet-1996}
		G. R. Baker, X. Li, and A. C. Morlet
		\newblock Analytic structure of 1D-transport equations
		with nonlocal fluxes
		\newblock {\em Physica D}, 91: 349--375, 1996.	
		
		
		\bibitem{Bae-Granero-2015-AM}
		H.~Bae and R.~Granero-Belinch\'{o}n.
		\newblock Global existence for some transport equations with nonlocal velocity.
		\newblock {\em Adv. Math.}, 269:197--219, 2015.
		
		\bibitem{Bahouri-Chemin-Danchin-2011}
		H.~Bahouri, J. Chemin and R.Danchin 
		\newblock Fourier analysis and nonlinear partial differential equations.
		\newblock {\em Grundlehren der Mathematischen Wissenschaften }, vol. 343, Springer, Heidelberg 2011.
		
		\bibitem{Bendall-Cotter-Holm-2019}
		T.~Bendall, C.~Cotter and D.~D. Holm
		\newblock Perspectives on the Formation of Peakons in the
		Stochastic Camassa-Holm Equation
		\newblock {\em arXiv:1910.03018v}, 2019.
		
		\bibitem{Bensoussan-1995-AAM}
		A.~Bensoussan.
		\newblock Stochastic {N}avier-{S}tokes equations.
		\newblock {\em Acta Appl. Math.}, 38(3):267--304, 1995.
		
		
		
		\bibitem{Berner-Jung-Palmer-12}
		J.~Berner, T.~Jung, and T.~N. Palmer.
		\newblock Systematic model error: The impact of increased
		horizontal resolution versus improved stochastic and deterministic parameterizations.
		\newblock {\em Journal of
			Climate,}  25(14):4946–4962, 2012.
		
		%		\bibitem{Takao-Bethen-2019}
		%	A.~Bethencourt and S.~Takao
		%	\newblock Well-posedness by noise for linear advection of $k$-forms.
		%	\newblock {\em 	arXiv:1904.13319}, 2019
		
		
		
		%		\bibitem{Bertozzi-Majda-2002}
		%	A.L.~Bertozzi and A.J.~Majda
		%	\newblock Vorticity and the Mathematical Theory of Incompressible Fluid Flow
		%	\newblock {\em Cambridge Univ. Press, Cambridge, UK}, 2002
		
		
		
		
		
		\bibitem{Breit-Feireisl-Hofmanova-2018-Book}
		D.~Breit, E.~Feireisl, and M.~Hofmanov\'{a}.
		\newblock {\em Stochastically forced compressible fluid flows}, volume~3 of
		{\em De Gruyter Series in Applied and Numerical Mathematics}.
		\newblock De Gruyter, Berlin, 2018.
		
		\bibitem{Bresan-Constantin-2007}
		A.~Bressan and A.~Constantin.
		\newblock Global conservative solutions of the Camassa-Holm equation.
		\newblock {\em Arch. Ration. Mech. Anal.}, 183(2):215–239, 2007. 
		
		\bibitem{Castro-Cordoba-Gancedo-Orive-2009}
		A.~Castro, D.~C\'ordoba, F.~Gancedo and R.~Orive.
		\newblock Incompressible flow in porous media with fractional diffusion,
		\newblock {\em Nonlinearity}, 22 (8), 1791–1815, 2009.
		
		\bibitem{Camassa-Holm-1993-PRL}
		R.~Camassa and D.~D. Holm.
		\newblock An integrable shallow water equation with peaked solitons.
		\newblock {\em Phys. Rev. Lett.}, 71(11):1661--1664, 1993.
		
		%		\bibitem{Chae-et-al-2012}
		%		D. Chae, P. Constantin, D. C\'ordoba, F. Gancedo and J. Wu,
		%		\newblock Generalized surface quasi-geostrophic equations with singular velocities.
		%		\newblock {\em  Comm. Pure Appl. Math.}, 65 (8):1037–1066, 2012.
		
		\bibitem{Chen-Gao-Guo-2012-JDE}
		Y.~Chen, H.~Gao, and B.~Guo.
		\newblock Well-posedness for stochastic {C}amassa-{H}olm equation.
		\newblock {\em J. Differential Equations}, 253(8):2353--2379, 2012.
		
		%		\bibitem{Chen-et-al-1998}
		%		S. Chen, C. Foias, D. D. Holm, E. Olson, E.S. Titi, and S. Wynne
		%		\newblock Camassa-Holm equations
		%		as a closure model for turbulent channel and pipe flow.
		%		\newblock {\em Physical Review Letters}, 81(24):5338, 1998.
		
		\bibitem{Constantin-Escher-1998}
		A.~Constantin and J.~Escher
		\newblock Wave breaking for nonlinear nonlocal shallow water equations.
		\newblock {\em Acta Math.}, 181(2):229–243,
		1998.
		\bibitem{Constantin-Escher-1998-2}
		A.~Constantin and J.~Escher
		\newblock  Well-posedness, global existence, and blowup phenomena for a periodic quasi-linear hyperbolic equation.
		\newblock {\em Comm. Pure Appl. Math.}, , 51(5):475–504, 1998.
		
		
		\bibitem{Constantin-Majda-Tabak-1994}
		P.~Constantin, A.~Majda and E.~Tabak.
		\newblock  Formation of strong fronts in the 2-D quasi-geostrophic thermal active
		scalar
		\newblock {\em Nonlinearity}, 7, 1495–1533, 1994.	 
		
		
		
		\bibitem{Cordoba-etal-2005-Annals}
		A.~C\'{o}rdoba, D.~C\'{o}rdoba, and M.~A. Fontelos.
		\newblock Formation of singularities for a transport equation with nonlocal
		velocity.
		\newblock {\em Ann. of Math. (2)}, 162(3):1377--1389, 2005.
		
		%		\bibitem{Cotter-etal-2018}
		%		C.~Cotter, D.~Crisan, D.~D. Holm, W.~Pan and I.~Shevchenko.	\newblock Modelling uncertainty using
		%		circulation-preserving stochastic transport noise in a 2-layer quasi-geostrophic model.
		%		\newblock {\em Foundations of Data Science, AIMS}, 2(2): 173--205, 2020.
		
		\bibitem{Cotter-etal-2019}
		C.~Cotter, D.~Crisan, D.~D. Holm, W.~Pan and I.~Shevchenko.	\newblock A Particle Filter for
		Stochastic Advection by Lie Transport (SALT): A case study for the damped and forced incompressible
		2-D Euler equation
		\newblock {\em arXiv:1907.11884 [stat.AP] }, 2019.
		
		%		\bibitem{Cotter-etal-2019-2}
		%		C.~Cotter, D.~Crisan, D.~D. Holm, W.~Pan and I.~Shevchenko.	\newblock Numerically modelling
		%		stochastic Lie transport in fluid dynamics.
		%		\newblock {\em  Multiscale Model. Simul. }, 17:192-232, 2019.
		
		%		\bibitem{Cotter-etal-2020}
		%		C.~Cotter, D.~Crisan, D.~D. Holm, W.~Pan and I.~Shevchenko.	\newblock Data Assimilation for a Quasi-Geostrophic Model with
		%		Circulation-Preserving Stochastic Transport Noise.
		%		\newblock {\em  Journal of Statistical Physics. }, 179:1186--1221, 2020.
		
		
		\bibitem{Cotter-etal-2017}
		C.~Cotter, G.~Gottwald and D.~D. Holm.
		\newblock Stochastic partial differential fluid
		equations as a diffusive limit of deterministic lagrangian multi-time dynamics.
		\newblock {\em  Proc. R. Soc.
			A, }, 473(2205):20170388, 2017.
		
		
		
		\bibitem{Crisan-Flandoli-Holm-2018-JNS}
		D.~Crisan, F.~Flandoli, and D.~D. Holm.
		\newblock Solution {P}roperties of a 3{D} {S}tochastic {E}uler {F}luid
		{E}quation.
		\newblock {\em J. Nonlinear Sci.}, 29(3):813--870, 2019.
		
		\bibitem{Crisan-Lang-2019}
		D.~Crisan and O.~Lang.
		\newblock Well-posedness for a stochastic 2-D Euler equation with transport noise
		\newblock {\em arXiv :1907.00451}, 2019.
		
		%		\bibitem{Crisan-Lang-2020}
		%		D.~Crisan and O.~Lang.
		%		\newblock Local well-posedness for the great lake equation with transport noise
		%		\newblock {\em arXiv:2003.03357 }, 2020.
		
		\bibitem{Crisan-Holm-2018}
		D.~Crisan and D.~Holm.
		\newblock Wave breaking for the stochastic Camassa-Holm equation,
		\newblock {\em Physica D: Nonlinear Phenomena }, vol. 376, pp. 138–143, 2018.
		
		\bibitem{Prato-Zabczyk-2014-Cambridge}
		G.~Da~Prato and J.~Zabczyk.
		\newblock {\em Stochastic equations in infinite dimensions}, volume 152 of {\em
			Encyclopedia of Mathematics and its Applications}.
		\newblock Cambridge University Press, Cambridge, second edition, 2014.
		
		\bibitem{Debussche-Glatt-Temam-2011-PhyD}
		A.~Debussche, N.~E. Glatt-Holtz, and R.~Temam.
		\newblock Local martingale and pathwise solutions for an abstract fluids model.
		\newblock {\em Phys. D}, 240(14-15):1123--1144, 2011.
		
		
		
		\bibitem{Dong-2008-JFA}
		H.~Dong.
		\newblock Well-posedness for a transport equation with nonlocal velocity.
		\newblock {\em J. Funct. Anal.}, 255(11):3070--3097, 2008.
		
		
		
		%		\bibitem{Drivas-Holm-2019}
		%		T.~Drivas and D.~D. Holm
		%		\newblock Circulation and energy theorem preserving stochastic fluids.
		%		\newblock {\em J. Proceedings of the Royal Society of Edinburgh: Section A Mathematics}, 1--39, 2019.
		%		
		
		\bibitem{Drivas-Leahy-Holm-2020}
		T.~Drivas, D.~D. Holm and J.~M.Leahy,
		\newblock Lagrangian Averaged Stochastic Advection by Lie Transport for Fluids.
		\newblock {\em Journal of Statistical Physics} 179, 1304–-1342, 2020.
		
		\bibitem{Fedrizzi-Flandoli-2013-JFA}
		E.~Fedrizzi and F.~Flandoli.
		\newblock Noise prevents singularities in linear transport equations.
		\newblock {\em J. Funct. Anal.}, 264(6):1329--1354, 2013.
		
		\bibitem{Flandoli-2011-book}
		F.~Flandoli.
		\newblock {\em Random perturbation of {PDE}s and fluid dynamic models}, volume
		2015 of {\em Lecture Notes in Mathematics}.
		\newblock Springer, Heidelberg, 2011.
		\newblock Lectures from the 40th Probability Summer School held in Saint-Flour,
		2010, \'{E}cole d'\'{E}t\'{e} de Probabilit\'{e}s de Saint-Flour.
		[Saint-Flour Probability Summer School].
		
		\bibitem{Flandoli-Gubinelli-Priola-2010-Invention}
		F.~Flandoli, M.~Gubinelli, and E.~Priola.
		\newblock Well-posedness of the transport equation by stochastic perturbation.
		\newblock {\em Invent. Math.}, 180(1):1--53, 2010.
		
		\bibitem{Flandoli-Luo-2019}
		F.~Flandoli and D.~Luo.
		\newblock Euler-Lagrangian approach to 3-D stochastic Euler equations.
		\newblock {\em J. Geom. Mech. }, 11 (2):153--165, 2019.
		
		
		%		\bibitem{Friedlander-Rusin-Vicol-2014}
		%		S. Friedlander, W. Rusin and V. Vicol
		%		\newblock The magneto-geostrophic equations: a survey.
		%		\newblock {\em Proceedings of the St. Petersburg Mathematical Society, Volume XV: Advances in Mathematical Analysis of Partial Differential Equations.}, D. Apushkinskaya, and A.I. Nazarov, eds. pp. 53-78, 2014.
		
		%		\bibitem{Friedlander-Vicol-2011}
		%		S. Friedlander, and V. Vicol
		%		\newblock  Global well-posedness for an advection-diffusion equation arising in magnetogeostrophic dynamics,
		%		\newblock {\em Ann. Inst. H. Poincar\''e Anal. Non Lin\'eaire.}, 28 (2),283–-301, 2011.
		
		%		\bibitem{Friedlander-et-al-2012}
		%		S. Friedlander, F. Gancedo, W. Sun and V. Vicol
		%		\newblock On a singular incompressible porous media equation,
		%		\newblock {\em J.
		%			Math. Phys.}, 53 (11), 20 pp, 2012.
		
		
		\bibitem{Fuchssteiner-Fokas-1981-PhyD}
		B.~Fuchssteiner and A.~S. Fokas.
		\newblock Symplectic structures, their {B}\"{a}cklund transformations and
		hereditary symmetries.
		\newblock {\em Phys. D}, 4(1):47--66, 1981/82.
		
		
		
		
		\bibitem{Gawarecki-Mandrekar-2010-Springer}
		L.~Gawarecki and V.~Mandrekar.
		\newblock {\em Stochastic differential equations in infinite dimensions with
			applications to stochastic partial differential equations}.
		\newblock Probability and its Applications (New York). Springer, Heidelberg,
		2011.
		
		\bibitem{GlattHoltz-Ziane-2009-ADE}
		N.~Glatt-Holtz and M.~Ziane.
		\newblock Strong pathwise solutions of the stochastic {N}avier-{S}tokes system.
		\newblock {\em Adv. Differential Equations}, 14(5-6):567--600, 2009.
		
		\bibitem{GlattHoltz-Vicol-2014-AP}
		N.~E. Glatt-Holtz and V.~C. Vicol.
		\newblock Local and global existence of smooth solutions for the stochastic
		{E}uler equations with multiplicative noise.
		\newblock {\em Ann. Probab.}, 42(1):80--145, 2014.
		
		
		%	\bibitem{GessMaurelli18}
		%	B.~Gess and M.~Maurelli.
		%	\newblock  Well-posedness by noise for scalar conservation laws.
		%	\newblock {\em Communications in Partial Differential Equations}, 43:12, 1702-1736, 2018.
		
		%		\bibitem{Gottwald-Crommelin-Franzke-16}
		%		G.~Gottwald, D.~Crommelin, and C.~LE. Franzke.
		%		\newblock Stochastic climate theory.
		%		\newblock {\em arXiv:1612.07474}, 2016.
		
		\bibitem{Gyongy-Krylov-1996-PTRF}
		I.~Gy\"{o}ngy and N.~Krylov.
		\newblock Existence of strong solutions for {I}t\^{o}'s stochastic equations
		via approximations.
		\newblock {\em Probab. Theory Related Fields}, 105(2):143--158, 1996.
		
		\bibitem{Holm-2015-ProcA}
		D.~D. Holm.
		\newblock Variational principles for stochastic fluid dynamics.
		\newblock {\em Proc. A.}, 471(2176):20140963, 19, 2015.
		
		\bibitem{Holm-Erwin-2019-Arxiv}
		D.~D. Holm and E.~Luesink.
		\newblock Stochastic wave-current interaction in thermal shallow water
		dynamics.
		\newblock {\em arXiv:1910.10627}, 2019.
		
		\bibitem{Holden-Raynaud-2007}
		H.~Holden and X.~Raynaud.
		\newblock Raynaud. Global conservative solutions of the Camassa-Holm equation—a Lagrangian point of view.
		\newblock {\em Comm. Partial Differential Equations}, 32(10-12):1511–1549, 2007.
		
		\bibitem{Hormander-1985-book-3}
		L.~H\"{o}rmander.
		\newblock {\em The analysis of linear partial differential operators. {III}},
		volume 274 of {\em Grundlehren der Mathematischen Wissenschaften [Fundamental
			Principles of Mathematical Sciences]}.
		\newblock Springer-Verlag, Berlin, 1985.
		\newblock Pseudodifferential operators.
		
		\bibitem{Kallianpur-Xiong-1995-book}
		G.~Kallianpur and J.~Xiong.
		\newblock Stochastic differential equations in infinite-dimensional spaces.
		\newblock 26:vi+342, 1995.
		\newblock Expanded version of the lectures delivered as part of the 1993
		Barrett Lectures at the University of Tennessee, Knoxville, TN, March 25--27,
		1993, With a foreword by Balram S. Rajput and Jan Rosinski.
		
		\bibitem{Karczewska-1998-AUMCSS}
		A.~Karczewska.
		\newblock Stochastic integral with respect to cylindrical {W}iener process.
		\newblock {\em Ann. Univ. Mariae Curie-Sk\l odowska Sect. A}, 52(2):79--93,
		1998.
		
		
		
		\bibitem{Kato-Ponce-1988-CPAM}
		T.~Kato and G.~Ponce.
		\newblock Commutator estimates and the {E}uler and {N}avier-{S}tokes equations.
		\newblock {\em Comm. Pure Appl. Math.}, 41(7):891--907, 1988.
		
		\bibitem{Kenig-Ponce-Vega-1991-JAMS}
		C.~E. Kenig, G.~Ponce, and L.~Vega.
		\newblock Well-posedness of the initial value problem for the {K}orteweg-de
		{V}ries equation.
		\newblock {\em J. Amer. Math. Soc.}, 4(2):323--347, 1991.
		
		\bibitem{Krylov-Rozovskiui-1979-chapter}
		N.~V. Krylov and B.~L. Rozovski\u{\i}.
		\newblock Stochastic evolution equations.
		\newblock In {\em Current problems in mathematics, {V}ol. 14 ({R}ussian)},
		pages 71--147, 256. Akad. Nauk SSSR, Vsesoyuz. Inst. Nauchn. i Tekhn.
		Informatsii, Moscow, 1979.
		
		
		
		\bibitem{Leslie-Quarini-79}
		D.~Leslie and G.~Quarini.
		\newblock The application of turbulence theory to the formulation of subgrid modelling procedures
		\newblock {\em Journal of Fluid Mechanics}, 91: 65--91, 1979.
		
		
		
		\bibitem{Leha-Ritter-1985-MA}
		G.~Leha and G.~Ritter.
		\newblock On solutions to stochastic differential equations with discontinuous
		drift in {H}ilbert space.
		\newblock {\em Math. Ann.}, 270(1):109--123, 1985.
		
		%	\bibitem{Li-Rodrigo-2008}
		%	D.~Li and J.~Rodrigo.
		%	\newblock Blow-up of solutions for a 1D transport equation with nonlocal velocity
		%	and supercritical dissipation
		%	\newblock {\em Adv. Math.}, 217 (6),  2563–-2568, 2008.
		
		%		\bibitem{Moffat-Loper-1994}
		%		H.K. Moffatt and D.E. Loper
		%		\newblock The magnetostrophic rise of a buoyant parcel in the earth’s core.
		%		\newblock {\em Geophysical Journal International}, 117 (2), 394–402, 1994.
		
		\bibitem{Pardoux-1972-CRAS}
		E.~Pardoux.
		\newblock Sur des \'{e}quations aux d\'{e}riv\'{e}es partielles stochastiques
		monotones.
		\newblock {\em C. R. Acad. Sci. Paris S\'{e}r. A-B}, 275:A101--A103, 1972.
		
		\bibitem{Prevot-Rockner-2007-book}
		C.~Pr\'{e}v\^{o}t and M.~R\"{o}ckner.
		\newblock {\em A concise course on stochastic partial differential equations},
		volume 1905 of {\em Lecture Notes in Mathematics}.
		\newblock Springer, Berlin, 2007.
		
		\bibitem{Ren-Tang-Wang-2020-Arxiv}
		P.~Ren, H.~Tang, and F.-Y. Wang.
		\newblock Distribution-path dependent nonlinear {SPDE}s with application to
		stochastic transport type equations.
		\newblock {\em arXiv:2007.09188}, 2020.
		
		\bibitem{Rohde-Tang-2020-Arxiv}
		C.~Rohde and H.~Tang.
		\newblock On a stochastic {C}amassa-{H}olm type equation with higher order
		nonlinearities.
		\newblock {\em J Dyn Diff Equat}, 2020, https://doi.org/10.1007/s10884-020-09872-1.
		
		
		
		\bibitem{Tang-2018-SIMA}
		H.~Tang.
		\newblock On the pathwise solutions to the {C}amassa-{H}olm equation with
		multiplicative noise.
		\newblock {\em SIAM J. Math. Anal.}, 50(1):1322--1366, 2018.
		
		\bibitem{Tang-2020-Arxiv}
		H.~Tang.
		\newblock Noise effects on dependence on initial data and blow-up for
		stochastic {E}uler--{P}oincar\'{e} equations.
		\newblock {\em arXiv:2002.08719}, 2020.
		
		\bibitem{Taylor-1991-book}
		M.~E. Taylor.
		\newblock {\em Pseudodifferential operators and nonlinear {PDE}}, volume 100 of
		{\em Progress in Mathematics}.
		\newblock Birkh\"{a}user Boston, Inc., Boston, MA, 1991.
		
		%		\bibitem{Ver81}
		%		A.~J. Veretennikov. 
		%		\newblock On strong solutions and explicit formulas for solutions of stochastic integral equations. 
		%		\newblock{Sbornik: Mathematics,} 39(3):387–403, 1981.
		
		\bibitem{Temam-1977-book}
		R.~Temam.
		\newblock {\em Navier-{S}tokes equations. {T}heory and numerical analysis}.
		\newblock North-Holland Publishing Co., Amsterdam-New York-Oxford, 1977.
		\newblock Studies in Mathematics and its Applications, Vol. 2.
		
		\bibitem{Zidikheri-F-10}
		M.~Zidikheri and J.~Frederiksen.
		\newblock Stochastic subgrid-scale modelling for non-equilibrium geophysical flows
		\newblock {\em Philosophical Transactions of the Royal Society A: Mathematica, Physical and Engineering Sciences }, 368: 145--160, 2010.
		
	\end{thebibliography}
\end{document}